\documentclass{elsarticle}
\usepackage{lineno,hyperref}
\modulolinenumbers[5]

\journal{Computers \& Mathematics with Applications}
\usepackage{geometry}
\usepackage{amsmath}
\usepackage{mathtools}
\usepackage{graphicx}
\usepackage{array}
\usepackage{color}
\usepackage{amsfonts,latexsym}
\usepackage{soul}
\usepackage{multirow}
\usepackage{bbm}
\usepackage{tabularx}
\usepackage{hyperref}
\usepackage{amssymb}
\usepackage{amsthm}
\usepackage[export]{adjustbox}
\usepackage{epstopdf}
\usepackage{algorithm}
\usepackage{algorithmic,color} 
\usepackage{multirow}
\usepackage{appendix}
\usepackage{cancel}
\usepackage{dsfont}
\usepackage{mathrsfs}
\usepackage{caption}
\usepackage{subcaption}

\newcommand{\R}{{\mathbb R}}
\newcommand{\dd}{{\textbf{d}}}
\newcommand{\nbold}{{\textbf{n}}}
\newcommand{\dv}{{\textnormal{div}}}
\newcommand{\WW}{{\textbf{W}}}
\newcommand{\sbold}{{\textbf{s}}}

\newtheorem{teo}{Theorem}[section]
\newtheorem{cor}[teo]{Corollary}
\newtheorem{prop}[teo]{Proposition}

\newtheorem{lemma}[teo]{Lemma}

\date{\today}

\begin{document}

\begin{frontmatter}

\title{Fast and stable schemes for non-linear osmosis filtering} 

%% Group authors per affiliation:
\author{L.Calatroni\fnref{myfootnote1}}
%\address{Radarweg 29, Amsterdam}
\fntext[myfootnote1]{Universite C\^ote d’Azur, CNRS, Inria, Laboratoire I3S, Sophia-Antipolis, France}

\author{S.Morigi\fnref{myfootnote2}}
%\address{Radarweg 29, Amsterdam}
\fntext[myfootnote2]{Department of Mathematics, University of Bologna, Bologna, Italy.}

\author{S.Parisotto\fnref{myfootnote3}}
%\address{Radarweg 29, Amsterdam}
\fntext[myfootnote3]{Department of Applied Mathematics and Theoretical Physics, University of Cambridge, Wilberforce Road, Cambridge, CB3 0WA, UK}

\author{G.A.Recupero\fnref{myfootnote2}}
%\address{Radarweg 29, Amsterdam}
%\fntext[myfootnote4]{Department of Mathematics, University of Bologna, Bologna, Italy}

%% or include affiliations in footnotes:
%\author[mymainaddress,mysecondaryaddress]{Elsevier Inc}
%\ead[url]{www.elsevier.com}

%\author[mysecondaryaddress]{Global Customer Service\corref{mycorrespondingauthor}}
%\cortext[mycorrespondingauthor]{Corresponding author}
%\ead{support@elsevier.com}

%\address[mymainaddress]{1600 John F Kennedy Boulevard, Philadelphia}
%\address[mysecondaryaddress]{360 Park Avenue South, New York}

\begin{abstract}
We consider a non-linear variant of the transport-diffusion osmosis model for solving a variety of imaging problems such as shadow/soft-light removal and compact data representation. 
The non-linear behaviour is encoded in terms of a  general scalar function $g$ with suitable properties, which allows to balance the diffusion intensity on the different regions of the image while preventing smoothing artefacts. For the proposed model, conservation properties (intensity and non-negativity) are proved and a variational interpretation is showed for specific choices of $g$. Upon suitable spatial discretisation,  both an explicit and a semi-implicit iterative scheme are considered, for which convergence restrictions and unconditional stability are proved, respectively. To validate the proposed modelling and the computational speed of the numerical schemes considered, we report several results and comparisons for the problem of shadow/light-spot  removal and compact data representation, showing that artefact-free and computationally efficient results are obtained in comparison to standard linear and anisotropic models, and state-of-the art approaches.
\end{abstract}


\end{frontmatter}

\section{Introduction}

The use of Partial Differential Equations (PDEs) in the context of image processing and computer vision is nowadays pretty well-established. Linear and non-linear diffusion PDEs have been indeed considered over the last decades for several several image processing problems, such as denoising, deconvolution, inpainting and segmentation, see, e.g., \cite{weickert98,aubert2006mathematical,ChanShen,schoenlieb}.
For these tasks, standard approaches resorting to the use of parabolic (anisotropic) non-linear diffusion PDEs of second or higher-order have been used, and their interpretation as gradient flows of suitable non-smooth regularisation functionals such as the edge-preserving Total Variation (TV) semi-norm have been drawn.
In this work, we concentrate on a non-linear PDE where a non-linear diffusion term is combined with a further transport term, which finds its motivations into the modelling of the physical process of \emph{osmosis}.

In biology, such term indicates a non-symmetric  transport process where molecules pass through a semipermeable membrane in such a way that at its steady state, the liquid concentrations on both sides of the
membrane differ. Osmosis is, for instance, the primary mechanism for transporting water inside and outside living cells, and it stands at the very basis of many medicine and engineering processes.
Due to its non-trivial (i.e. non-flat) steady states, such process can be seen indeed as the non-symmetric counterpart of standard diffusion processes in the sense that, during evolution, the probability of moving from inside to outside the cell through the membrane is not equal to the probability of performing the reverse process \cite{Hagenburg12o}.

The corresponding interpretation of the osmosis process in the context of image analysis involves the study of how intensity values at each pixel are propagated within the image domain. Previous works have considered osmosis filtering in the context of imaging, either in linear isotropic \cite{weickert-ssvm} or anisotropic \cite{Parisotto2019AnisotropicOF} form, proposing also some efficient numerical schemes based on standard finite-difference discretisation \cite{vogel-ssvm} and/or operator splitting \cite{Calatroni_ADI}.
For all these models, the combination of diffusion with transport is made thanks to the use of a drift term defined in terms of a vector field which, intuitively, models the non-symmetric nature of the process.
Interestingly, such PDEs are all in divergence form, thus allowing for provable conservation of average intensity of the initial image as well as non-negativity. Furthermore, upon a particular choice of the drift term, an explicit expression of the  non-flat steady state can be found.
When applied to imaging tasks aiming at balancing intensity information (i.e. for shadow/light-spot removal problems and/or for image fusion), the advantage of using an  osmosis approach is that it can integrate mismatching image data in a seamless way (see also \cite{Parisotto2020} for a variational osmosis-based model for image fusion  showing the advantage with-respect-to state of the art methods). Recently, a more sophisticated approach has been proposed in \cite{Benalia} 
where a non-local version of the osmosis operator along with an anisotropic regularisation term have been integrated in a variational formulation.
This model achieves good performance in shadow removal, but at the price of a high computational cost.

Due to its differential structure, osmosis filtering shows similarities to gradient domain methods introduced in computer graphics for a wide number of surface editing and reconstruction applications, such as the famous Poisson editing, see e.g., \cite{Poisson,GraDomain}.
Compared to these approaches, image osmosis  offers the additional advantage of being invariant under multiplicative illumination changes and provide very good results in shadow removal and image cloning. However, due to the underlying linear smoothing enforced by the diffusion terms, standard osmosis approaches typically suffer from some reconstruction artefacts (typically, oversmoothing) which prevent their practical use.
For reference, we further refer the reader to \cite{Finlayson1,Finlayson2,ShadowReview2013,Guo,Le} where alternative approaches to shadow removal are used, based, for instance, of suitable mapping into appropriate colour spaces and/or learning approaches.

\paragraph{Contribution} In this paper, we present a non-linear version of the image osmosis model, which avoids standard smoothing artefacts of linear models thanks to the use of a non-linear diffusivity function promoting isotropic edge-stopping diffusion. 
Analytically, we prove that the proposed evolution model maintains the same conservation properties of the original linear model, and that it can be interpreted as the gradient flow of a suitable non-smooth energy for specific choices of the drift term.
We then show that analogous properties hold also at a discrete level, whenever appropriate finite-difference discretisation stencils and explicit and semi-implicit schemes are used. For those schemes, we prove conditional and unconditional stability, respectively. Such property is particularly interesting as it allows for the fast computation of the desired steady-state solution.
Thanks to such stability property, the proposed model is fully automatic as no hyperparameters need to be tuned in it.
The proposed model and algorithms are validated on exemplar image processing tasks, such as shadow removal, light removal and compact data representation, and several comparisons with other state-of-the-art  as well as alternative osmosis filtering models are given. 

\paragraph{Structure of the paper} In Section \ref{sub:linear}, a short review of existing osmosis filtering models and of their use in applications is given.
In Section \ref{sec:model}, the proposed continuous non-linear osmosis model is proposed and its conservation and variational properties are proved. In Section \ref{sec:discr}, spatial and temporal discretisation schemes are studied, in particular from the point of view of consistency with the aforementioned properties in the continuous setting and in terms of stability and convergence.
In Section \ref{sec:results}, several results for three different imaging applications (shadow removal, light removal and compact data representation) are showed, confirming that the use of the proposed non-linear model improves upon linear and state-of-the art methods.
Some conclusions are reported in Section \ref{sec:concl}.

\paragraph{Notation} We will denote the set of non-negative and positive real numbers by $\R_{+}$ and $\R_+^*$, respectively. In the following Sections 
\ref{sub:linear} and \ref{sec:model} we will use the \textbf{bold} notation to denote vectorial fields in $\R^N$. To avoid unnecessary heavy notation in the discretised setting introduced starting from Section \ref{sec:discr}, we will use standard (unbold) notation to denote both vectors  in $\R^N$ and matrices in $\R^{N\times N}$.

\section{Linear osmosis filtering}\label{sub:linear}

The original image osmosis model firstly introduced in \cite{weickert-ssvm} in the context of imaging is a parabolic, linear and isotropic transport-diffusion PDE defined on a rectangular domain $\Omega\subset \R^2$ for an initial positive image $f:\Omega\rightarrow\R_{+}^*$ (the extension to the vectorial case is immediate). 
Its transport component is defined in terms of a drift vector field $\dd:\Omega\rightarrow \R^2$ whose expression will be made more precise in the following. Assuming homogeneous Neumann boundary conditions, the model reads:
\begin{align}\label{eq:linear}
    \begin{cases}
    \partial_t u = \dv \left( \, \nabla u - \dd u\right) \quad & \text{on } \Omega \times (0,T], \\
    \left< \nabla u - \dd u , \nbold \right> = 0 & \text{on } \partial \Omega \times (0,T];
    \\
    u(x,0) = f(x) & \text{on } \Omega,   
    \end{cases}
\end{align}
where $T\in (0,+\infty)$ is a final time, and  $\nbold$ is the outer normal vector
to the image boundary $\partial\Omega$. In \cite{weickert-ssvm}, it  has been proved that the evolution process \eqref{eq:linear} preserves the average mass (grey value) of the image as well as its non-negativity. 
Moreover, if the drift $\dd$ is defined in a \emph{canonical} form, i.e. in terms of a given reference image $v:\Omega\rightarrow\R_+^*$ as:
\begin{equation}\label{eq:d_canonical}
   \dd := \nabla \log v, 
\end{equation}
then it can be showed that the steady-state solution $w: \Omega \to \R_{+}$ is nothing but a multiplicative rescaling of $v$, that is $w=cv$, with $c>0$ \cite[Proposition Ic)]{weickert-ssvm}, see Figure \ref{fig:plain_osmosis} for a visual illustration.
It can be shown that analogous properties hold upon suitable finite-difference discretisation \cite{vogel-ssvm}.

\begin{figure}[h]
    \centering
    \begin{subfigure}[b]{0.21\textwidth}
        \centering
        \includegraphics[width=\textwidth]{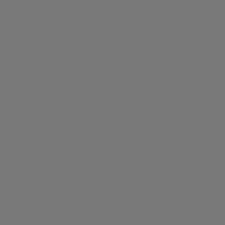} 
        \caption{$f$}
        \label{fig:plain_1}
    \end{subfigure}
    \begin{subfigure}[b]{0.21\textwidth}
        \centering
        \includegraphics[width=\textwidth]{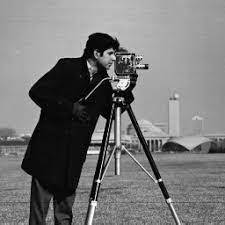}
        \caption{$v$}
        \label{fig:plain_2}
    \end{subfigure}
    \begin{subfigure}[b]{0.21\textwidth}
        \centering
        \includegraphics[height=1\textwidth,width=1\textwidth]{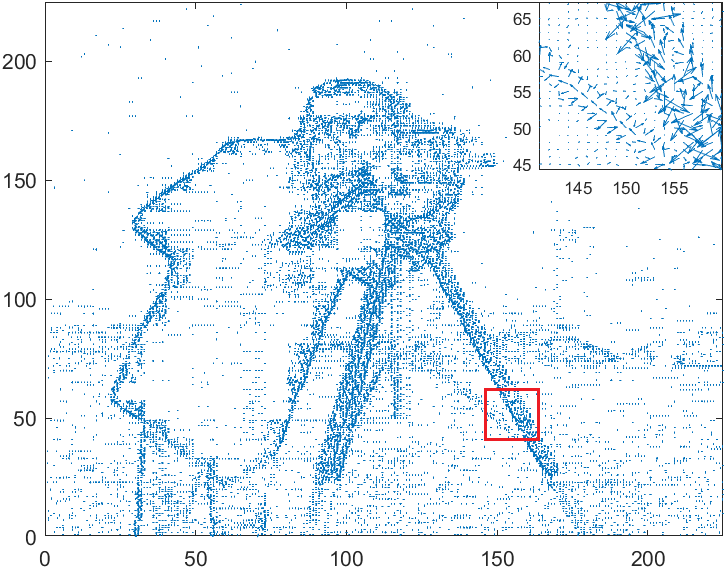}
        \caption{$\dd$}
        \label{fig:plain_3}
    \end{subfigure}
    \begin{subfigure}[b]{0.21\textwidth}
        \centering
        \includegraphics[width=\textwidth]{photoop.jpeg}
        \caption{$w$}
        \label{fig:plain_4}
    \end{subfigure}
    \caption{Evolution of the osmosis model \eqref{eq:linear}. From left to right: (constant) initial image $f$, reference image $v$, vector representation of drift $\dd$ (with zoom) and steady state $w$.}
    \label{fig:plain_osmosis}
\end{figure}

An anisotropic linear version of  \eqref{eq:linear}, with analogous conservation and convergence properties, has been developed in \cite{Parisotto2019AnisotropicOF} in order to favour directional linear diffusion. Similarly as in \cite{weickert98}, this has been done by encoding local directional information of the image in a positive semi-definite symmetric tensor $\WW$. The corresponding anisotropic model takes the form:
\begin{align}\label{eq:aniso}
    \begin{cases}
    \partial_t u = \dv \left( \WW\, (\nabla u - \dd u) \right) \quad & \text{on } \Omega \times (0,T], 
    \\
    \left< \WW (\nabla u - \dd u) , n \right> = 0 & \text{on } \partial \Omega \times (0,T];
    \\
    u(x,0) = f(x) & \text{on } \Omega.
    \end{cases}
\end{align}

For both models \eqref{eq:linear} and \eqref{eq:aniso}, there is of course no clear interest in observing convergence to a rescaled version $w$ of an image which is need in advance to define the model.
However, upon slight modifications of the drift term $\dd$,  osmosis models can be applied to many different problems such as shadow removal, compact data representation and image cloning, as shown in \cite{weickert-ssvm}. In the following, we review how these tasks can be tackled by osmosis.

\subsection{Applications of the osmosis filtering}\label{sec:appl}

Osmosis models in the form \eqref{eq:linear} and \eqref{eq:aniso} can be easily modified for several imaging applications.
To do so, a suitable decomposition of the image domain $\Omega$ in terms of a disjoint union of subsets is required. Let thus $\Omega = \Omega_{in} \sqcup \Omega_b \sqcup \Omega_{out}$ be a given partition of $\Omega$, where $\Omega_b$ denotes a possibly expanded boundary region separating $\Omega_{in}$ from $\Omega_{out}$. 
Given the reference image $v$, a simple modification of the drift term $\dd$ consists in forcing it to vanish in correspondence of all points of $\Omega_b$ so that an artificial set of discontinuities is created. In mathematical terms, this means choosing for all $x\in\Omega$:
\begin{equation}\label{eq:d_shrem}
    \dd(x) := \left(\nabla (\log v) ~\chi_{\Omega_b^C} \right) (x) =
    \begin{cases}
    \frac{\nabla v(x)}{v(x)} \quad &\text{if }x\in \Omega_{in} \cup \Omega_{out}, \\
    0 & \text{if }x\in \Omega_b,
    \end{cases},
\end{equation}
where for $S\subset \Omega$, $\chi_S:\Omega\to\left\{0,1\right\}$ denotes the characteristic function of the subset $S$.
Practically, the role of the characteristic function is played upon discretisation by a mask matrix $M\in \left\{0,1\right\}^{|\Omega|}$ whose null pixels identify the region $\Omega_{b}$.
As $\dd$ is set to zero on $\Omega_b$, the transport term vanishes therein, hence equations \eqref{eq:linear} and \eqref{eq:aniso} induce linear isotropic/anisotropic diffusion, respectively.
Note that the preliminary detection of $\Omega_b$ can be obtained, for instance, by standard image segmentation techniques, such as, e.g., edge detection \cite{ChanShen}.

\paragraph{Shadow removal}
For \textit{shadow removal} problems the given image $f$ is characterised by the presence of a constant (i.e. with no space-variant intensity) shadowed region $S\subset\Omega$, see Fig. \ref{fig:shadowlinear_1}. The  mask $M$ is thus associated here to a preliminary detection of the shadow boundary $\Omega_b$ (for instance by using \cite{ShadowReview2013,BabaSiggraph2003}),
see Fig.\ref{fig:shadowlinear_2}.
By choosing as a reference image $v=f$ in \eqref{eq:d_shrem}, we can thus evolve model \eqref{eq:linear} as done, e.g., in \cite{weickert-ssvm,vogel-ssvm,Calatroni_ADI}, to obtain the result in Fig.\ref{fig:shadowlinear_3}.

Due to its intrinsic transport-diffusion properties, we observe that the linear osmosis filtering does correctly perform shadow removal by balancing the intensity between the shadowed/un-shadowed regions of the image. 
However, due to the choice \eqref{eq:d_shrem}, the drift term $\dd$ vanishes on the $\Omega_b$, hence,   pure Laplace diffusion is enforced therein.
By considering a non-linear version of \eqref{fig:plain_osmosis} as described in the following sections, we will improve upon such drawback and obtain the more accurate result reported in Fig.\ref{fig:shadowlinear_4}.
Furthermore, in order to address the more realistic situations of shadows with blurred (i.e. not sharp) boundaries as it is observed in natural images, we will further show how the proposed model can be applied also in case of soft-edged shadows, see Section \ref{sec:shrem}.

\begin{figure}[h!]
    \centering
    \begin{subfigure}[b]{0.2\textwidth}
        \centering
        \includegraphics[width=\textwidth]{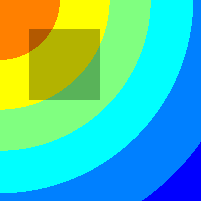}
        \caption{$f=v$}
        \label{fig:shadowlinear_1}
    \end{subfigure}
    \begin{subfigure}[b]{0.2\textwidth}
        \centering
        \fbox{\includegraphics[width=0.91\textwidth]{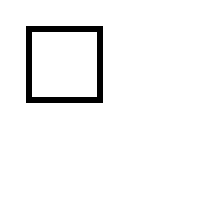}}
        \caption{mask $M$}
        \label{fig:shadowlinear_2}
    \end{subfigure}
    \begin{subfigure}[b]{0.2\textwidth}
        \centering
        \includegraphics[width=\textwidth]{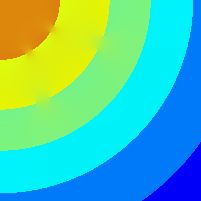}
        \caption{$w$ (linear)}
        \label{fig:shadowlinear_3}
    \end{subfigure}
    \begin{subfigure}[b]{0.2\textwidth}
        \centering
        \includegraphics[width=\textwidth]{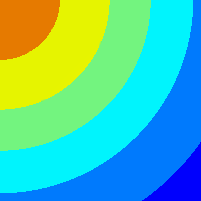}
        \caption{$w$ (non-linear)}
        \label{fig:shadowlinear_4}
    \end{subfigure}
    \caption{Shadow removal using linear \eqref{eq:linear} and non-linear (proposed) osmosis filtering.}
    \label{fig:shadowlinear}
\end{figure}

\paragraph{Spot-light removal} A problem analogous to shadow removal is the \textit{spot-light removal} problem where, similarly, the task is to remove a light spot from a given image $f$ (Fig. \ref{fig:lightlinear_1}). Given the segmentation of the constant spot light boundary (Fig. \ref{fig:lightlinear_2}), the use of osmosis removes the spot by balancing the intensity between the two regions in the image.
The results obtained by applying standard linear and the proposed non-linear osmosis filtering are shown in Fig. \ref{fig:lightlinear_3}-\ref{fig:lightlinear_4}.

\begin{figure}[h!]
    \centering
    \begin{subfigure}[b]{0.2\textwidth}
        \centering
        \includegraphics[width=\textwidth]{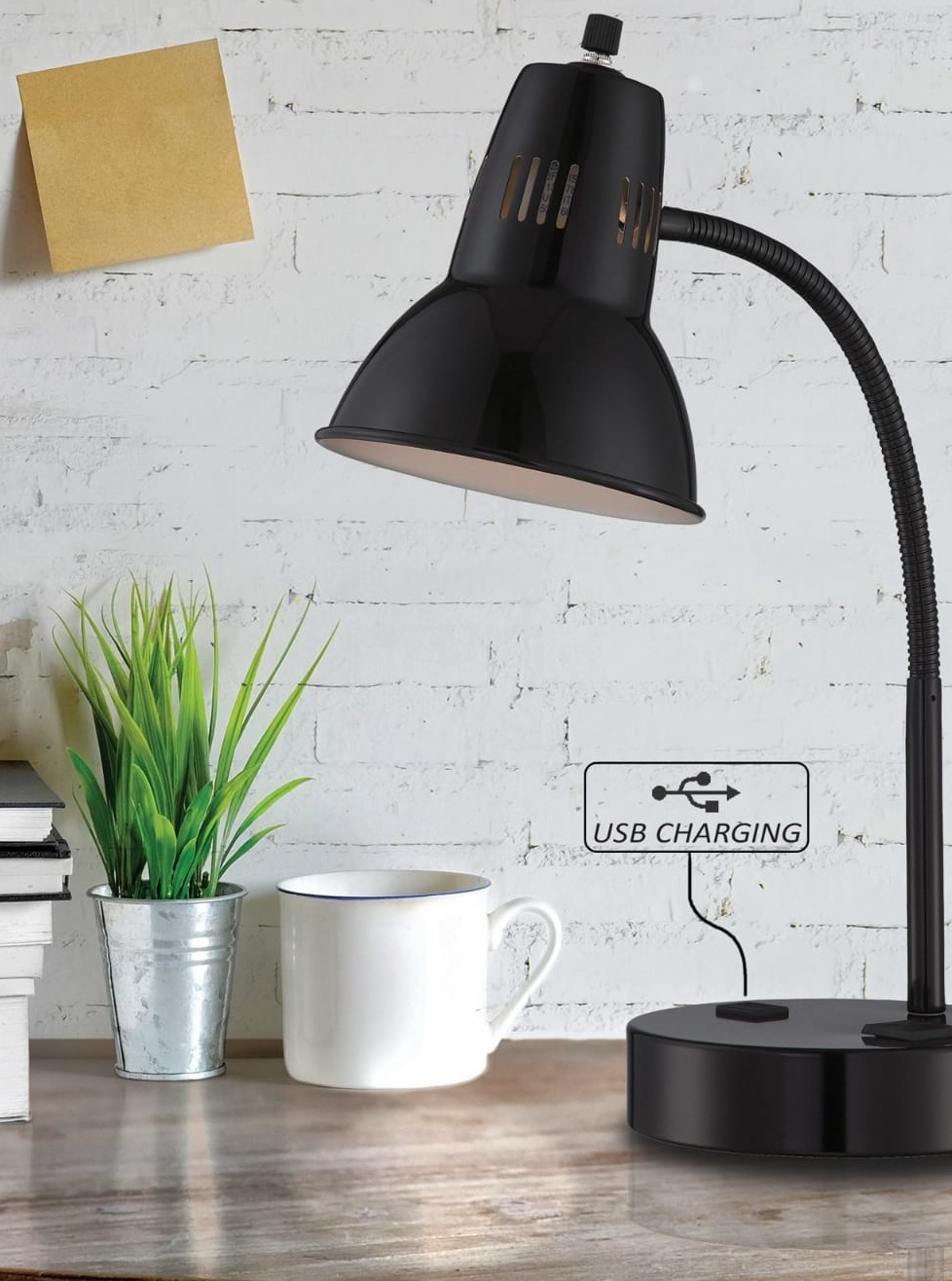}
        \caption{$f=v$}
        \label{fig:lightlinear_1}
    \end{subfigure}
    \begin{subfigure}[b]{0.2\textwidth}
        \centering
        \fbox{\includegraphics[width=0.94\textwidth]{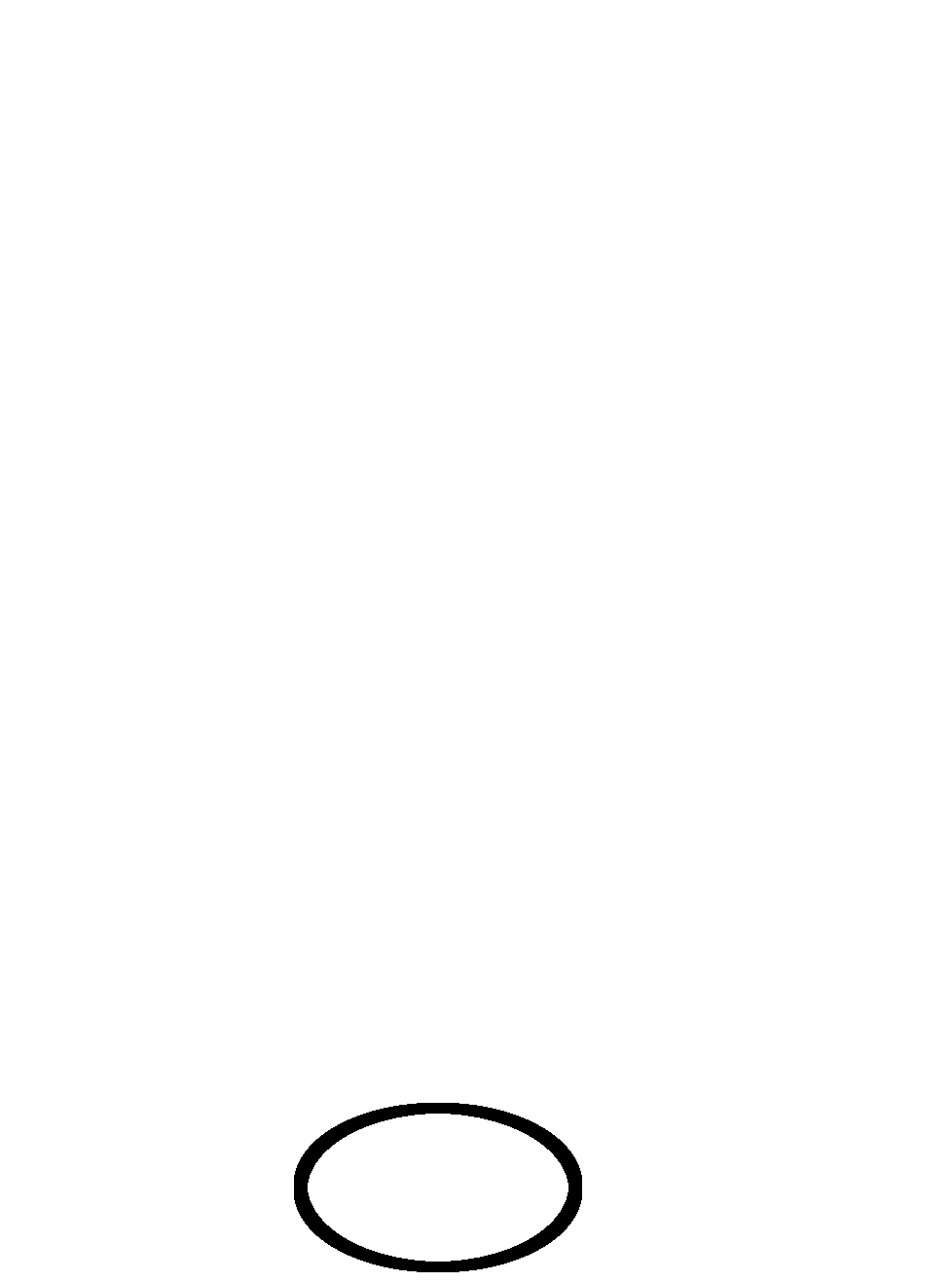}}
        \caption{mask $M$}
        \label{fig:lightlinear_2}
    \end{subfigure}
    \begin{subfigure}[b]{0.2\textwidth}
        \centering
        \includegraphics[width=\textwidth]{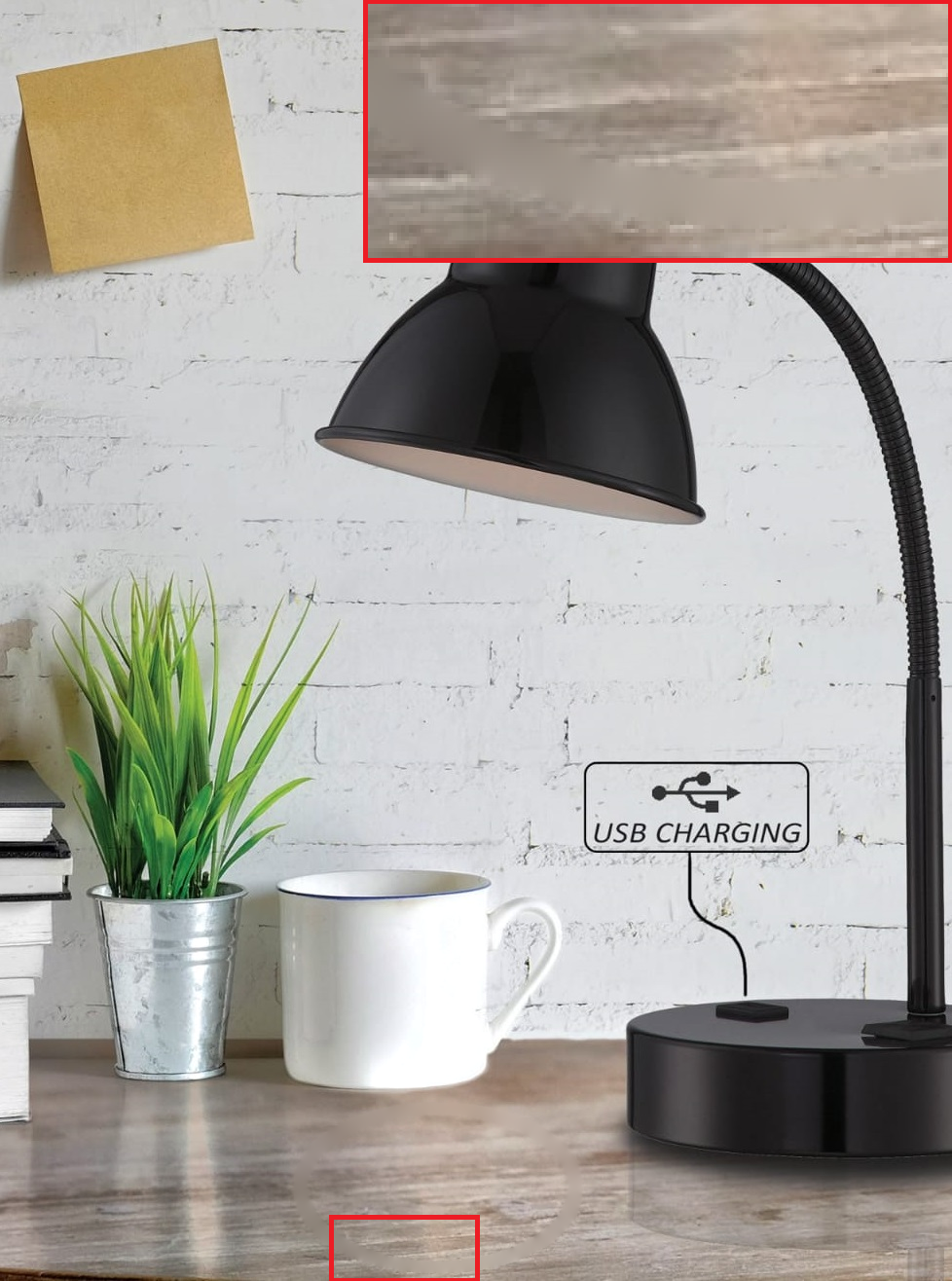}
        \caption{$w$ (linear)}
        \label{fig:lightlinear_3}
    \end{subfigure}
    \begin{subfigure}[b]{0.2\textwidth}
        \centering
        \includegraphics[width=\textwidth]{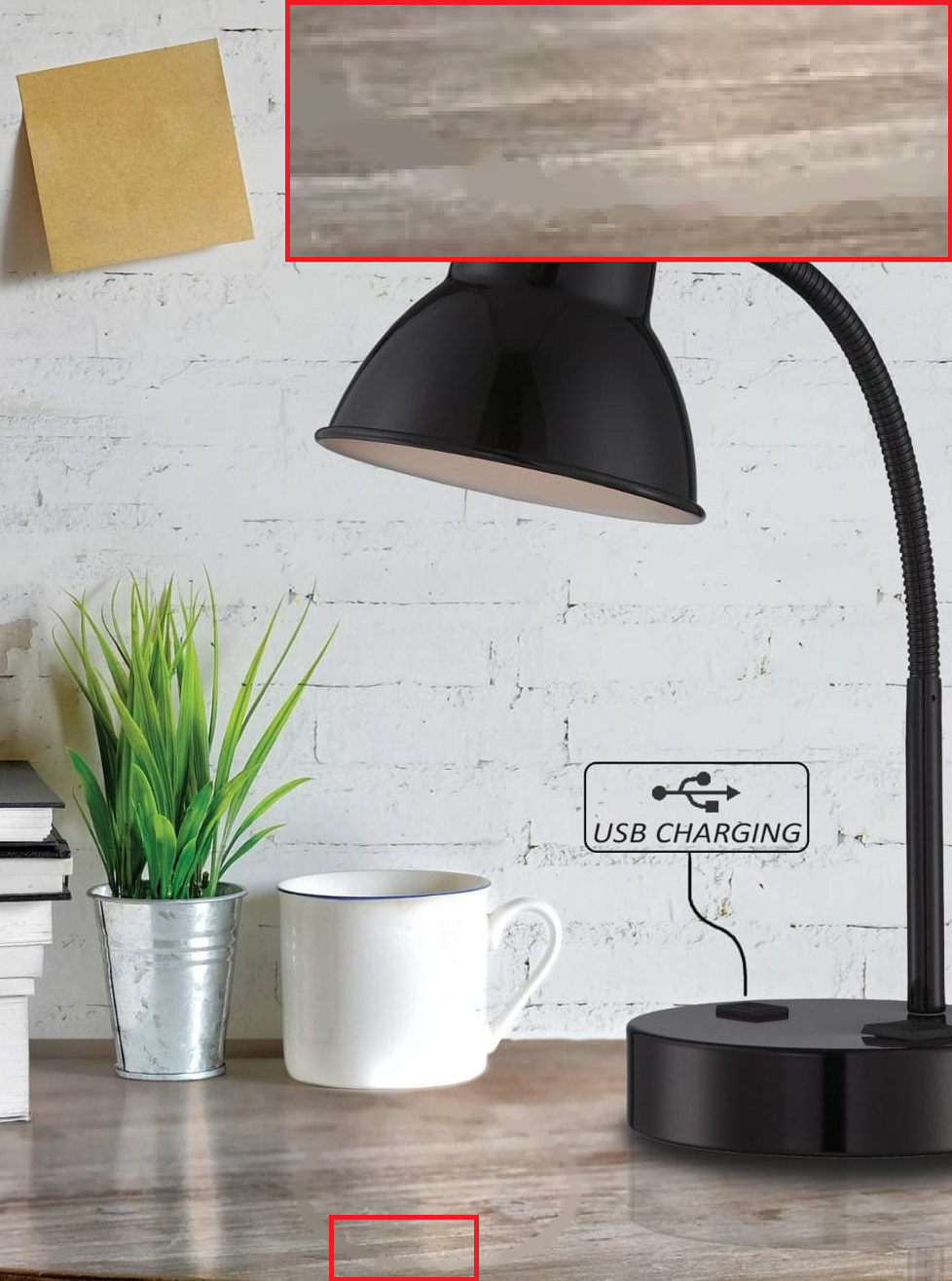}
        \caption{$w$ (non-linear)}
        \label{fig:lightlinear_4}
    \end{subfigure}
    \caption{Spot-light removal using linear and non-linear (proposed) osmosis filtering \eqref{eq:linear}, with zoom.}
    \label{fig:lightlinear}
\end{figure}

\paragraph{Compact image representation} The task of \textit{compact image} (or, generally, \emph{data}) \emph{representation} consists in representing a given image with as little information as possible using some limited, but significant, image content (see, e.g. \cite{Carlsson1988SketchBC}, \cite{Mainberger2011EdgebasedCO}). 
Such problem has been previously considered in \cite{weickert-ssvm} for the linear osmosis model: there, a greyscale image is reconstructed by osmosis evolution by starting from the  intensity values on pre-detected edges which is encoded in the definition of the drift-term $\dd$ so that:
\begin{align}\label{eq:cdr_drift}
    \dd := \left(\frac{\nabla v}{v} \chi_{\Omega_b}  \right)(x)= \begin{cases}
    \frac{\nabla v(x)}{v(x)} \quad & \text{if }x\in \Omega_b,\\
    0  &\text{if }x\in \Omega \setminus \Omega_b,
    \end{cases}
\end{align}
where $\Omega_b$ denotes here the set of the edges.
The task is thus a sort of interpolation process where starting from $f$ (Fig. \ref{fig:cdr_linear_1}) and using only the data assigned on the mask $M$ (Fig. \ref{fig:cdr_linear_2}) to define the drift term as in \eqref{eq:cdr_drift}, the osmosis evolution produces the result in Fig. \ref{fig:cdr_linear_3}. Fig. \ref{fig:cdr_linear_4} is the ground-truth image, reported here for comparison.

\begin{figure}[h!]
    \centering
    \begin{subfigure}[b]{0.2\textwidth}
        \centering
        \includegraphics[width=\textwidth]{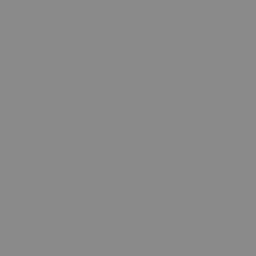}
        \caption{$f=$ const}
        \label{fig:cdr_linear_1}
    \end{subfigure}
    \begin{subfigure}[b]{0.2\textwidth}
        \centering
        \fbox{\includegraphics[width=0.92\textwidth]{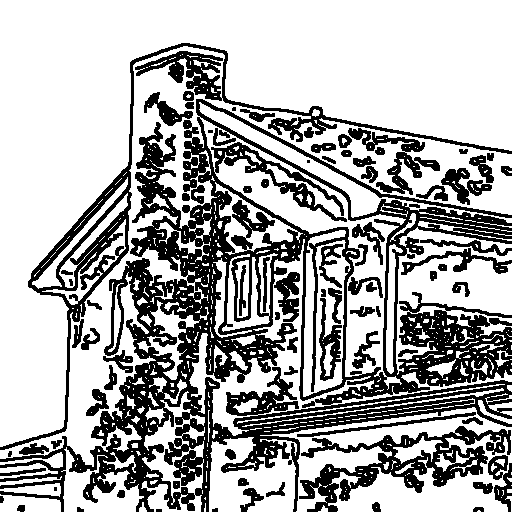}}
        \caption{mask $M$}
        \label{fig:cdr_linear_2}
    \end{subfigure}
    \begin{subfigure}[b]{0.2\textwidth}
        \centering
        \includegraphics[width=\textwidth]{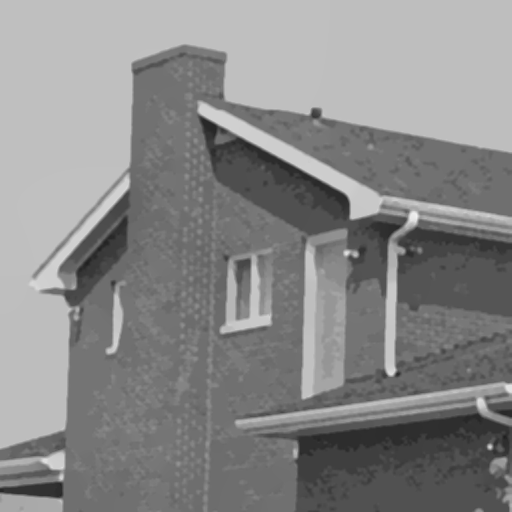}
        \caption{$w$ (non-linear)}
        \label{fig:cdr_linear_3}
    \end{subfigure}
    \begin{subfigure}[b]{0.2\textwidth}
        \centering
        \includegraphics[width=\textwidth]{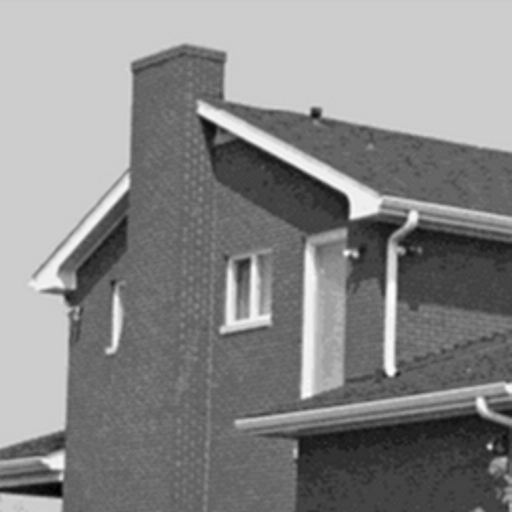} 
        \caption{Reference $v$}
        \label{fig:cdr_linear_4}
    \end{subfigure}
    \caption{Compact data representation using the proposed non-linear osmosis model.}
    \label{fig:cdr_linear}
\end{figure}

\section{Analysis of the continuous model}\label{sec:model}

Given a rectangular image domain $\Omega \subset \R^2$ with boundary $\partial \Omega$ and a finite time $T>0$,
let $\dd:\Omega\to\R^2$ be a given drift
vector field, 
$f \in L^\infty(\Omega;\R_{+}^*)$  a positive greyscale image (the extension to the RGB case is straightforward)
and $g:\Omega\times (0,T]\to \R_{+}^*$ be a positive (non-linear) diffusivity function driving the evolution process. The proposed non-linear osmosis model is given by the following drift-diffusion PDEs:
\begin{equation}
    \partial_t u(x,t) = \dv \left( g(x,t) \, \left(\nabla u(x,t) - \dd(x) u(x,t)\right)\right) \quad  \text{for }(x,t) \in \Omega \times (0,T], \label{eq:g-osmosis} 
\end{equation}
which we endow with homogenous Neumann boundary conditions and initial condition
\begin{align}
\label{eq:g-osmosis1}
    \begin{cases}
    \left< \nabla u - \dd u , \nbold \right> = 0 & \text{on } \partial \Omega \times (0,T];
    \\
    u(x,0) = f(x) & \text{on } \Omega .  
    \end{cases}
\end{align}

The model \eqref{eq:g-osmosis} favors diffusion and transport according to the function $g(\cdot)$.
The following result shows that thanks to its divergence form, model \eqref{eq:g-osmosis} enjoys standard conservation properties for any $g$.

\begin{prop}[Conservation properties]\label{prop:properties}
Any solution $u:\Omega\times (0,T] \rightarrow\R$ of the non-linear osmosis process \eqref{eq:g-osmosis} satisfies the following properties:
\begin{enumerate}
    \item The average grey value is preserved:
    \begin{equation}
        \frac{1}{|\Omega|} \int_\Omega u(x,t) dx = \frac{1}{|\Omega|} \int_\Omega f(x) dx =: \mu(f) \qquad \forall t>0;
    \end{equation}
    \item The evolution preserves non-negativity:
    \begin{equation}
        u(x,t)\geq 0 \quad \forall x\in\Omega, \forall t>0.
    \end{equation}
\end{enumerate}
\end{prop}
\begin{proof}
1) Define, for all $t\in(0,T]$, $m(t) := ( \int_\Omega u(x,t)dx ) / |\Omega|$. Then, by \eqref{eq:g-osmosis},
\begin{align*}
    \frac{dm}{dt} &= \frac{1}{|\Omega|} \int_\Omega \partial_t u(x,t) dx = \frac{1}{|\Omega|}\int_\Omega \dv \left( g(x,t)(\nabla u(x,t) - \dd(x) u(x,t)) \right)  dx = \\
    &= \frac{1}{|\Omega|} \int_{\partial\Omega} \left< g(s,t)(\nabla u(s,t) - \dd(s) u(s,t)) , \nbold(s) \right> ds = 0,\qquad \forall t>0
\end{align*}
thanks to the divergence theorem and by imposing Neumann boundary conditions.

2) Let now $\tau>0$ be the smallest time such that $\min_x u(x,t) = 0$. Suppose that this minimum is obtained at a $\xi \in \text{int}(\Omega)$. Then $u(\xi,\tau)=0$ and $\nabla u(\xi,\tau)=0$. Computing the time derivative of $u$ at point $(\xi,\tau)$ gives:
\begin{align*}
    \partial_t u(\xi, \tau) &= \left( \dv ( g(\cdot,\cdot)(\nabla u(\cdot,\cdot) - \dd(\cdot) u(\cdot,\cdot)) \right) (\xi,\tau) = \\
    &= g(\xi,\tau) \Delta u(\xi,\tau) + \left< \nabla u (\xi,\tau), \nabla g(\xi,\tau) \right> - \dv (\dd(\cdot) g(\cdot,\cdot) u(\cdot,\cdot)) (\xi,\tau) = \\
    &= g(\xi,\tau) \Delta u(\xi,\tau) + 0 - (\dv~\dd(\cdot))(\xi,\tau)\cdot (u(\xi,\tau)g(\xi,\tau)) - \left< \dd(\cdot),\nabla(g(\cdot,\cdot)u(\cdot,\cdot)) \right>(\xi,\tau) = \\
    &= g(\xi,\tau) \Delta u(\xi,\tau) - 0 - \left< \dd(\cdot),\nabla u(\cdot,\cdot)~g(\cdot,\cdot) + u(\cdot,\cdot)\nabla g(\cdot,\cdot) \right>(\xi,\tau) = \\
    & = g(\xi,\tau) \Delta u(\xi,\tau) - 0 = g(\xi,\tau) \Delta u(\xi,\tau),
\end{align*}
by standard properties of the divergence operator.
Then, at $(\xi,\tau)$ the evolution behaves as the diffusion equation $u_t = g\Delta u$, where $g$ is positive. Hence, the operator $\mathcal{L}:= g\Delta$ is elliptic so the standard minimum/maximum principle can be applied.
This tells us that for any $t \geq \tau$ the solution of the non-linear model remains non-negative, as desired.
\end{proof}

We now propose a special choice of the diffusivity function $g: \Omega  \times (0,T] \to \R_{+}^*$ and set for all $x\in\Omega$ and $t\in(0,T]$:
\begin{equation}\label{eq:g_def}
   g(x,t):= g( |\sbold(x,t)|):=\frac{1}{| \sbold (x,t)|}, \quad  \text{with } \;  \sbold(x,t):=\nabla u(x,t) - \dd(x) u(x,t),
\end{equation}
which will be used to characterise the steady states of \eqref{eq:g-osmosis} as minimisers of a suitable energy functional. 
To show that, we first need the following lemma.
\begin{lemma}\label{lemma:iff}
Let $\Omega$ be an open rectangular domain in $\R^2$ and $\mathbf{w}\in C^1(\Omega,\R^2)\cap C^0(\overline{\Omega},\R^2)$. Then:
\begin{equation}\label{eq:w_diff}
    \int_{\partial\Omega} \left<\mathbf{w},\nbold\right> \, z\;ds - \int_\Omega \dv (\mathbf{w})\,z\;dx = 0 \quad \forall \; z\in C^0(\overline{\Omega}) 
\end{equation}
if and only if
\begin{equation}\label{eq:w_pde}
    \begin{cases}
    \left< \mathbf{w},\nbold\right> = 0 \quad on \; \partial\Omega \\
    \dv(\mathbf{w}) = 0 \quad on \; \Omega.
    \end{cases}
\end{equation}
\end{lemma}
\begin{proof}
Eq. \eqref{eq:w_pde} implies eq. \eqref{eq:w_diff} obviously. For the inverse implication, consider the two alternative cases: either $\dv(\mathbf{w})(x)=0$ for all $x\in\Omega$, or there exists a point $x_0\in\Omega$ s.t. $\dv(\mathbf{w})(x_0)\neq 0$.
We prove by contradiction that the second case is not possible. Suppose that $\dv(\mathbf{w})(x_0)>0$. Since $\mathbf{w}\in C^1(\Omega;\R^2)$, then there exists a neighbourhood $B(x_0,r)\subset\Omega$ s.t. $\dv(\mathbf{w})(x)>0$ for all $x\in B(x_0,r)$. Now, let $z\in C^0(\overline{\Omega})$ be a function that satisfies \eqref{eq:w_diff}. Consider 
$$ z^* := z+\Phi(\chi_{B(x_0,r)}) \in C^0(\overline{\Omega})$$
where $\Phi(\chi_{B(x_0,r)})$ is a non-negative smooth regularization of $\chi_{B(x_0,r)}$, with support in $B(x_0,r)$. Then, we have
\begin{align*}
    \int_{\partial\Omega} \left<\mathbf{w},\nbold\right> \, z^*\;&ds - \int_\Omega \dv(\mathbf{w})\,z^*\;dx = \\
    &=\int_{\partial\Omega} \left<\mathbf{w},\nbold\right> \, z\;ds - \int_\Omega \dv(\mathbf{w})\,z\;dx - \int_\Omega \dv(\mathbf{w}) \Phi(\chi_{B(x_0,r)}) \, dx = \\
    & =  0 - \int_\Omega \dv(\mathbf{w}) \Phi(\chi_{B(x_0,r)}) \, dx =
    \\
    & =  - \int_{B(x_0,r)} \dv(\mathbf{w}) \Phi(\chi_{B(x_0,r)}) \, dx < 0 
\end{align*}
because $supp(\Phi(\chi_{B(x_0,r)}))\subset B(x_0,r)$, $\Phi(x)\geq 0$ and $\dv(\mathbf{w})>0$ in $B(x_0,r)$. We have thus reached a contradiction, as for $z^*\in C^0(\overline{\Omega})$ equation \eqref{eq:w_diff} is not verified. Therefore, it must be $\dv(\mathbf{w})(x)=0$ for all $x\in\Omega$. We claim that this implies also that
\begin{equation}\label{eq:wn}
   \int_{\partial\Omega} \left<\mathbf{w},\nbold\right> \, z\;ds = 0 \quad \forall \; z\in C^0(\overline{\Omega},\R).
\end{equation}

By contradiction, consider a point $x_1\in\partial\Omega$ s.t. $\left<\mathbf{w},\nbold\right>(x_1)\neq 0$ (suppose positive). Since $\mathbf{w}\in C^0(\overline{\Omega};\R^2)$, there exists $B(x_1,r)$ s.t. $\left<\mathbf{w},\nbold\right>(x)>0$ for all $x\in B(x_1,r)\cap \partial\Omega$. Let now $z\in C^0(\overline{\Omega})$ be a function satisfying \eqref{eq:wn} and define $z^*:=z+\tilde{\Phi}(\chi_{B(x_1,r)}) $,
where $\tilde{\Phi}(\cdot)$ is as a smooth regularisation of $\chi_{B(x_1,r)}$ with support in $B(x_1,r)$.
Then,
\begin{align*}
    \int_{\partial\Omega} \left<\mathbf{w},\nbold\right> \, z^*\;ds &= \int_{\partial\Omega} \left<\mathbf{w},\nbold\right> \, z\;ds + \int_{\partial\Omega} \left<\mathbf{w},\nbold\right> \, \Phi(\chi_{B(x_1,r)}) \;ds = \\
    & =  0 + \int_{\partial\Omega} \left<\mathbf{w},\nbold\right> \, \Phi(\chi_{B(x_1,r)}) \; ds =
    \\
    & = \int_{B(x_1,r)\cap\partial\Omega} \left<\mathbf{w},\nbold\right> \, \Phi(\chi_{B(x_1,r)}) \;ds > 0
\end{align*}
because $supp(\Phi(\chi_{B(x_1,r)}))\cap\partial\Omega\subset B(x_1,r)\cap\partial\Omega$, $\Phi(x)\geq 0$ and $\left<\mathbf{w},\nbold\right>(x)>0$ in $B(x_1,r)\cap\partial\Omega$. 

We thus built a function $z^*$ that does not satisfy \eqref{eq:wn}, which is of course a contradiction. Therefore, $\left<\mathbf{w},\nbold\right>(x)=0$ for all $x\in\partial\Omega$.
\end{proof}

Thanks to this lemma, we are now able to interpret the non-linear PDE model \eqref{eq:g-osmosis} as the gradient flow of a suitably defined energy functional $E(u)$ in the special case when the  diffusivity function $g$ is defined as in \eqref{eq:g_def}.

\begin{prop}\label{prop:energy}
Let $v\in H^1(\Omega;\R_+^*)$ be a  given reference image, $\dd=\nabla(\log v)$ the canonical drift vector field associated to $v$ and $g$ be defined as in \eqref{eq:g_def}.
Then, a function $u^*: \Omega\times (0,T]\to \R_+$ satisfies the following steady state equation
    \begin{equation}\label{eq:sse}
        \dv \left(  g( | \nabla u^* -\dd u^*|) \left(\nabla u^* - \dd u^*\right)\right) = 0
    \end{equation}
    with boundary condition $ \left<\nabla u^* - \dd u^*,\nbold\right>=0 $,
    if and only if $u^*$ is a stationary point of the energy functional
    \begin{equation}\label{eq:E1}
        E(u) := \int_\Omega \left|\nabla\left(\frac{u}{v}\right)(x)\right| dx.
    \end{equation}
    Moreover, $u^*$ is a multiplicative rescaling of  $v$ where the rescaling constant is the ratio between the average gray value of $f=u(\cdot,0)$ and $v$, i.e.:
    $$
    u^*(x)=\frac{\mu(f)}{\mu(v)} v(x),\quad\forall x\in\Omega.
    $$
\end{prop}
\begin{proof}
Conventionally, for $x\in\Omega$ and $t\in(0,T]$ such that $\nabla u(x,t) - \dd(x) u(x,t)=0$, we set
$$ \frac{\nabla u(x,t) - \dd(x) u(x,t)}{|\nabla u(x,t) - \dd(x) u(x,t)|}=0. $$
Let us now consider any test function $z\in C^\infty_c(\Omega;R_+^*)$. We have
\begin{align}\label{eq:dt_E}
    \left(\frac{d}{d\eta}E(u+\eta z)\right)& _{\big|{\eta=0}} = \left( \int_\Omega \frac{d}{d\eta} \left|\nabla\left(\frac{u+\eta z}{v}\right)\right| dx \right) _{\big|{\eta=0}} = \nonumber \\
    &= \left( \int_\Omega \left< \frac{\nabla(\frac{u+\eta z}{v})}{|\nabla(\frac{u+\eta z}{v})|} , \nabla \left(\frac{z}{v}\right)\right> dx \right)_{\big|{\eta=0}} = \nonumber \\
    &=  \int_\Omega \left< \frac{\nabla(\frac{u}{v})}{|\nabla(\frac{u}{v})|} , \nabla \left(\frac{z}{v}\right)\right> dx  = \nonumber \\
    &=  \int_\Omega \left< \frac{\nabla u - \dd u}{|\nabla u - \dd u|} , \nabla\left( \frac{z}{v}\right)\right> dx  = \nonumber \\
    & = \int_\Omega \left[ \dv\left( \frac{\nabla u - \dd u}{|\nabla u - \dd u|}\cdot \frac{z}{v} \right) - \frac{z}{v}\, \dv \left( \frac{\nabla u - \dd u}{|\nabla u - \dd u|} \right) \right] dx = \nonumber \\
    &= \int_{\partial\Omega} \left< \frac{\nabla u - \dd u}{|\nabla u - \dd u|} \cdot \frac{z}{v}, \nbold \right> ds - \int_\Omega \frac{z}{v}\, \dv \left( \frac{\nabla u - \dd u}{|\nabla u - \dd u|} \right) dx 
\end{align}
The function $u^*$ satisfying \eqref{eq:sse} and the homogeneous Neumann boundary condition is thus a stationary point of the energy $E(u)$, i.e. \eqref{eq:dt_E} vanishes at $u^*$ and the reverse holds thanks to Lemma \ref{lemma:iff}.

To conclude the proof, we have that it is easy to observe that every image $u^*=cv$, with $c\in\R_+^*$, is a stationary point of the energy $E(\cdot)$. Hence, such functions are solutions of the steady-state equation \eqref{eq:sse}.
Since the osmosis evolution preserves the average grey value by Proposition \ref{prop:properties}, one can further show that:
\[
\mu(f) = \mu(u^*) = \frac{1}{|\Omega|}\int_{\Omega} c v(x) dx = c \mu(v),
\]
whence $c = \mu(f)/\mu(v) >0$.
\end{proof}

An analogous result can be proved as a corollary of the previous proposition by allowing the further dependence on a smoothing parameter $p\in[1,2)$. In Section \ref{sec:results}, we will validate our model for different values of $p$, evaluating the practical effects of its choice on exemplar tests.

\begin{cor}\label{cor:th}
Under the same assumptions of Proposition \ref{prop:energy} and defining for $p\in[1,2)$ the diffusivity term $g_p$ as
\begin{equation}\label{eq:gp}
g_p(x,t):=g_p(|\mathbf{s}(x,t)|):=\frac{1}{|\mathbf{s}(x,t)|^{p}},\qquad \forall (x,t)\in\Omega\times (0,T], 
\end{equation}
then the steady-state of the energy functional 
\begin{equation}\label{eq:Ep}
    E_p(u) := \int_\Omega \frac{v^{1-p}}{2-p}\left|\nabla\left(\frac{u}{v}\right)(x)\right|^{2-p} \; dx
\end{equation}
satisfies:
    \begin{equation}\label{eq:p-osmosis}
        \dv \left(  g_p( | \nabla u^* -\dd u^*|) \left(\nabla u^* - \dd u^*\right)\right) = 0.
    \end{equation}
\end{cor}
A possible direction of future work consists in  investigating the $p$-osmosis model  \eqref{eq:p-osmosis} so as to extend the osmosis filtering to a spatially adaptive filtering, as proposed for image denoising, e.g., in \cite{p-laplace}.

\section{Model discretisation using finite difference schemes}\label{sec:discr}

In this section we introduce fast and stable computational
methods for the nonlinear osmosis PDE model \eqref{eq:g-osmosis} by rewriting it using the linearity of the divergence operator as follows
\begin{align}
\label{eq:sep}
    \partial_t u = \dv \left( g(u) \nabla u\right)  - \dv \left( g(u) \dd u\right),
\end{align}
along with the boundary conditions and initial condition $f\in L^{\infty}(\Omega;\R_+^*)$ as in \eqref{eq:g-osmosis1}.

We use Finite Difference Method (FDM)  to approximate the solution of \eqref{eq:sep} on a regular grid $\Omega_h\subset \Omega$ approximating the space domain $\Omega$ of the PDE, where $h>0$ denotes the spatial grid size along both horizontal and vertical directions.
We first apply a semi-discretisation in space on $\Omega_h$, and then a full discretisation in time on $(0,T]$ using both explicit and semi-implicit strategies, with the final objective of providing a numerical scheme which shows stability to any fixed time-step discretisation parameter.

\medskip

In the following, we thus assume $u\in\R^{m_x\times m_y}$ to be a discretised image defined on $\Omega_h$  of size $N:=m_x m_y$ and we denote by $u_{i,j}$ the value of $u$ at pixel $(i,j)$ for $i=1,\ldots,m_x$, $j=1,\ldots,m_y$. To avoid singularities appearing due to the special choice of the diffusivity function $g$ in \eqref{eq:g_def}, we will use an $\epsilon$-regularised version defined in terms of $0<\epsilon\ll 1$ 
as $g(s)=1/|\mathbf{s}|_{\epsilon}$ with $\mathbf{s}=\nabla u -\dd u$, with 
$$
 |\mathbf{s}|_{\epsilon}:=\sqrt{(s^x)^2+(s^y)^2+\epsilon}
$$
where $s^x$ and $s^y$ denote the horizontal and vertical components of $\mathbf{s}$, respectively, and the smoothing parameter $\epsilon>0$ ensures that $g(s)$ is well-defined at each point.
Considering the canonical drift term $\dd=\nabla \log v = \nabla v / v$, the value of $g$ at $(i,j)$ is approximated using central finite difference schemes for the discretization of the partial first-order derivatives by
\begin{align}\label{eq:g-discr}
g_{i,j} &= g(s_{ij})
\\
s_{ij} &= \left( \frac{u_{i+1,j}-u_{i-1,j}}{2h} - \frac{v_{i+1,j}-v_{i-1,j}}{2h} \cdot \frac{u_{i,j}}{v_{i,j}} \;,
 \frac{u_{i,j+1}-u_{i,j-1}}{2h} - \frac{v_{i,j+1}-v_{i,j-1}}{2h} \cdot \frac{u_{i,j}}{v_{i,j}} \right).  \notag
\end{align}

\begin{figure}[h]
    \centering
    \includegraphics[width=0.4\textwidth]{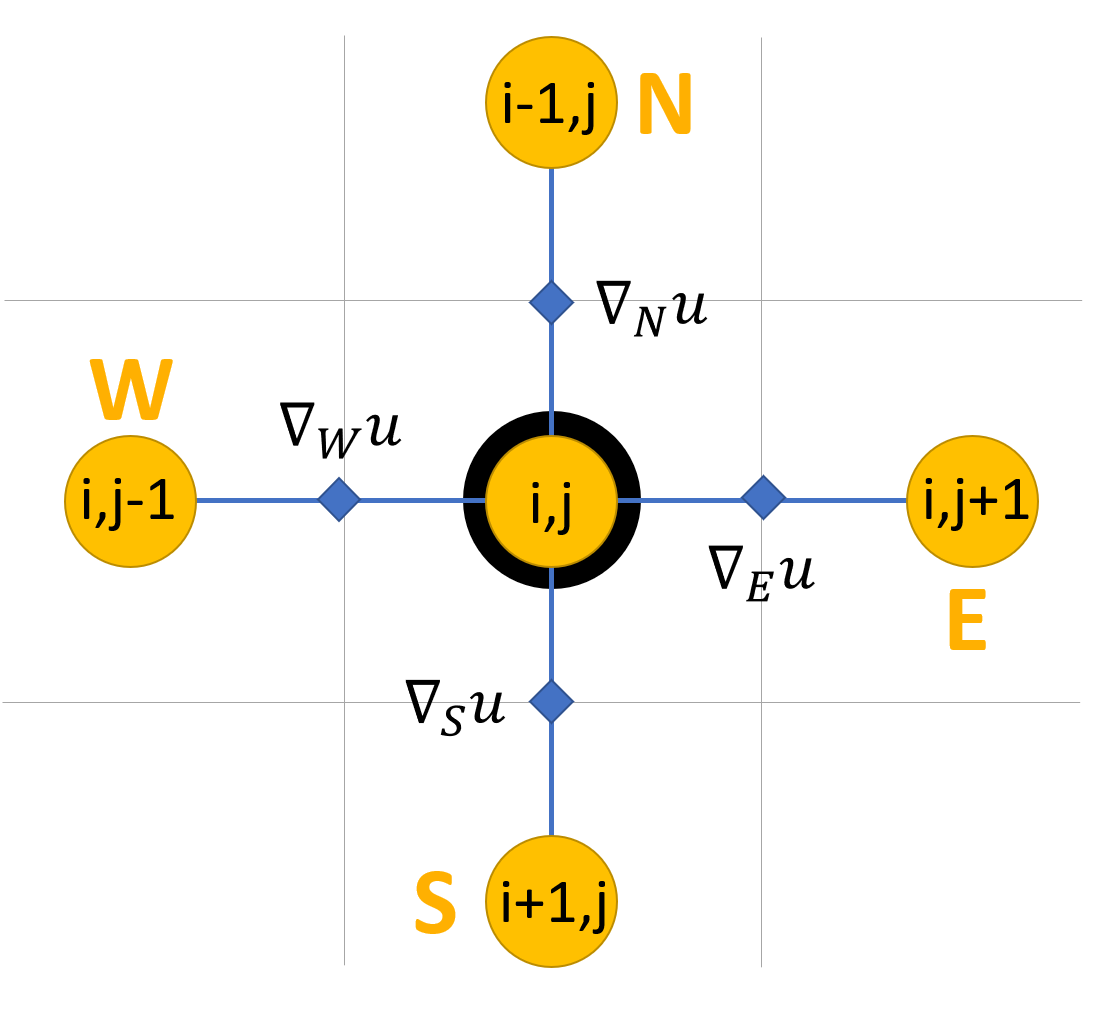}
     \caption{Finite difference stencil: the pixel neighbourhood of $(i,j)$ are denoted by south (S), north (N), east (E), west (W).}
    \label{fig:scheme}
\end{figure}

\subsection{Semi-discretisation in space}

The spatial differential operators in \eqref{eq:sep},
will be discretised by locating the values of $\nabla u$, $\dd =(d^x,d^y)$, $u$ and the $\epsilon$-regularised diffusivity function $g$ on the edges between two adjacent pixels, following the stencil illustrated in Fig. \ref{fig:scheme}. Hence, the semi-discretisation in space of \eqref{eq:sep} at $(i,j)$ applying second-order central finite differences reads:
\begin{align}\label{eq:discr}
    (\partial_t & u)_{ij}
     = \frac{1}{h} \Big[ \left[ \nabla_N u \cdot g_{N_{i,j}} + \nabla_S u \cdot g_{S_{i,j}} + \nabla_E u \cdot g_{E_{i,j}} + \nabla_W u \cdot g_{W_{i,j}} \right] + \nonumber
    \\
    & -
    \left[ d_N^y \cdot g_{N_{i,j}} \cdot u_{N_{i,j}} + d^y_S \cdot g_{S_{i,j}} \cdot u_{S_{i,j}} + d_E^x \cdot g_{E_{i,j}} \cdot u_{E_{i,j}} + d_W^x \cdot g_{W_{i,j}} \cdot u_{E_{i,j}} \right]
    \Big]
\end{align}
where each term is defined as in Table \ref{tab1}.

\begin{table}[h]
\begin{center}
{\renewcommand{\arraystretch}{2}
\begin{tabular}{|cccc|}
    \hline
    $\nabla_N u = \frac{u_{i+1,j}-u_{i,j}}{h}$, &
    $\nabla_S u = \frac{u_{i-1,j}-u_{i,j}}{h}$, &
    $\nabla_E u = \frac{u_{i,j+1}-u_{i,j}}{h}$, &
    $\nabla_W u = \frac{u_{i,j-1}-u_{i,j}}{h}$,
    \\ 
    \hline
    $g_{N} = \frac{g_{i+1,j}+g_{i,j}}{2}$, &
    $g_{S} = \frac{g_{i-1,j}+g_{i,j}}{2}$, & 
    $g_{E} = \frac{g_{i,j+1}+g_{i,j}}{2}$, &
    $g_{W} = \frac{g_{i,j-1}+g_{i,j}}{2}$,
    \\
    \hline
    $u_{N} = \frac{u_{i+1,j}+u_{i,j}}{2}$, &
    $u_{S} = \frac{u_{i-1,j}+u_{i,j}}{2}$, & 
    $u_{E} = \frac{u_{i,j+1}+u_{i,j}}{2}$, &
    $u_{W} = \frac{u_{i,j-1}+u_{i,j}}{2}$,
    \\
    \hline
    $d_N^y = \frac{2(v_{i+1,j}-v_{i,j})}{h(v_{i+1,j}+v_{i,j})}$, &
    $d_S^y = \frac{2(v_{i-1,j}-v_{i,j})}{h(v_{i-1,j}+v_{i,j})}$, & 
    $d_E^x = \frac{2(v_{i,j+1}-v_{i,j})}{h(v_{i,j+1}+v_{i,j})}$, &
    $d_W^x = \frac{2(v_{i,j-1}-v_{i,j})}{h(v_{i,j-1}+v_{i,j})}$.
    \\\hline
\end{tabular}}
\caption{FD discretization schemes.}
\label{tab1}
\end{center}
\end{table}
Note that the discretisation \eqref{eq:discr} also holds on boundary pixels, having preliminarily set $\dd=0$ on the boundary edges and mirroring the image therein in order to satisfy the Neumann homogeneous boundary condition \eqref{eq:g-osmosis}.

We thus rewrite \eqref{eq:discr} by incorporating the discretisations reported in Table \ref{tab1} and collecting the coefficients of the pixel $(i,j)$ and its neighbours:
\begin{align}\label{eq:discr_complete}
    (\partial_t u)_{i,j} &= 
    \,   u_{i-1,j}\, \frac{g_{S}}{h} \left( \frac{1}{h} - \frac{d^y_S}{2} \right) 
    + u_{i+1,j}\, \frac{g_{N}}{h} \left(  \frac{1}{h} - \frac{d^y_N}{2} \right) + \nonumber \\
    +& u_{i,j-1}\, \frac{g_{W}}{h} \left(  \frac{1}{h} - \frac{d^x_W}{2} \right)
    + u_{i,j+1}\, \frac{g_{E}}{h} \left(  \frac{1}{h} - \frac{d^x_E}{2} \right) +  \\
    +& u_{i,j} \Bigg[ 
    \frac{g_{S}}{h} \left( - \frac{1}{h} - \frac{d^y_S}{2} \right) +
    \frac{g_{N}}{h} \left( - \frac{1}{h} - \frac{d^y_N}{2} \right) + %\\
    \frac{g_{W}}{h} \left( - \frac{1}{h} - \frac{d^x_W}{2} \right) +
    \frac{g_{E}}{h} \left( - \frac{1}{h} - \frac{d^x_E}{2} \right)
      \Bigg] . \nonumber
\end{align}

In order to represent \eqref{eq:discr_complete} in matrix-vector form, we first rearrange into a single vector 
$u(t)= (u_1(t),\ldots,u_N(t))^T\in\R^N$ the values of the evolving image $u(T)$ at each pixel $i=1,\ldots,N$.
Then, we write semi-discretised-in-space model which leads to the following Cauchy problem:
\begin{equation}\label{eq:M_form}
    \begin{cases}
    \frac{d u}{d t} = A(u)\,u(t), \\
    u(0) = f,
    \end{cases}
\end{equation}
where $A(u)\in\R^{N\times N}$ is a non-symmetric sparse matrix, with only five non-zero diagonals and with entries given by
\begin{align}\label{eq:A}
    a_{ij}:=
    \begin{cases}
    {\displaystyle \sum_{k \in \mathcal{N}(i)} \frac{g_k+g_i}{2h^2}\left( -1- \frac{v_k-v_i}{v_k+v_i}\right)} & \text{if } i=j,
    \\
    {\displaystyle \frac{g_j+g_i}{2h^2}\left( 1- \frac{v_j-v_i}{v_j+v_i}\right)} \quad &\text{if $i\neq j$, $j\in\mathcal{N}(i)$},
    \\
    0 & \text{if $i\neq j$, $j\not\in\mathcal{N}(i)$.}
    \end{cases}
\end{align}
where $\mathcal{N}(i)$ represents the set of the neighbours in directions $\{N,E,O,W\}$ of the pixel $i$ and $g_i$ is evaluated at pixel $i$ according to formula \eqref{eq:g-discr}.

\medskip

The next proposition shows some useful properties of the matrix $A(u)$.

\begin{prop}
For every vector $u\in\R^N$, let $A(u)\in\R^{N\times N}$ be the matrix with entries given in \eqref{eq:A}. Then, the following properties hold:
\begin{enumerate}
    \item $A(u)$ has non-negative off-diagonals,
    \item all column sums of $A(u)$ are 0,
    \item $A(u)$ is irreducible,
    \item $A(u)v=0$.
\end{enumerate}
\end{prop}
\begin{proof}
1) According to \eqref{eq:A}, if pixel $i$ and pixel $j$ are not adjacent, then the entry $a_{ij}$ is zero. On the other hand, as $g_i,g_j,v_i,v_j> 0$, then $a_{ij}>0, \, \forall i,j$.

2) Considering an arbitrary column $l$ of $A$, we have:
\begin{align*}
\sum_{i=1}^{N} a_{i,l} &= a_{l,l} + \sum_{j\in\mathcal{N}(l)} a_{j,l} = \\
&= \sum_{j\in\mathcal{N}(l)} \frac{g_j+g_l}{2h^2}\left( -1- \frac{v_j-v_l}{v_j+v_l}\right) + \sum_{j\in\mathcal{N}(l)} \frac{g_l+g_j}{2h^2}\left( 1- \frac{v_l-v_j}{v_l+v_j}\right) = 0.
\end{align*}

3) Since $a_{i,j}>0$ for all adjacent $i,j$, so the directed graph $G_A$ associated to $A$ is strongly connected and $A$ is irreducible.

4) The value of the $ij$ coordinate of the vector $A(u)v$ is obtained from the right hand side of the expression \eqref{eq:discr}, substituting $u$ with $v$. Writing explicitly every term as described in Table \eqref{tab1}, it is easy to see that they cancels each other. Therefore $A(u)v=0$.
\end{proof}

The properties showed in the lemma above are crucial for the following. Typically, they are proved in the context of symmetric space-discretisation matrices (see, e.g., \cite{weickert98}), but have been shown to hold also in the case of non-symmetric diffusion-transport operators such as the osmosis one, see, e.g., \cite{vogel-ssvm,Parisotto2019AnisotropicOF}.

\subsection{Fully-discretisation in time}

We present two possible fully discretised approximations of \eqref{eq:M_form}, obtained by applying Euler integration methods in time. 
Let $\tau>0$ denote a uniform time-step and for $k=0,1,\ldots$ let  $u^k$ be the approximation of $u$ at discrete time $k\tau \in (0,T]$.

\begin{itemize}
\item \textbf{Explicit discretisation}: We first consider the simplest space discretisation  
at the $k$th discrete time step
that leads to the following full explicit scheme 
\begin{equation}\label{eq:explicit}
    \frac{u^{k+1}-u^k}{\tau} = A(u^k) u^k \qquad \Longrightarrow \qquad  u^{k+1} = (I+\tau A(u^k)) u^k. \tag{E}
\end{equation}
\item \textbf{Semi-implicit discretisation}: Alternatively, we can consider a semi-implicit scheme where the nonlinear terms of the equation
are treated from the previous time step, while the linear ones are considered at the
current time step; this leads to the semi-implicit scheme
\begin{equation}\label{eq:semi-implicit}
    \frac{u^{k+1}-u^k}{\tau} = A(u^k) u^{k+1} \qquad \Longrightarrow \qquad  u^{k+1} = (I-\tau A(u^k))^{-1} u^k.  \tag{S.I.}
\end{equation}
\end{itemize}
In both cases, the time-stepping process approximating \eqref{eq:M_form} can be represented in terms of a matrix $P\in\R^{N\times N}$ as
\begin{equation}\label{eq:P_process}
    \begin{cases}
    u^{k+1} = P(u^k)\, u^k,\\
    u^0 = f.
    \end{cases}
\end{equation}
where 
\begin{equation}
    P=I+\tau A, \quad  \text{in} \quad \text{\eqref{eq:explicit}}, \qquad
    P=(I-\tau A)^{-1} \quad \text{in} \quad \text{\eqref{eq:semi-implicit}}.
    \label{eq:Pmat}
\end{equation} 

The following proposition provides sufficient conditions on matrix $P$ such that the evolution process
\eqref{eq:P_process} is stable.

\begin{lemma}\label{lemma:P_stab}
Let $f\in(\R_+^*)^N$ and let assume that $P(u^k)\in\R^{N\times N}$, defined in \eqref{eq:Pmat}, satisfies the following assumptions for all $k\in\mathbb{N}$:
\begin{enumerate}
    \item[(A1)] all column sums of $P$ are 1,
    \item[(A2)] $P$ is non-negative.
\end{enumerate}
Then the iterative process \eqref{eq:P_process} is stable with respect to the $1$-norm, i.e. $\|u^k\|_1\leq\|f\|_1$ $\forall k\in\mathbb{N}$.
\end{lemma}
\begin{proof}
For all $k$, we recursively have
\begin{align*}
u^{k} &= P(u^{k-1})u^{k-1} = P(u^{k-1})\cdot P(u^{k-2})\cdot \ldots \cdot P(u^1)\cdot P(f)\cdot f \\
\|u^k\|_1 &\leq \|P(u^{k-1})\|_1\cdot\|P(u^{k-2})\|_1 \cdot \ldots \cdot \|P(u^1)\|_1 \cdot \|P(f)\|_1 \cdot \|f\|_1.
\end{align*}
Since $P(u^k)$ has column sums equal to one, then $\|P(u^j)\|_1=1$ for all $j=1,\ldots,k-1$ and $ \|u^k\|_1 \leq 1 \cdot \|f\|_1 $ for all $k\in\mathbb{N}$.
\end{proof}

We are interested in verifying if the conservation properties of the non-linear osmosis model proved in Proposition \ref{prop:properties} for the continuum setting hold in the discrete case.
The following lemma adds two other conditions on the matrix $P$ needed for the process to preserve the average mass and the non-negativity. Furthermore, it provides a useful information on the localisation of the eigenvalues of $P$.

\begin{lemma}\label{lemma:P_prop_discr}
Let $f\in(\R_+^*)^N$ and consider the iterative process \eqref{eq:P_process}, where, for all $k$, the matrix $P(u^k)\in\R^{N\times N}$ satisfies both the assumptions (A1)-(A2) in Lemma \ref{lemma:P_stab} and the following 
\begin{itemize}
    \item[(A3)]$P$ is irreducible,
    \item[(A4)]$P$ has only positive diagonal entries.
\end{itemize}
Then we have
\begin{enumerate}
    \item[(1)] the average grey value is preserved:
    $$ \sum_{i=1}^N u_i^k = \sum_{i=1}^N f_i,$$
    \item[(2)] the evolution preserves non-negativity: $u_i^k\geq 0$ for all $i=1,...,N$, $k>0$,
    \item[(3)] $P$ has a simple eigenvalue $\lambda=1$, with a corresponding positive eigenvector, and the other eigenvalues have absolute value strictly less than one.
\end{enumerate}
\end{lemma}
\begin{proof}
(1) Thanks to the hypotheses (A1) of Lemma \ref{lemma:P_stab}, we can write
$$ \sum_{i=1}^N u_i^k = \sum_{i=1}^N \sum_{j=1}^N P_{ij}u_j^{k-1} = \sum_{j=1}^N \left(\sum_{i=1}^N P_{ij}\right) u_j^{k-1} = \sum_{j=1}^N u_j^{k-1} = \ldots = \sum_{j=1}^N f_j. $$

(2) Applying the first step to the non-negative initial image $f$, we have
$$ u_i^1 = p_{ii} f_i + \sum_{i\neq j} p_{ij} f_j \geq 0 \qquad i=1,\ldots,N $$
since $f_i, f_j, p_{ii},p_{ij}\geq 0$. This observation can be applied recursively for all $k$.

(3) In Lemma \ref{lemma:P_stab} we proved that $\|P\|_1=1 $ and, therefore, the eigenvalues $\lambda$ of $P$ satisfies $|\lambda|\leq 1$.
Moreover, $\lambda=1$ is an eigenvalue of $P$. In fact, if we define the vector $e=(1,\ldots,1)$, then $ P^T e^T = e^T$, because column-wise $P$ sums up to one.
By Gerschgorin's theorem, we know that the eigenvalues lie within the set $\Lambda$ of all Gerschgorin's disk:
$$ \Lambda := \bigcup_{j=1}^N B_j \qquad \text{with} \qquad B_j:=\left\{ z\in\mathbb{C}\, \Big| \, |z-p_{jj}|\leq \sum_{i\neq j} |p_{ij}| \right\}. $$
Now, since $P$ is non-negative with unit column sums and positive diagonal entries, the right hand side of the last inequality becomes
$$\sum_{i\neq j} |p_{ij}| = \sum_{i\neq j} p_{ij} = 1-p_{jj} \in [0,1).$$
Therefore $\lambda=1$ is the only eigenvalue of $P$ s.t. $|\lambda|=1$, while for the others $|\lambda|<1$. Moreover, $P$ has spectral radius $\rho(P)=1$. By Perron-Frobenius theorem (see, e.g., Th.8.4.4 in \cite{horn_johnson_1985}) we can thus conclude as the spectral radius of a $P$ is a simple eigenvalue associated to  a positive eigenvector.
\end{proof}

Up to this point, we have defined sufficient conditions on the matrix $P$ to make  the evolution process \eqref{eq:P_process}  stable for any time step $k$, and guarantee the desired conservation properties.
Specifically, we stated that $P$
must satisfy the assumptions \textit{(A1)-(A4)} for all $k$.

With the following theorem, we show that both the explicit \eqref{eq:explicit} and the semi-implicit \eqref{eq:semi-implicit} time-discretisation schemes verify such conditions.

\begin{teo}[Numerical Stability]
Let $f\in(\R_+^*)^N$ and for $u\in\R^N$ let $A$ be the matrix with entries $a_{ij}$, with $i,j=1,\ldots,N$, defined in \eqref{eq:A}. Then
\begin{enumerate}
	\item the explicit scheme \eqref{eq:explicit} is conditionally stable for
	\begin{equation}
	    \tau\leq 1/\max\{|a_{11}|,...,|a_{NN}|\}
	\label{eq:stabcond}
	\end{equation}	
	\item the semi-implicit scheme \eqref{eq:semi-implicit} is unconditionally stable for any $\tau>0$.
\end{enumerate}
\end{teo}
\begin{proof}
Thanks to Lemma \ref{lemma:P_stab}, we simply need to show that, in both cases, $P$ is non-negative and has column sums equal to 1.

1) In the explicit scheme, $P=I+\tau A$, where $A$ has non-negative off-diagonals and column sums zero. Thus $P$ has column sums 1 for any $\tau$ and  is non-negative only if condition \eqref{eq:stabcond} is satisfied.

2) In the semi-implicit scheme, $P=(I-\tau A)^{-1}$. 
$A$ has column sums zero and for all $j=1,...,N$ we have
\begin{align*}
    -a_{jj} = \sum_{i\neq j} a_{ij} \quad \Longrightarrow \quad  
    1-\tau a_{jj} > \tau \sum_{i\neq j} a_{ij}.
\end{align*}
Hence, the matrix $I-\tau A$ is strictly column diagonally dominant and non-singular. Moreover, it has non-positive off-diagonal entries. Therefore, $I-\tau A$ is a non-singular M-matrix (see \cite{BERMAN19791}, Th. 6.2.3 ($C_{10}$)) and its inverse has only strictly positive entries (see \cite{BERMAN19791}, Th. 6.2.7). This satisfies (A1) of Lemma \ref{lemma:P_stab}.
Lastly, all column sums of $P=(I-\tau A)^{-1}$ are 1. In fact, if we define $e=(1,\ldots,1)$, we have $e=e(I-\tau A)$ and, equivalently, $e=e(I-\tau A)^{-1}$.
\end{proof}

The following result investigates under which sufficient conditions the discretised solution obtained by solving \eqref{eq:P_process} satisfies the conservation properties introduced in Proposition \ref{prop:properties} along the time iterations $k$ and how a steady-state can be computed. 

\begin{teo}
\label{teo:stab}
Let $f\in(\R_+^*)^N$. Then, the solution $u^k$ computed at discrete time $k\tau$  
with $g$ defined as in \eqref{eq:g-discr} preserves its mass and non-negativity for any $k\geq 1$. Moreover, it converges to the steady-state solution $u^*$ defined by
\begin{equation}
    u^* \, \, = \, \, \frac{\mu_f}{\mu_v} v,
\end{equation}
if
\begin{enumerate}
	\item $\tau\leq 1/\max\{|a_{1,1}|,...,|a_{N,N}|\}$, when the explicit scheme \eqref{eq:explicit} is applied
	\item $\tau>0$, when the semi-implicit scheme \eqref{eq:semi-implicit} is applied.
\end{enumerate}
\end{teo}
\begin{proof}
We start by proving that, in both cases, the assumptions of Lemma \ref{lemma:P_prop_discr} are verified. We have already verified in the proof of Theorem \ref{teo:stab} when $P$ satisfies (A1) and (A2). Two other conditions, namely  (A3) and (A4), still need to be checked.

1) We know that $a_{ii}<0$ for all $i$ because $A$ has column sums zero and non-negative off-diagonal entries. Therefore, the condition on $\tau$ ensures that the diagonal entries of $P$ are strictly positive: $1+\tau a_{ii}>0 \; \forall i=1,...,N$.
$P=I+\tau A$ is irreducible because it has the same directed graph of $A$, which is irreducible.

2) In  Theorem \ref{teo:stab} we showed that $P=(I-\tau A)^{-1}$ has only strictly positive entries. Therefore, its diagonal is positive and it is irreducible because its directed graph is strongly connected.
Hence, independently on which time discretization we chose, explicit or semi-implicit, we can apply Lemma \ref{lemma:P_prop_discr}, eventually with restrictions on $\tau$, which guarantees that the evolution verifies the conservation properties. 
Moreover, the matrices $P(u^k)$ have eigenvalues with magnitude strictly less than one, with the exception of the simple eigenvalue $\lambda = 1$, that has a corresponding positive eigenvector.
Now, we observe that this eigenvector is the same for all matrices $P(u^k)$. In fact, since $A(u^k)v=0$, we have:
\begin{enumerate}
    \item in the explicit case $P(u^k)v=(I+\tau A(u^k))v = v$;
    \item in the semi-implicit case
$ P(u^k)^{-1}v = (I-\tau A(u^k))v=v$ and, hence, $P(u^k)v=v.$ 
\end{enumerate}
Therefore, the evolution $u^{k+1}=P(u^k)u^k$ attenuates all components outside the eigenspace of $\lambda=1$ to zero and the process converges to a positive eigenvector $w$ with the same average grey value as the initial vector $f$. In other terms, as $k\rightarrow\infty$, we have
$$ u^k \longrightarrow w = \frac{\mu_f}{\mu_v} v $$
\end{proof}

Numerically, the proposed semi-implicit scheme \eqref{eq:semi-implicit} 
can be solved at each iteration by means of any efficient preconditioned linear iterative solver
suitable for sparse, diagonally dominant M-matrices, such as e.g., SOR (successive over-relaxation) method. 

\section{Numerical results}\label{sec:results}
As described in Section \ref{sec:appl}, upon specific choices of the drift term $\dd$ in \eqref{eq:g-osmosis}, the nonlinear osmosis model can be applied to the shadow removal (see Section \ref{sec:shrem}), light-spot removal (see Section \ref{sec:lirem}) and the compact image representation (see Section \ref{sec:cdr}) task. We will compare qualitatively and quantitatively the results by our approach with the ones obtained by other osmosis-based methods as well as with  state-of-the-art approaches. Furthermore, some considerations on the computational efficiency of the proposed models will be made. Reconstruction quality will be measured in terms of the well-known Structural Similarity Index (SSIM $\in  [0,1]$) error between the original non-shadowed image and the reconstructed image, due to its intrinsic dependence of luminance, contrast and structure features. Note that the explicit iterative scheme \eqref{eq:explicit} is practically unusable for most of the examples considered, due to the constraint \eqref{eq:stabcond} on the time-step $\tau$ which for the following examples forced constrains of the type $\tau<10^{-8}$ which make the use of such scheme very limited in practice. On the contrary, the unconditional stability of the semi-implicit scheme \eqref{eq:semi-implicit} allows for the use possibly large values of the time-step. In the examples, we used $\tau=10^{3}$, which provides convergence and good level of accuracy in few iterations. For the processing of colour images, the proposed model is run on each R, G, B channel independently.
All algorithms are tested on a AMD Ryzen 5 3450U processor (2.10GHz) with 16GB RAM using MATLAB R2021a. 

\subsection{Shadow removal}\label{sec:shrem}

Recalling the shadow removal problem described in Section \ref{sec:shrem}, we test the non-linear model by setting the starting point of the osmosis evolution and the reference image as $u^0=v=f$, the given image. We remind that for such problem, the drift term $\dd$ is defined as in \eqref{eq:d_shrem}.
This choice corresponds to consider two different evolutions in the regions composing the image domain which explicitly read
\begin{align}
\label{eq:sr_model}
    \begin{cases}
    \partial_t u = \displaystyle{\dv \left( \frac{\nabla u}{|\nabla u|} \right)} \quad &\text{on } \Omega_b \times (0,T]\\
    \\
    \partial_t u = \displaystyle{\dv \left( \frac{\nabla u - \frac{\nabla f}{f}u}{\left|\nabla u - \frac{\nabla f}{f}u\right|} \right)} & \text{on } (\Omega_{in} \cup \Omega_{out}) \times (0,T] \\
    \end{cases}
\end{align}
with boundary conditions defined as in \eqref{eq:g-osmosis1}.
Note that choosing $\dd$ in a canonical form according \eqref{eq:d_shrem} allows to preserve the image content outside $\Omega_b$, while diffusing non-linearly, in the form of a Total Variation flow, on $\Omega_b$ itself.
The intensity balancing effect in the regions $\Omega_{in}$ and $\Omega_{out}$ depends solely by the evolution equation on $\Omega_b$, the region where the drift term is set to zero. Moreover, the diffusivity function $g$ allows to tune the diffusion intensity on the pixels in $\Omega_b$ depending on the value of $\nabla u$. In particular, $g$ slows down the diffusion on the pixels with high gradient values, which correspond to edges that need to be preserved.

We test the performance of the proposed non-linear shadow removal model both in case of hard shadows and for more realistic soft-edged shadows. 

In the case of hard shadows, synthetic examples $f$  can be easily generated by a constant multiplicative rescaling of a ground-truth image $f^*$ only in correspondence of a bounded region $S\subset\Omega$.
Denoting by $c\in(0,1)$ the (unknown) constant loss of luminosity on $S$, we can construct first the shadow image by setting it as:
\begin{equation}%\label{eq:s_shadow}
    s =
    \begin{cases}
    c \quad &\text{on} \; S \\
    1 & \text{on} \; \Omega\setminus S,
    \end{cases}
\end{equation}
and obtain a shadowed version $f$ of $f^*$ by Hadamard (point-wise) multiplication
\begin{equation}
\label{eq:s_shadow}
    f =
    \begin{cases}
    f^* \odot s &\text{(hard shadow)} \\
    f^*\odot \left( G_{\sigma} * s \right) & \text{(soft shadow)} 
    \end{cases}
\end{equation}
where `$*$' denotes the standard convolution product and $G_\sigma$ is a Gaussian convolution kernel with standard deviation $\sigma>0$. Such convolution allows to generate an area of penumbra around shadow boundaries, which is often the case in real-world images where the shadow edges can not be precisely distinguished.

Fig. \ref{fig:hard_shadow} shows the results of the proposed nonlinear osmosis model solved by means of the semi-implicit scheme on two hard-shadowed images. For such images, it is easy to distinguish whether a pixel belongs to the shadowed or to the un-shadowed region. Hence, we can use a thin (2 pixel wide) mask covering one pixel from each side of the boundary between the two regions. This choice is sufficient to accurately remove the shadow while preserving all features on the shadow boundary.

\begin{figure}[h]
    \centering
    \begin{tabular}{cc|c}
        \includegraphics[width=0.16\textwidth]{figs/17.png} & 
        \fbox{\includegraphics[width=0.15\textwidth]{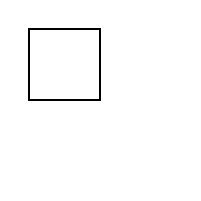}} &
        \includegraphics[width=0.16\textwidth]{figs/17_divmatr_tau1000T100002epsilon1e-07.png} 
        \\
        \includegraphics[width=0.16\textwidth]{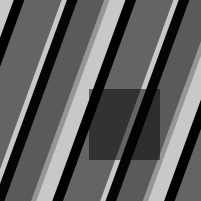} & \fbox{\includegraphics[width=0.15\textwidth]{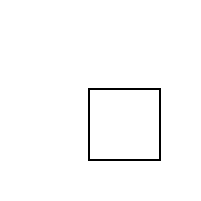}} &
        \includegraphics[width=0.16\textwidth]{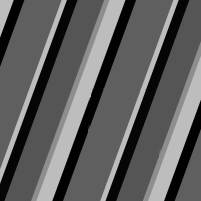}
        \\
        (a) Input $f=v$ & (b) mask 2px & (c) output 
    \end{tabular}
    \caption{Shadow removal on images with hard shadows,.}
    \label{fig:hard_shadow}
\end{figure}

Fig. \ref{fig:soft_shadow} shows the results obtained for the soft-edged shadow removal problem in correspondence of two different masks (wide and thin) and three input images where shadows are characterised by penumbra boundaries of increasing dimension (from left to right) corresponding to different increasing values of $\sigma\in\{0,1,2\}$ in \eqref{eq:s_shadow}.  
\begin{figure}[b!]
    \centering
    \begin{tabular}{c|ccc}
    & $\sigma=0$ & $\sigma = 1$ & $\sigma=2$
    \\
    &
    \includegraphics[width=0.2\textwidth]{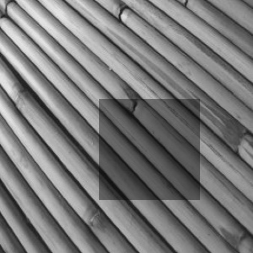} &
    \includegraphics[width=0.2\textwidth]{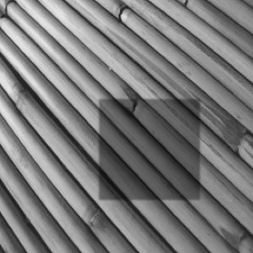} &
    \includegraphics[width=0.2\textwidth]{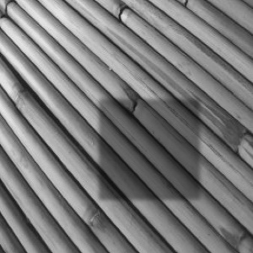}
    \\
    \hline
    6px & & & \\
    \fbox{\includegraphics[width=0.18\textwidth]{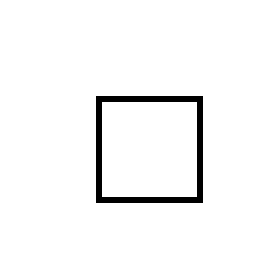}} &
    \includegraphics[width=0.2\textwidth]{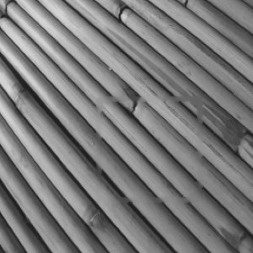} &
    \includegraphics[width=0.2\textwidth]{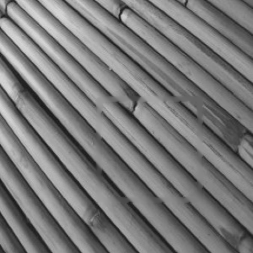} &
    \includegraphics[width=0.2\textwidth]{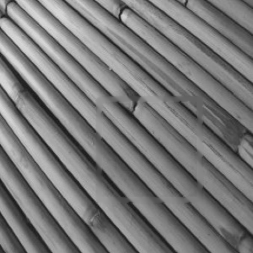}
    \\
    2px & & & \\
    \fbox{\includegraphics[width=0.18\textwidth]{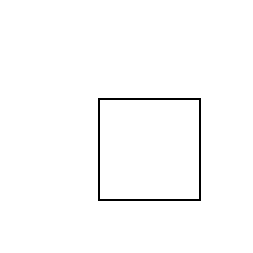}} &
    \includegraphics[width=0.2\textwidth]{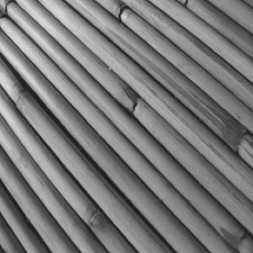} &
    \includegraphics[width=0.2\textwidth]{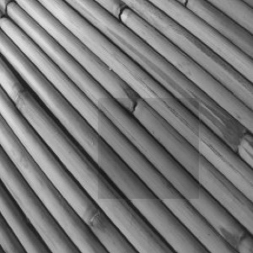} &
    \includegraphics[width=0.2\textwidth]{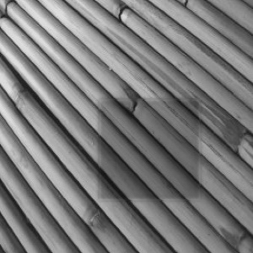}
    \end{tabular}
    \caption{Shadow removal on input images $f$ corrupted by a constant shadow convolved with convolution kernel $G_\sigma$ with increasing standard deviation $\sigma$.
    Two different mask thicknesses, 6px and 2px, are considered.}
    \label{fig:soft_shadow}
\end{figure}
As seen before, the thin mask is a good choice  for the case of hard shadows (i.e. $\sigma=0$), but it becomes less effective as $\sigma$ increases, i.e. as the penumbra region becomes wider, as it cannot cover appropriately the boundary between the two regions $\Omega_{in}$ and $\Omega_{out}$.
In these cases, it is necessary to enlarge the shadow mask. Using a wider one, we observe that the shadows are properly removed and features are well-preserved, in particular near the pixels with limited intensity variation.

Note, however, that using a large mask may result in some drawbacks, as shown in Fig. \ref{fig:sr_draw}, where a wide (6 pixel) mask is used. As this choice enlarges the region where the non-linear TV-type diffusion is induced, typical TV drawbacks are observed. For instance, we observe in Fig. \ref{fig:sr_draw_2} that some artefacts appear in the attempt of connecting circular structures (see the boundary between the yellow and the green region). Furthermore, as it can be observed in Fig. \ref{fig:sr_draw_4},  when the shadow mask is too large some undesired connectivity effect is observed. As it is well-known, the use of TV as a regulariser in inpainting problems forces the minimisation of the length of level lines causing artefacs, a property that can be seen through the use of the standard coarea formula \cite{schoenlieb}.

\begin{figure}[t]
    \centering
    \begin{subfigure}[b]{0.2\textwidth}
        \centering
        \includegraphics[width=\textwidth]{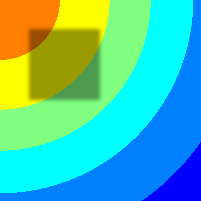}
        \caption{Input $f=v$}
        \label{fig:sr_draw_1}
    \end{subfigure}
    \begin{subfigure}[b]{0.2\textwidth}
        \centering
        \includegraphics[width=\textwidth]{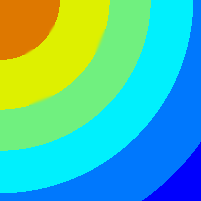} 
        \caption{Output}
        \label{fig:sr_draw_2}
    \end{subfigure}
        \begin{subfigure}[b]{0.2\textwidth}
        \centering
        \includegraphics[width=\textwidth]{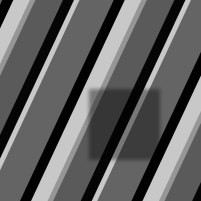}
        \caption{Input $f=v$}
        \label{fig:sr_draw_3}
    \end{subfigure}
        \begin{subfigure}[b]{0.2\textwidth}
        \centering
        \includegraphics[width=\textwidth]{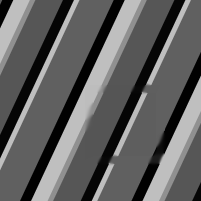} 
        \caption{Output}
        \label{fig:sr_draw_4}
    \end{subfigure}
    \caption{Drawbacks of shadow removal on images with soft shadows, using a 6 pixels wide mask.}
    \label{fig:sr_draw}
\end{figure}

A possible remedy for these two drawbacks consists in employing a different diffusivity function $g$.
As we remarked in Corollary \ref{cor:th}, a slight variation of the non-linear model depending on a  diffusivity function $g_p$ depending on a constant parameter $p\in[1,2)$ as in \eqref{eq:gp} can be considered. The use of higher values of $p$ helps to better identify the regions of the shadow boundary that are crossed by the underlying edges of the image. In these regions, since $g_p(u)=1/|\nabla u|^{p}$, we have the effect of slowing down the diffusion,  thus favouring edge preservation. 

\begin{figure}[b!]
    \centering
    \begin{tabular}{cccc}
    \includegraphics[height=0.2\linewidth]{figs/17_soft.png} &
    \includegraphics[height=0.2\linewidth]{figs/16_soft.png}&
    \includegraphics[height=0.2\linewidth]{figs/44_2.png}
    &
    \includegraphics[height=0.2\linewidth]{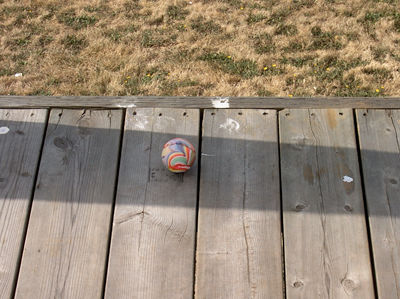}
    \\
    \includegraphics[height=0.2\linewidth]{figs/17_soft_divmatr_tau1000T100010epsilon1e-07.png} &
    \includegraphics[height=0.2\linewidth]{figs/16_soft_divmatr_tau1000T100008epsilon1e-07.png} &
    \includegraphics[height=0.2\linewidth]{figs/44_2_divmatr_tau1000T100006epsilon1e-07.png} &
    \includegraphics[height=0.2\linewidth]{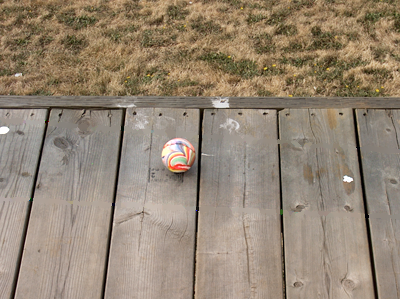}
    \\
    \includegraphics[height=0.2\linewidth]{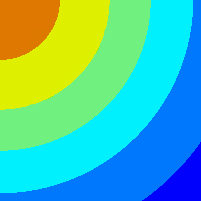} &
    \includegraphics[height=0.2\linewidth]{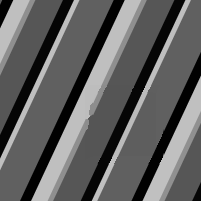} &
    \includegraphics[height=0.2\linewidth]{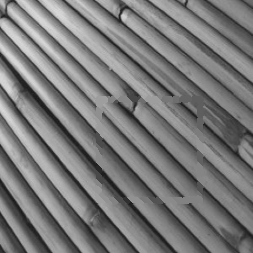} &
    \includegraphics[height=0.2\linewidth]{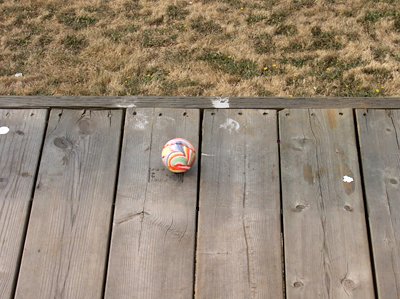}
    \end{tabular}
    \caption{Results obtained using the diffusivity function $g_p$ with $p=1$ (second row) and $p=1.9$ (third row).}
    \label{fig:otherp}
\end{figure}

In Figure \ref{fig:otherp} we show some results with $p=1$ (second row) and $p=1.9$ (third row) on the images of Fig. \ref{fig:sr_draw} and on two other examples. The use of a possible space-variant strategy adapting the value of $p$ to the local image content would mitigate the previously observed drawbacks and, as such, is an interesting direction of future research.

%\medskip

In the following, we compare qualitatively and quantitatively our results with the ones obtained by means of the standard linear osmosis model \eqref{eq:linear} and its anisotropic variant \eqref{eq:aniso} proposed in \cite{Parisotto2019AnisotropicOF} to limit the smoothing effect due to Laplace diffusion on the shadow boundary. 

In Fig. \ref{fig:compare_hard}, some examples of hard shadow removal are shown.
We notice that the linear model is effective in removing the shadow, but produces visible blurring on the shadow boundary. The anisotropic model is able to connect well the underlying features of the image, but does not well balance the image intensity between the inside and the outside of the shadow. Our model works well for both tasks, achieving very good SSIM values in comparison with the original ground truth images.

\begin{figure}[h!]
    \centering
    \begin{tabular}{c|ccc}
        Input $f=v$ & Linear \cite{weickert-ssvm} & Anisotropic \cite{Parisotto2019AnisotropicOF} & Non-linear  \\
        \includegraphics[width=0.2\textwidth]{figs/17.png} &
        \includegraphics[width=0.2\textwidth]{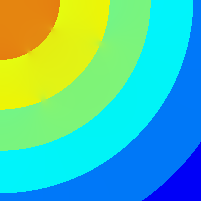} &
        \includegraphics[width=0.2\textwidth]{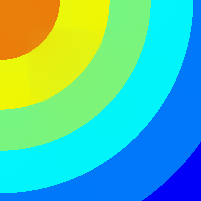} &
        \includegraphics[width=0.2\textwidth]{figs/17_divmatr_tau1000T100002epsilon1e-07.png}
        \\
        (a) & 0.9943 & 0.9966 & \textbf{0.9979}
        \\
        \includegraphics[width=0.2\textwidth]{figs/44_0.png} &
        \includegraphics[width=0.2\textwidth]{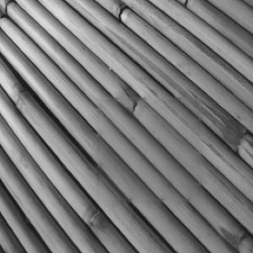} &
        \includegraphics[width=0.2\textwidth]{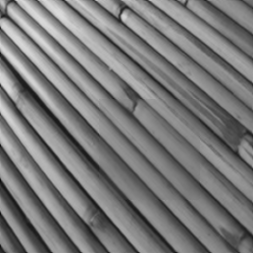} &
        \includegraphics[width=0.2\textwidth]{figs/44_0_divmatr_tau1000T100002epsilon1e-07.png}
        \\
        (b) & 0.9624 & 0.9544 & \textbf{0.9951}
    \end{tabular}
    \caption{Comparison between linear, anisotropic and nonlinear osmosis models on hard-shadowed images, with associated SSIM values.}
    \label{fig:compare_hard}
\end{figure}

In Fig. \ref{fig:compare_soft} we compare the results obtained by the three different models on soft-shadowed images and real images.
We observe that around the mask region our model performs as good as the anisotropic one. However, while the anisotropic model tends to make the whole image slightly blurred (due to the smoothing kernels required to defined the diffusion tensor $\WW$),  the non-linear model preserves the sharpness of the image everywhere.

\begin{figure}[h!]
    \centering
    \begin{tabular}{c|ccc}
        Input $f=v$ & Linear \cite{weickert-ssvm} & Anisotropic \cite{Parisotto2019AnisotropicOF} & Non-linear \\
        \includegraphics[width=0.2\textwidth]{figs/17_soft.png} &
        \includegraphics[width=0.2\textwidth]{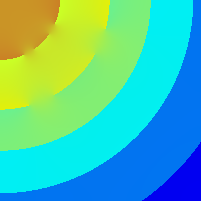} &
        \includegraphics[width=0.2\textwidth]{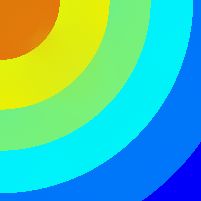} &
        \includegraphics[width=0.2\textwidth]{figs/17_soft_divmatr_tau1000T100010epsilon1e-07.png}
        \\
        (c) & 0.9712 & \textbf{0.9937} & 0.9931
        \\
        \includegraphics[width=0.2\textwidth]{figs/44_2.png} &
        \includegraphics[width=0.2\textwidth]{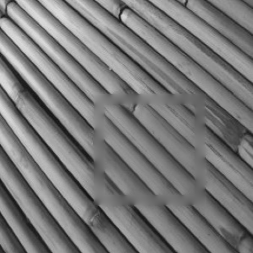} &
        \includegraphics[width=0.2\textwidth]{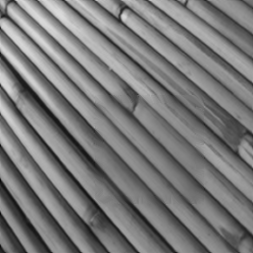} &
        \includegraphics[width=0.2\textwidth]{figs/44_2_divmatr_tau1000T100006epsilon1e-07.png}
        \\
        (d) & 0.9412 & 0.9464 & \textbf{0.9538}
        \\
        \hline & & &
        \\
        \includegraphics[width=0.2\textwidth]{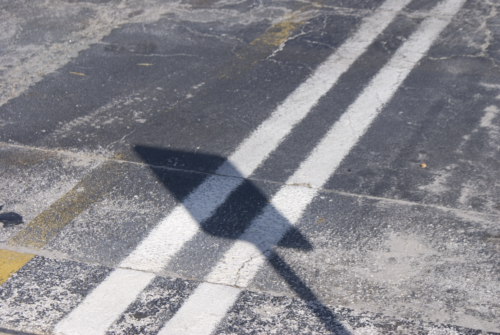} &
        \includegraphics[width=0.2\textwidth]{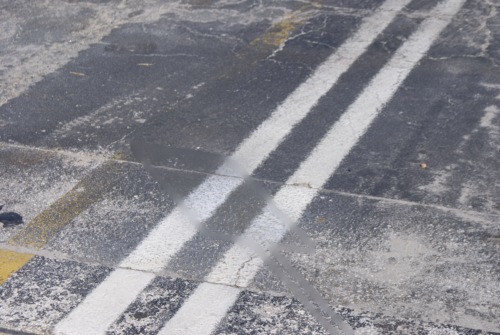} &
        \includegraphics[width=0.2\textwidth]{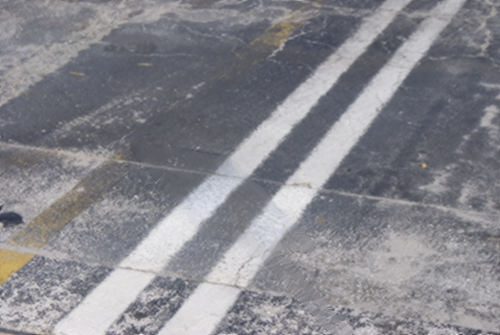} &
        \includegraphics[width=0.2\textwidth]{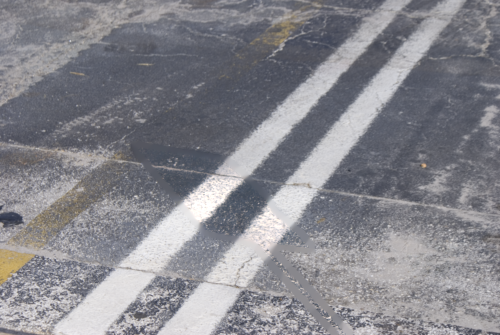}
        \\
        (e) & & &
        \\
        \includegraphics[width=0.2\textwidth]{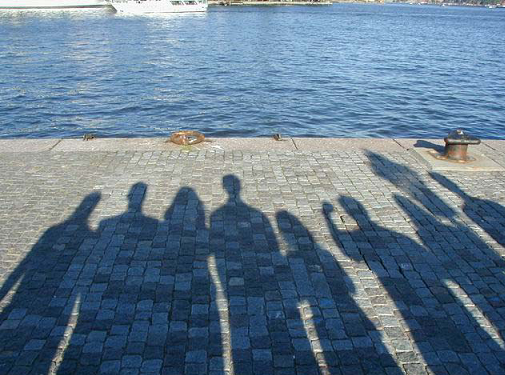} &
        \includegraphics[width=0.2\textwidth]{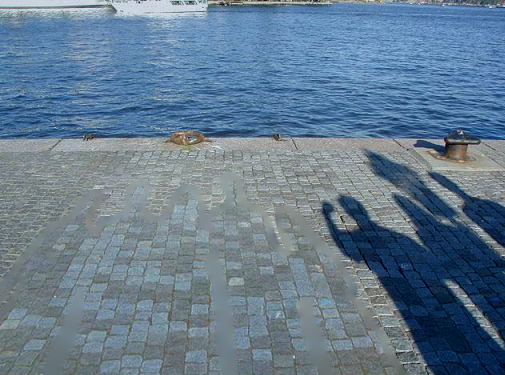} &
        \includegraphics[width=0.2\textwidth]{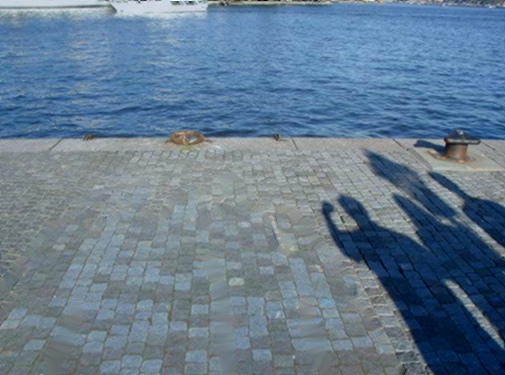} &
        \includegraphics[width=0.2\textwidth]{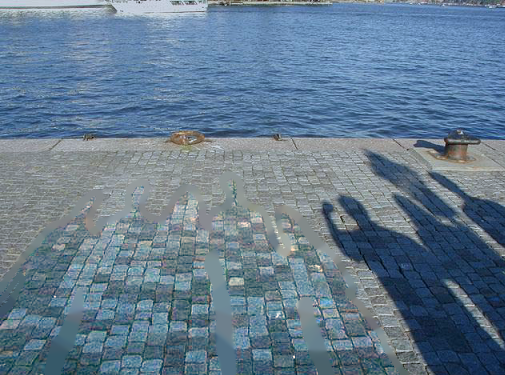}
        \\
        (f) & & &
        \\
        \includegraphics[width=0.2\textwidth]{figs/image.png} &
        \includegraphics[width=0.2\textwidth]{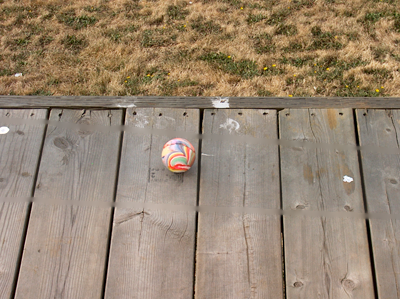} &
        \includegraphics[width=0.2\textwidth]{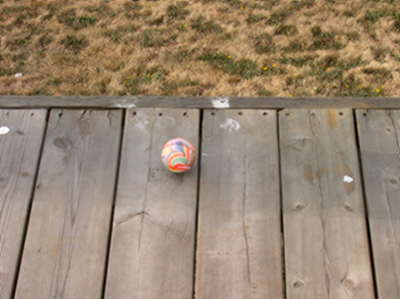} &
        \includegraphics[width=0.2\textwidth]{figs/image_divmatr_tau13000T100003epsilon1e-07.png}
        \\
        (g) & & &
    \end{tabular}
    \caption{Results of osmosis models on artificially soft-shadowed images and natural images, with associated SSIM when computable.}
    \label{fig:compare_soft}
\end{figure}

\medskip 

We now compare the computational efficiency required to compute the numerical solution of the non-linear model by using the semi-implicit scheme \eqref{eq:semi-implicit} with the linear \cite{weickert-ssvm} (solved with a fully implicit scheme with $\tau=10^3$), and the anisotropic model \cite{Parisotto2019AnisotropicOF} (solved by exponential integration).
\begin{table}[h!]
\centering
\begin{tabular}{c | r r | r r | r r}
     & \multicolumn{2}{c}{Linear \cite{weickert-ssvm}} & \multicolumn{2}{|c}{Anisotropic \cite{Parisotto2019AnisotropicOF}} & \multicolumn{2}{|c}{Non-linear} \\
    \hline (a) & 3.98 & (7) & 27.39 & (4) & \textbf{1.25} & \textbf{(3)}
    \\
    \hline (b) & 6.44 & (19) & 29.71 & (5) & \textbf{1.45} & \textbf{(4)} 
    \\
    \hline (c) & 11.54 & (20) & 20.77 & \textbf{(2)} & \textbf{1.80} & (3)
    \\
    \hline (d) & 8.66 & (26) & 18.70 & \textbf{(3)} & \textbf{3.53} & (8) 
    \\
    \hline (e) & 31.63 & (9) & 1307.06 & (10) & \textbf{10.37} & \textbf{(3)} \\
    \hline (f) & 105.55 & (30) &  2592.00 & (18) & \textbf{27.08} & \textbf{(7)} \\
    \hline (g) & 19.14 & (9) & 723.48 & (10) & \textbf{8.62} & \textbf{(4)} 
\end{tabular}
\caption{Execution time (in seconds) and number of iterations till convergence (in brackets) for results in Fig. \ref{fig:compare_hard} (a,b) and Fig. \ref{fig:compare_soft} (c,d,e,f,g). 
}
\label{tab:it_time}
\end{table}
In this respect, we report in Table \ref{tab:it_time}  the execution time (in seconds) and the number of iterations (in brackets) for the three models  when applied to the images of Figure \ref{fig:compare_hard} and \ref{fig:compare_soft}, using as same stopping criteria a mean squared error w.r.t. to the ground truth smaller then a given tolerance for artificially corrupted images (a,b,c,d), and a sufficiently small ($10^{-3}$) relative change for real-world images (e,f,g).
From a visual inspection of Fig. \ref{fig:compare_soft}, we can observe that the results obtained by the anisotropic model are qualitatively comparable with the non-linear ones, although they require a much more significant computational effort, see Table \ref{tab:it_time}. This is due to the more sophisticated nature of the anisotropic model \eqref{eq:aniso} which requires the computation of the matrix field $\WW$ and its possible update at every iteration. 
Interestingly, the proposed non-linear model appears faster than the linear model in reaching convergence. This is due to the action of the diffusivity term $g$, which accelerates the osmosis process on $\Omega_{in}\cup\Omega_{out}$.

Finally, we provide an extensive list of comparisons with alternative approaches proposed for shadow removal Illuminant invariant approaches are used, e.g., in
\cite{Finlayson1,Finlayson2}, where an "intrinsic" image is extracted from the input image and then shadow removal is treated as reintegration problem, using this image as a guide to derive gradient information or similarity measures between shadowed and lighted regions. A variational approach to shadow and spot-light removal has been proposed in \cite{HKLM21}, in the context of image decomposition. Colour transfer techniques, such as \cite{Liu,Shor}, use gradient information to recover the shadow-free image by solving a Poisson equation.
Other methods consists in scaling the shadow region by a suitable factor, modelling the shadow as uniform \cite{Finlayson3} or even non-uniform \cite{Arbel}.
Another class of solvers treat the task as a matting problem, considering the shadow pixels as foreground and the lit pixels as background \cite{Wu1,Wu2}.

\begin{figure}[h!]
    \centering
    \begin{tabular}{c|c|cccc}
        \includegraphics[width=0.15\textwidth]{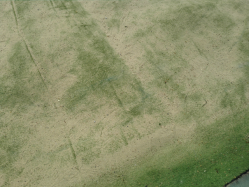} &
        \includegraphics[width=0.15\textwidth]{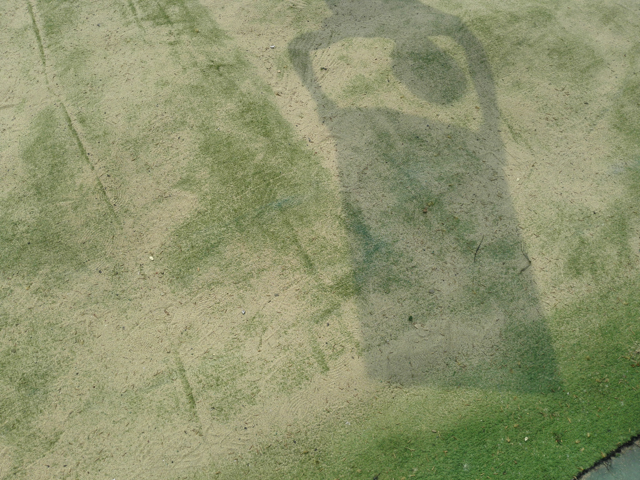} &
        \includegraphics[width=0.15\textwidth]{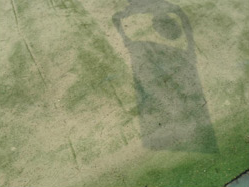} &
        \includegraphics[width=0.15\textwidth]{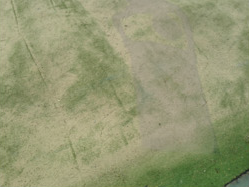} &
        \includegraphics[width=0.15\textwidth]{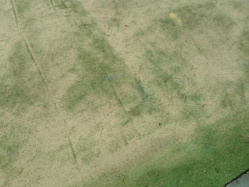} &
        \includegraphics[width=0.15\textwidth]{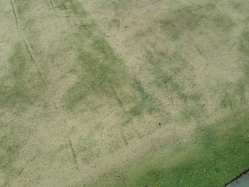}
        \\ 
         & & 0.9172 & 0.9603 & 0.9667 & \textbf{0.9676}
        \\
        \includegraphics[width=0.15\textwidth]{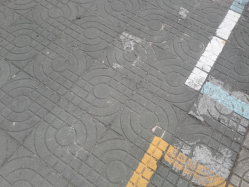} &
        \includegraphics[width=0.15\textwidth]{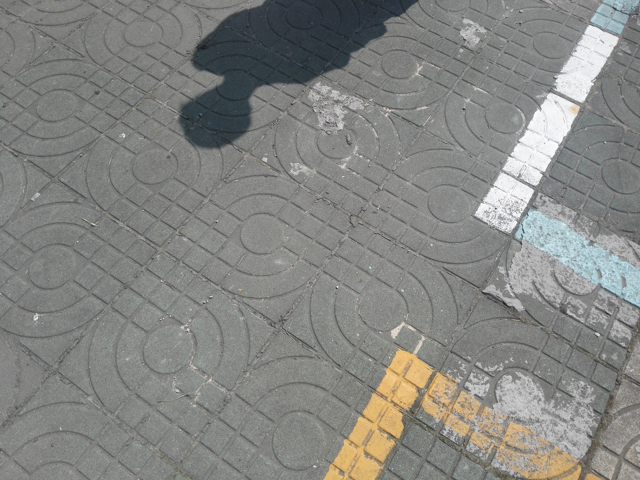} &
        \includegraphics[width=0.15\textwidth]{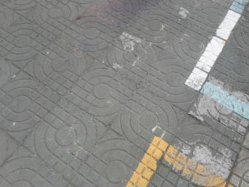} &
        \includegraphics[width=0.15\textwidth]{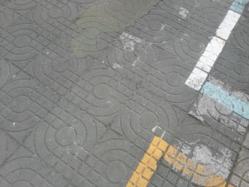} &
        \includegraphics[width=0.15\textwidth]{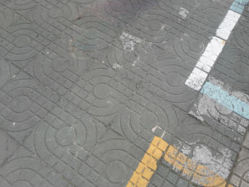} &
        \includegraphics[width=0.15\textwidth]{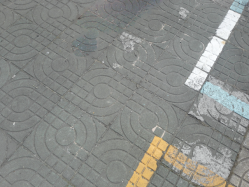}
        \\
         & & 0.7816 & 0.9482 & 0.9557 & \textbf{0.9726}
        \\
        \includegraphics[width=0.15\textwidth]{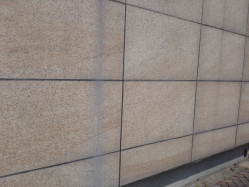} &
        \includegraphics[width=0.15\textwidth]{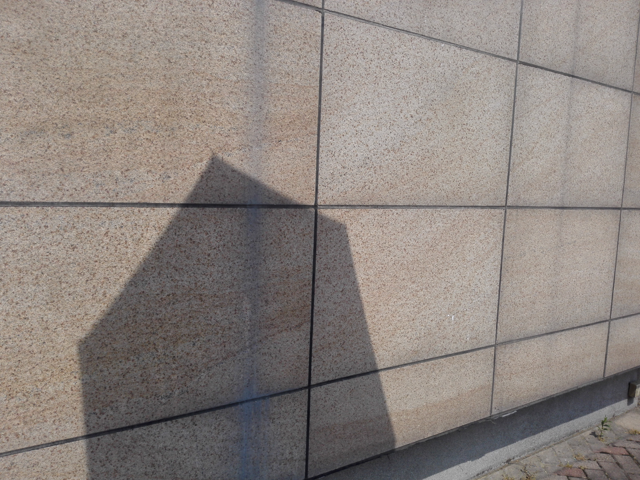} &
        \includegraphics[width=0.15\textwidth]{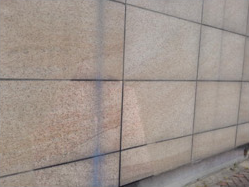} &
        \includegraphics[width=0.15\textwidth]{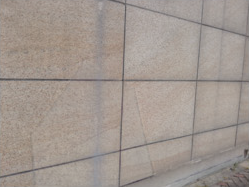} &
        \includegraphics[width=0.15\textwidth]{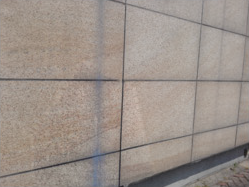} &
        \includegraphics[width=0.15\textwidth]{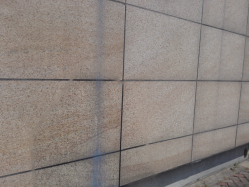}
        \\
         & & 0.8288 & 0.9092 & 0.9350 & \textbf{0.9488}
        \\
        \includegraphics[width=0.15\textwidth]{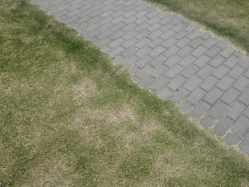} &
        \includegraphics[width=0.15\textwidth]{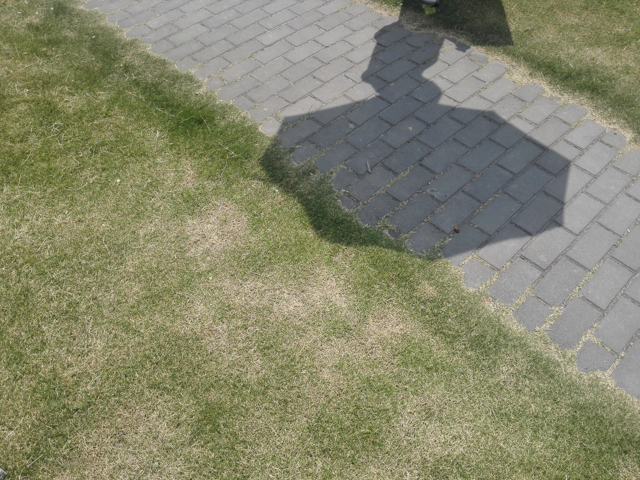} &
        \includegraphics[width=0.15\textwidth]{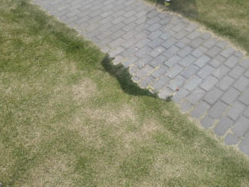} &
        \includegraphics[width=0.15\textwidth]{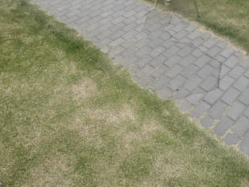} &
        \includegraphics[width=0.15\textwidth]{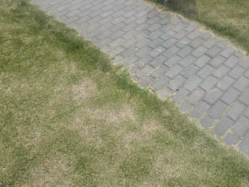} &
        \includegraphics[width=0.15\textwidth]{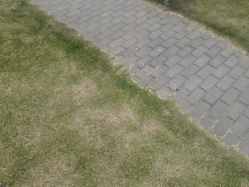}
        \\
         & & 0.8335 & 0.9312 & 0.9265 & \textbf{0.9650}
        \\
        \includegraphics[width=0.15\textwidth]{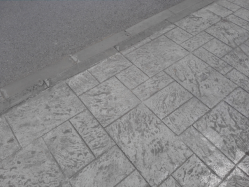} &
        \includegraphics[width=0.15\textwidth]{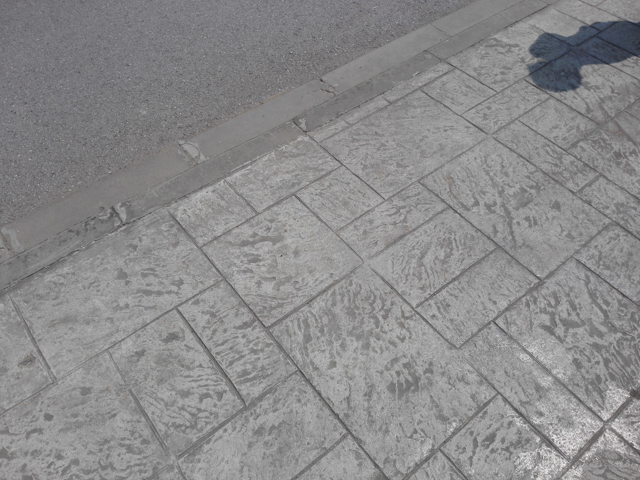} &
        \includegraphics[width=0.15\textwidth]{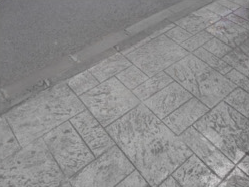} &
        \includegraphics[width=0.15\textwidth]{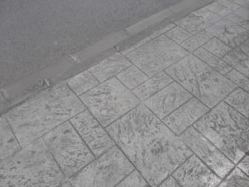} &
        \includegraphics[width=0.15\textwidth]{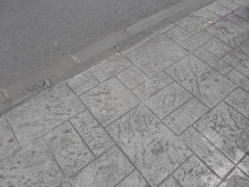} &
        \includegraphics[width=0.15\textwidth]{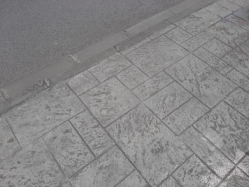}
        \\
         & & 0.8106 & 0.9500 & 0.9531 & \textbf{0.9733}
        \\
        \includegraphics[width=0.15\textwidth]{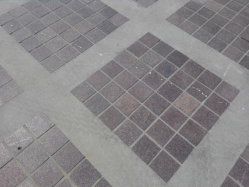} &
        \includegraphics[width=0.15\textwidth]{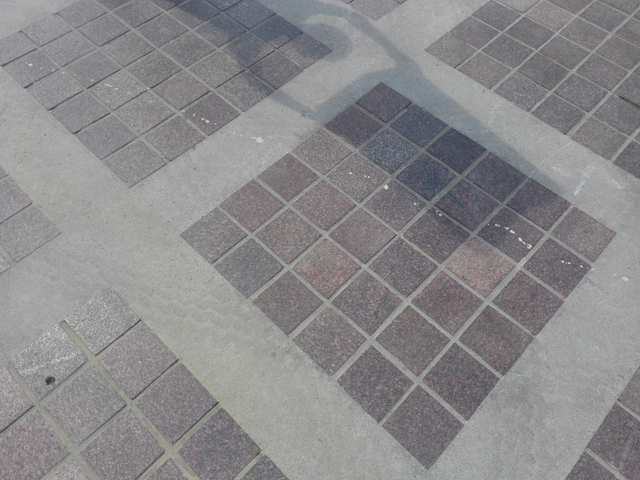} &
        \includegraphics[width=0.15\textwidth]{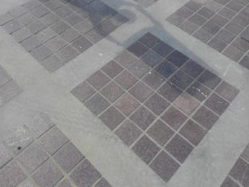} &
        \includegraphics[width=0.15\textwidth]{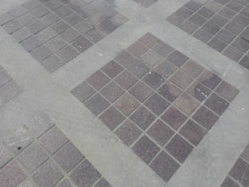} &
        \includegraphics[width=0.15\textwidth]{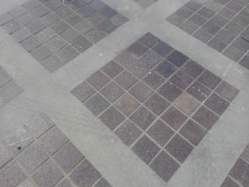} &
        \includegraphics[width=0.15\textwidth]{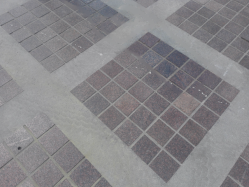}
        \\
         & & 0.6657 & 0.8174 & 0.8134 & \textbf{0.9008}
         \\
         \hline Original & Input $f=v$ & \cite{Guo} & \cite{Le} & \cite{Benalia} & ours
        \end{tabular}
    \caption{Comparison between shadow removal approaches, with associated SSIM values.}
    \label{fig:compare_others}
\end{figure}

In Fig. \ref{fig:compare_others} we 
report a comparison between the results obtained by the proposed non-linear osmosis model and three of the most recent and successful methods for shadow removal\footnote{The test images have been downloaded from the ISTD dataset, available at \url{https://github.com/DeepInsight-PCALab/ST-CGAN}.}: a deep learning approach, \cite{Le}, a region-based shadow detection and removal \cite{Guo} and, finally, a variational approach based on a non-local version of the osmosis model, \cite{Benalia} followed by a further contrast correction step as described by the authors.
We observe that our method produces visually pleasant results compared to state-of-the-art methods and more sophisticated variants of the standard osmosis model \cite{Benalia} further requiring a contrast correction step. Quantitatively, it outperforms all other approaches in terms of SSIM values.

\subsection{Spot-light removal}\label{sec:lirem}

Recalling Section \ref{sec:appl}, an analogous use of the osmosis model \eqref{eq:sr_model} can be done  for the spot-light removal problem. Here, a synthetic spot-lighted image $f$ is  obtained by Hadamard product $f \odot (G_{\sigma} * \ell)$ between a light- spot-free image $f^*$ and a light image (with bounded lighted region $L\subset \Omega$ ) $\ell$ defined as
\begin{equation}\label{eq:l_light}
    \ell =
    \begin{cases}
    c \quad &on \quad L \\
    1 & on \quad \Omega\setminus L,
    \end{cases}
\end{equation}
with $c>1$. Positive values of the standard deviation $\sigma$ produce soft-lighted images, which require a larger mask. Fig. \ref{fig:lr2} shows the results for four increasing levels of $\sigma$. Analogous comments to the ones made in the case of shadow removal applications can be done in this case, too.

\begin{figure}[h]
    \centering
    \begin{tabular}{ccc|cc}
    $\sigma = 0$ &
    \includegraphics[height=0.18\linewidth]{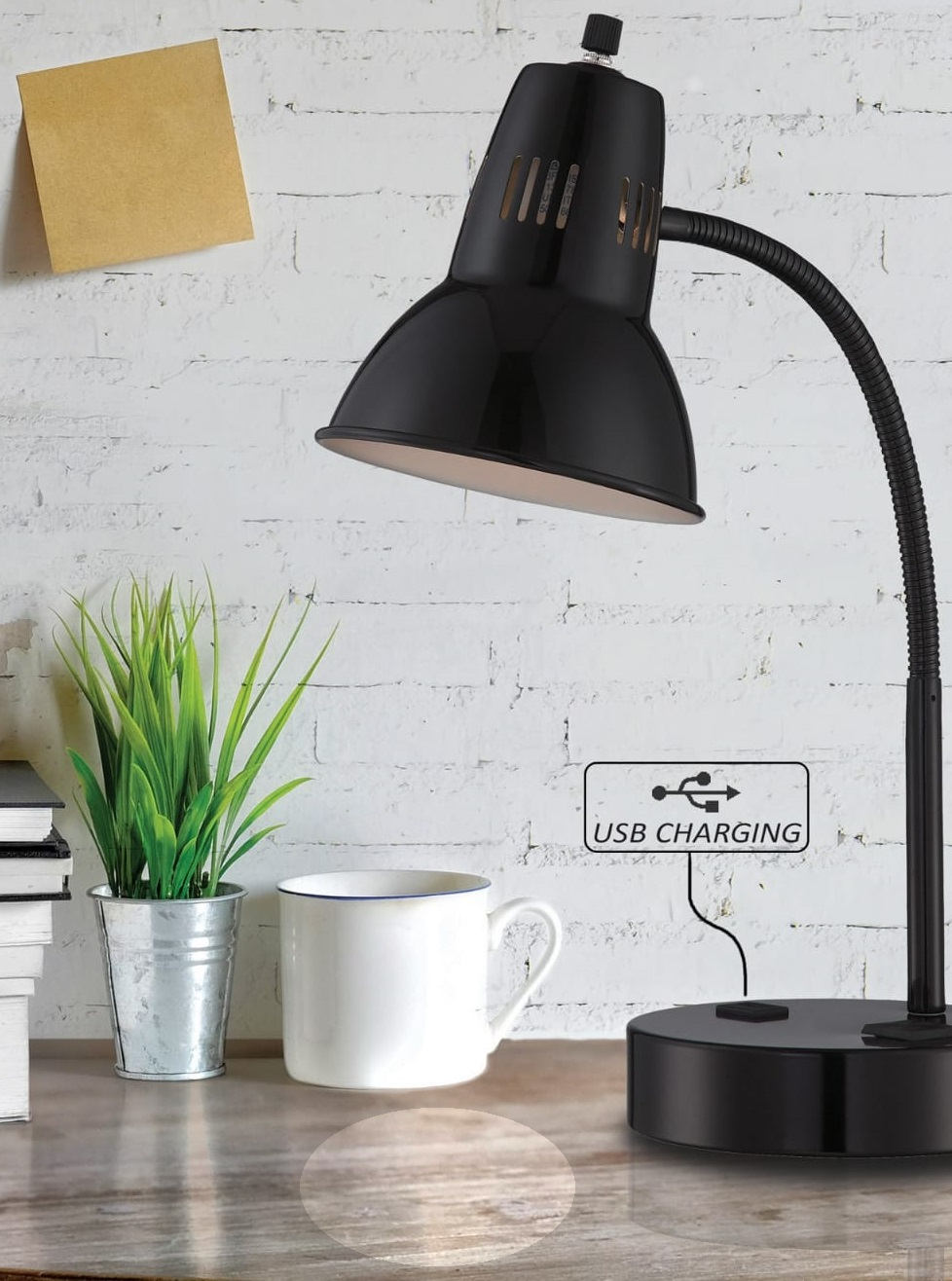}    &
    \includegraphics[height=0.18\linewidth]{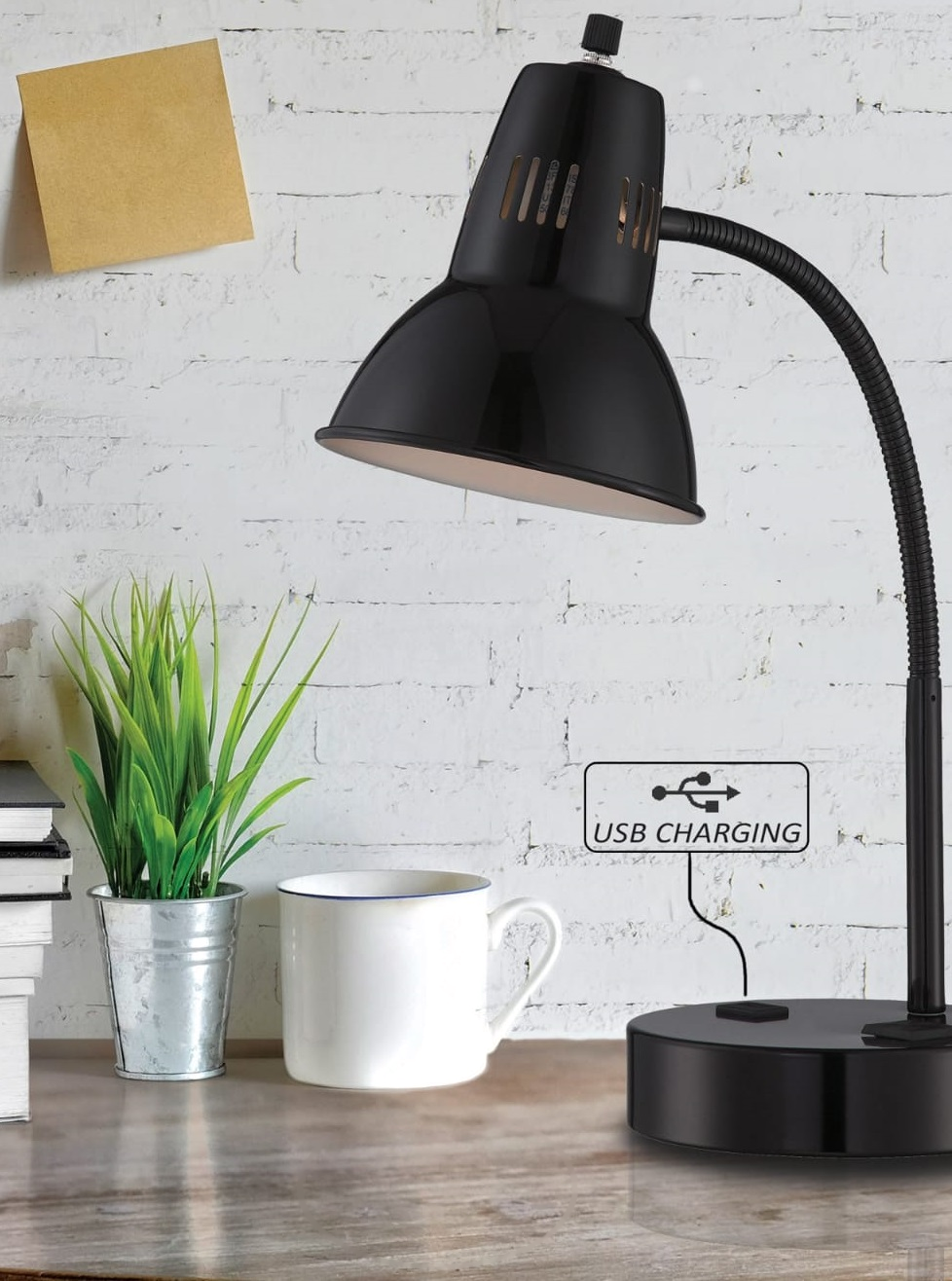}  &
    \includegraphics[height=0.18\linewidth]{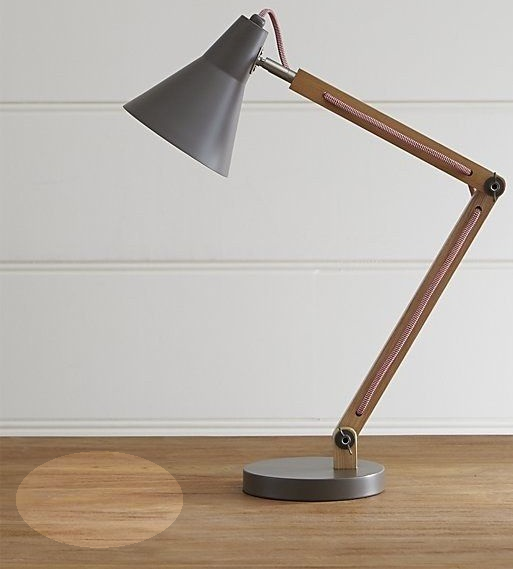} &
    \includegraphics[height=0.18\linewidth]{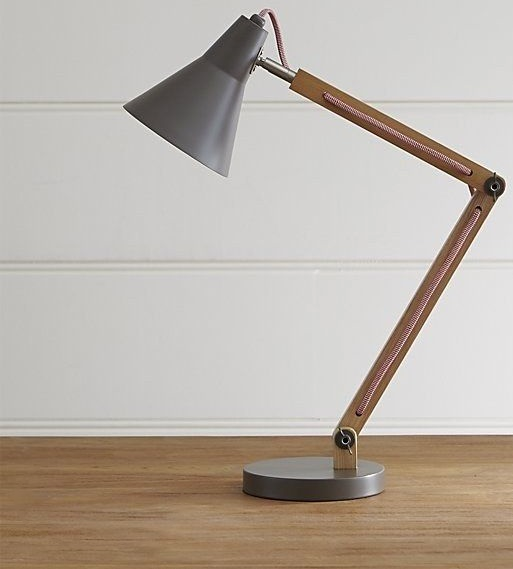}
    \\
    $\sigma = 2$ &
    \includegraphics[height=0.18\linewidth]{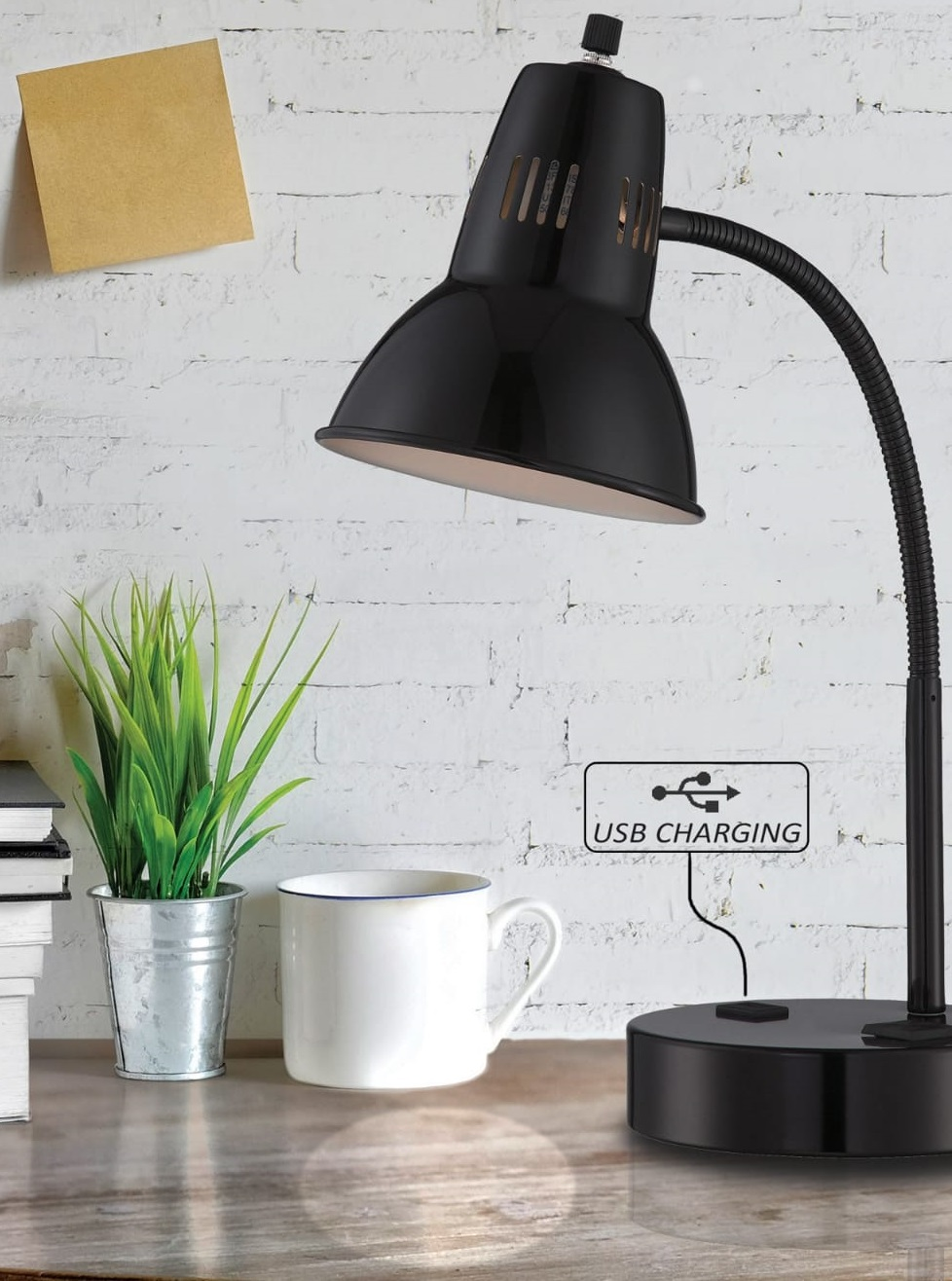}  &
    \includegraphics[height=0.18\linewidth]{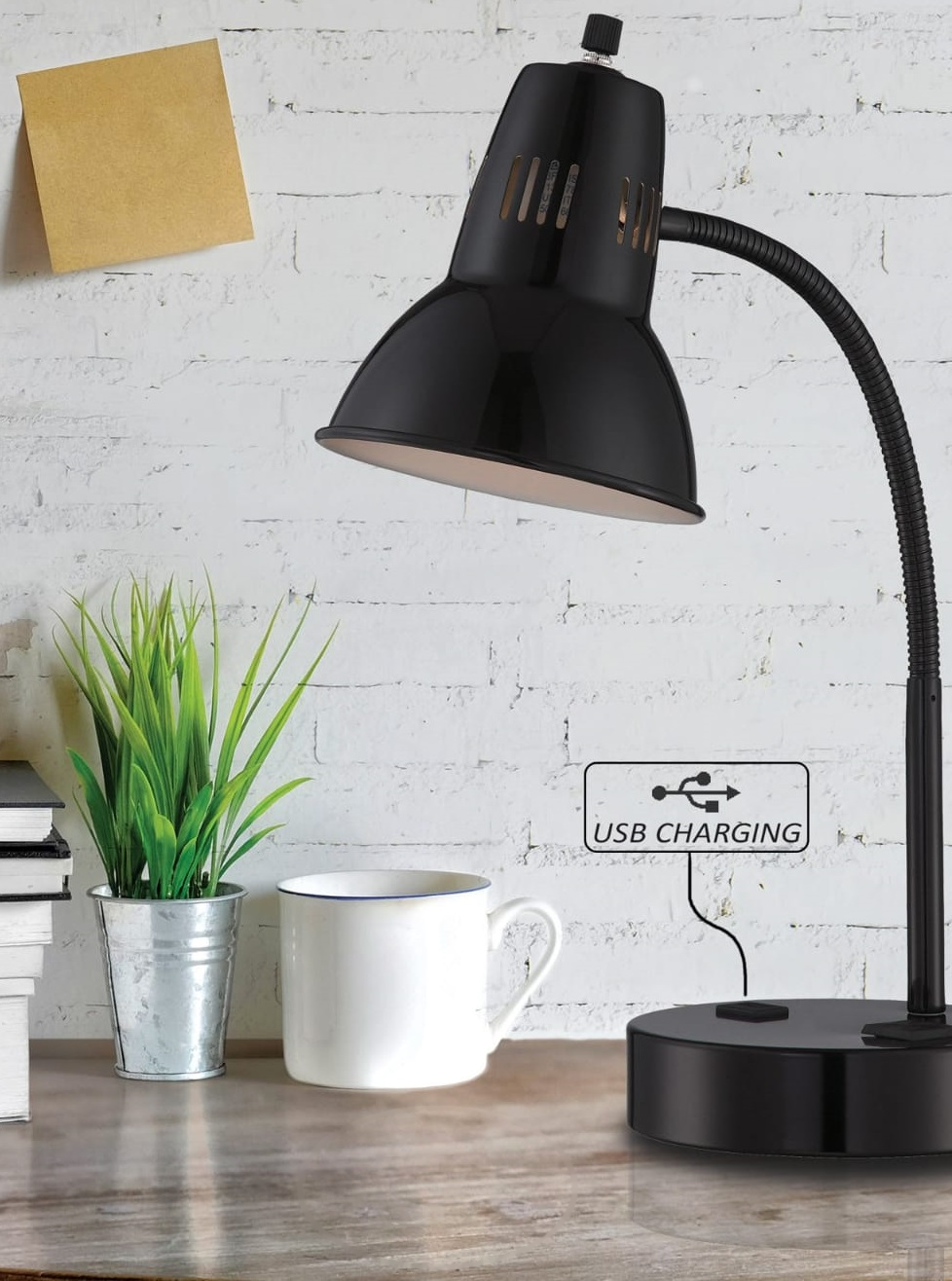}  &
    \includegraphics[height=0.18\linewidth]{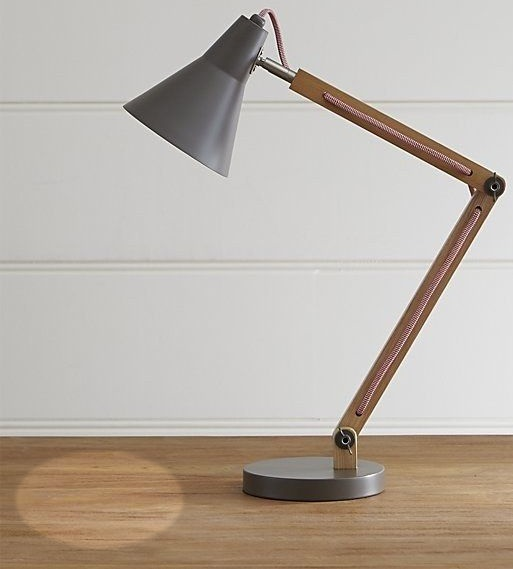} &
    \includegraphics[height=0.18\linewidth]{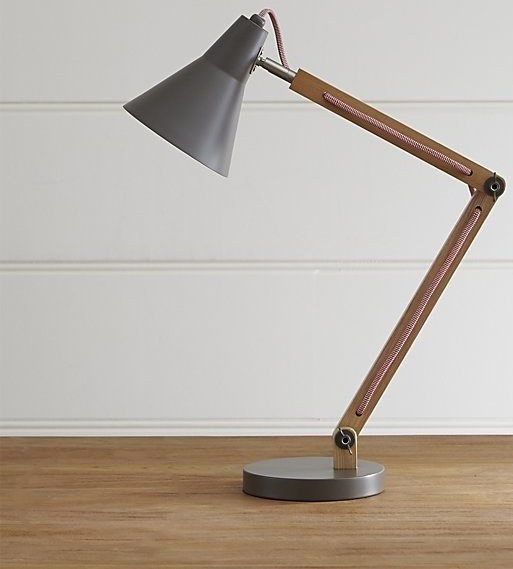}
    \\
    $\sigma = 5$ &
    \includegraphics[height=0.18\linewidth]{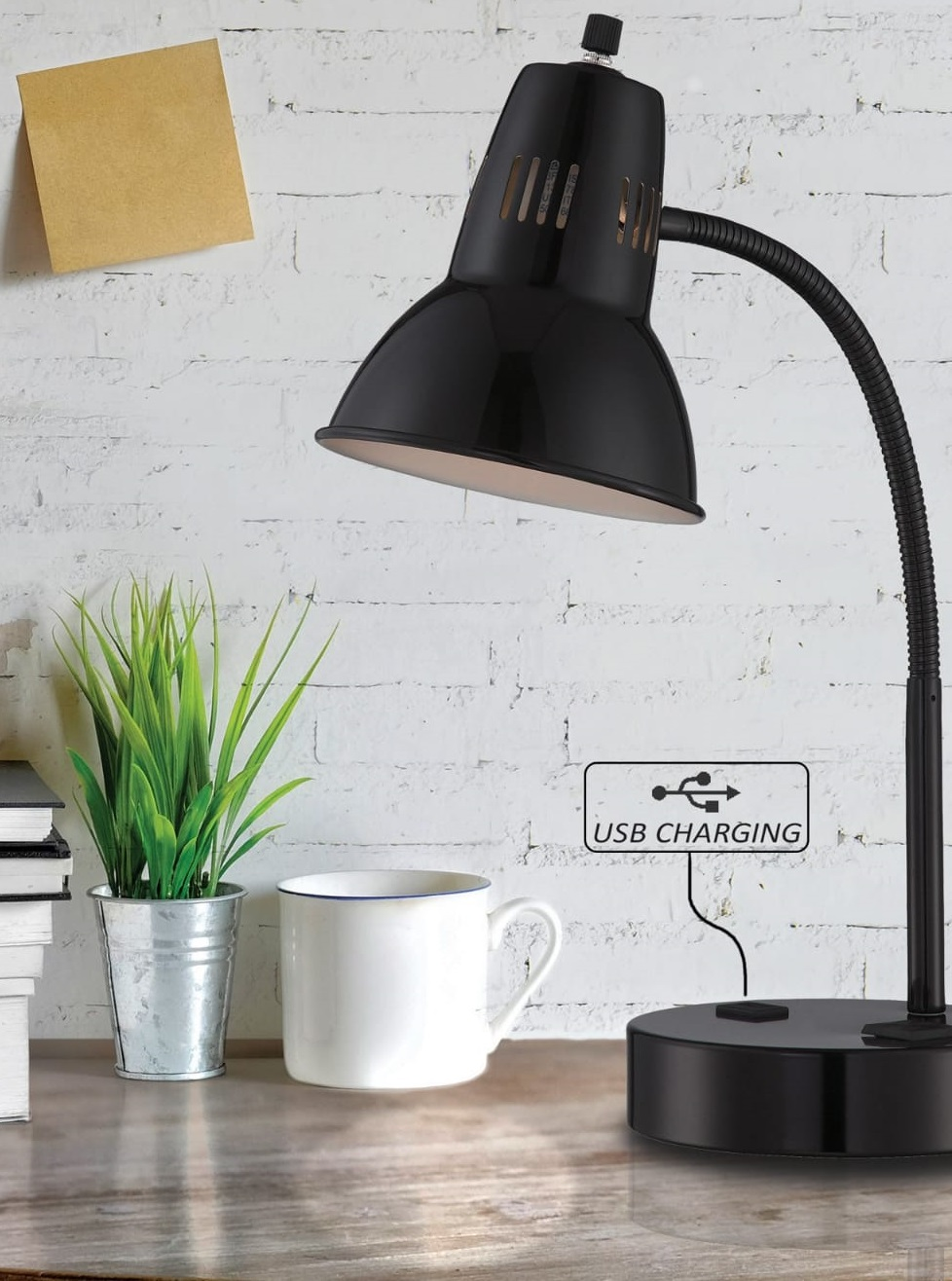}  &
    \includegraphics[height=0.18\linewidth]{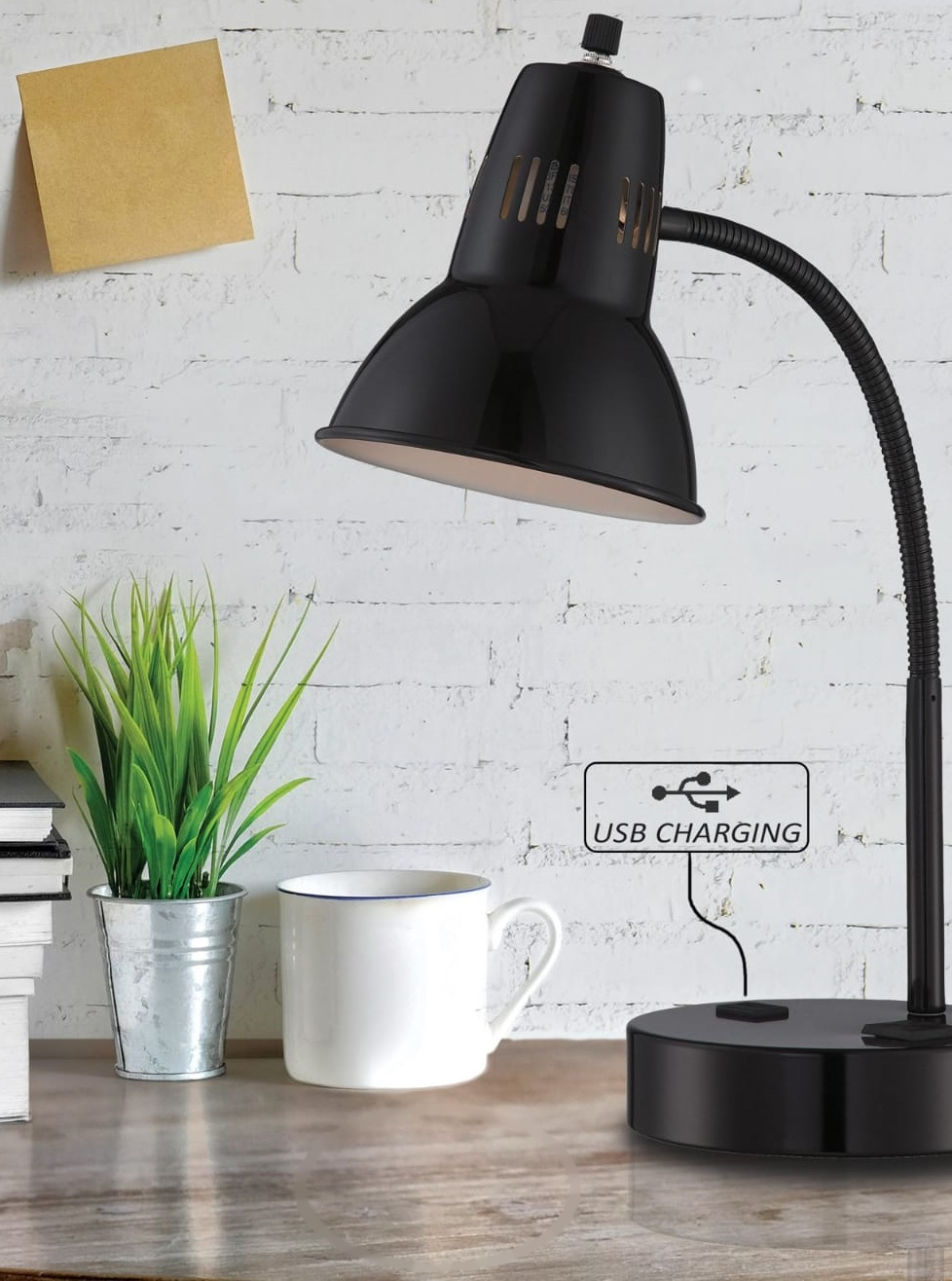}  &
    \includegraphics[height=0.18\linewidth]{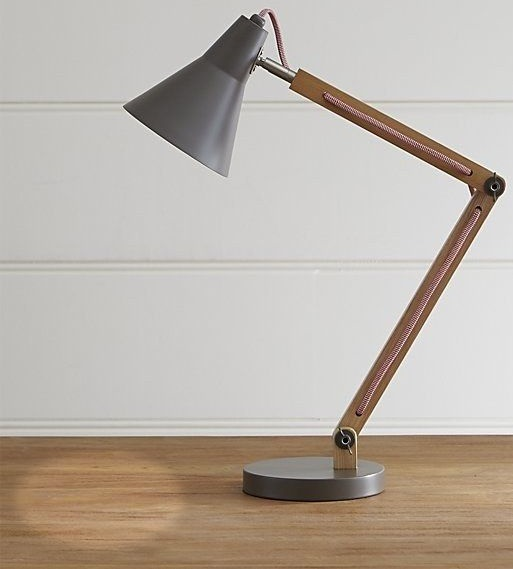} &
    \includegraphics[height=0.18\linewidth]{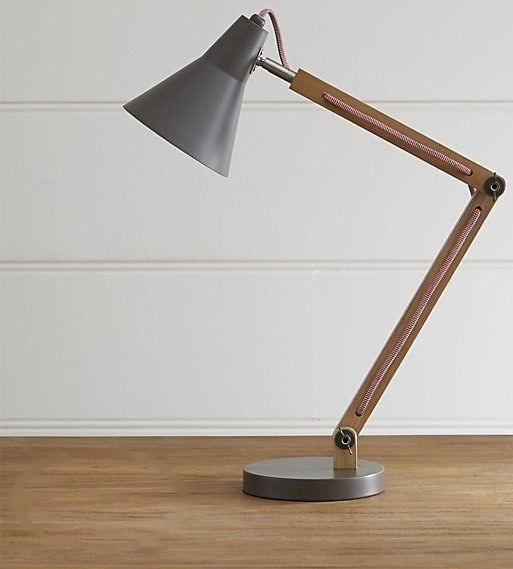}
    \\
    $\sigma = 8$ &
    \includegraphics[height=0.18\linewidth]{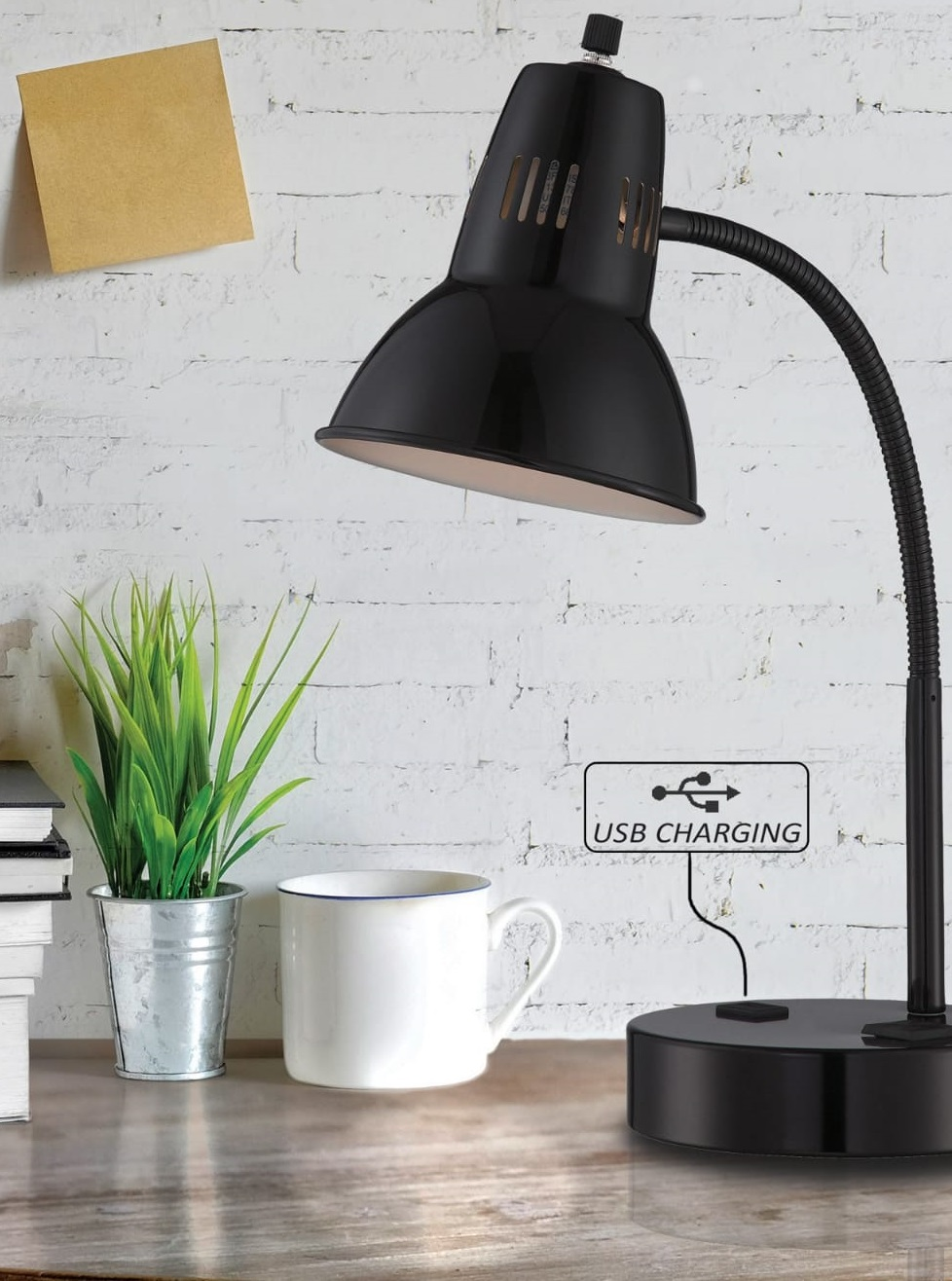}  &
    \includegraphics[height=0.18\linewidth]{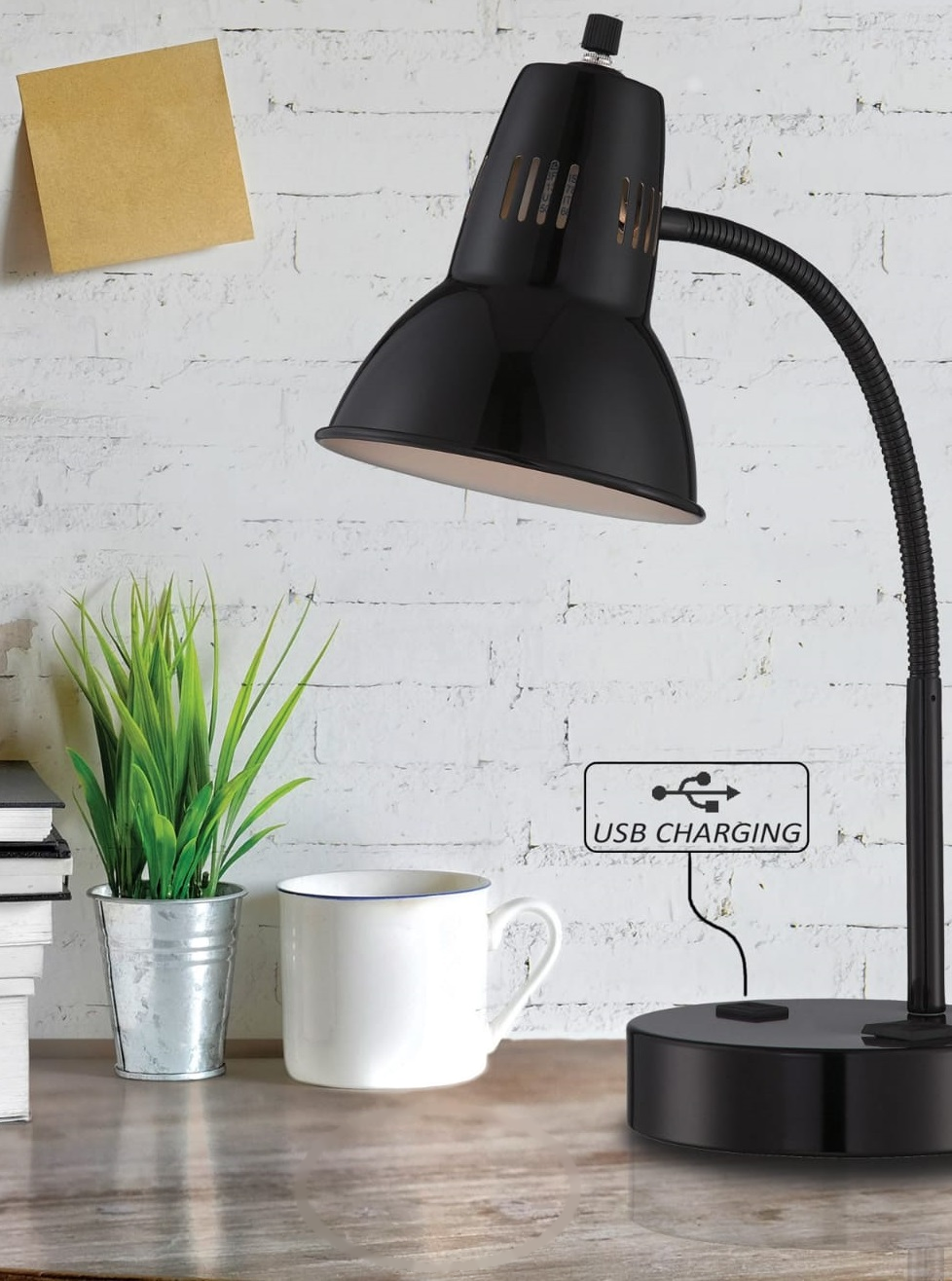}  &
    \includegraphics[height=0.18\linewidth]{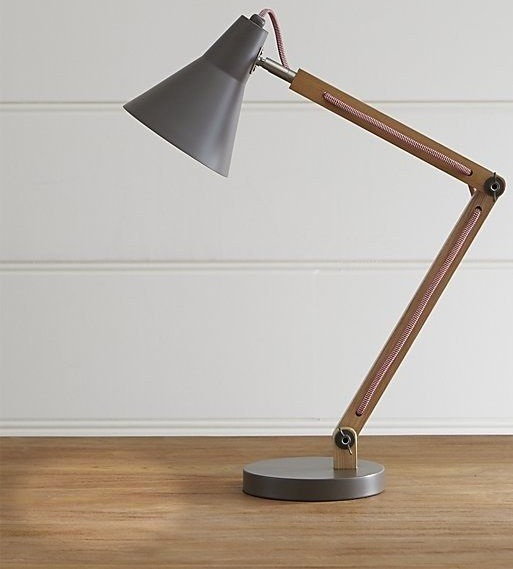} &
    \includegraphics[height=0.18\linewidth]{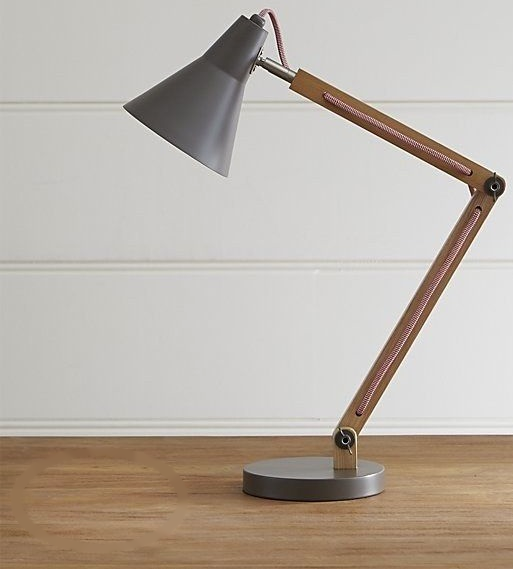}\\
    & (a) Input & (b) Results & (c) Input & (d) Results 
    \end{tabular}
    \caption{Spot-light removal results. Input images were obtained by convolving the light-free image with a Gaussian kernel with variance $\sigma$, whose values are reported on the left column.}
    \label{fig:lr2}
\end{figure}

\subsection{Compact data representation}\label{sec:cdr}

As recalled in Section \ref{sec:cdr}, image osmosis can also be applied to the problem of compact data representation, by means of a proper definition of the drift term.
Given an image $v$ with average grey value $\mu_v$, we thus consider $M=\Omega_b$ to be a pre-computed edge-mask via standard segmentation algorithms (Canny, Sobel, \ldots) and define the drift term as in \eqref{eq:cdr_drift}. Explicitly, the two equations guiding the osmosis evolution in this case then read:
\begin{align} \label{eq:CDR}
    \begin{cases}
    \partial_t u = \dv\left( \frac{\nabla u}{|\nabla u|} \right) \quad &\text{on } (\Omega\setminus\Omega_b) \times [0,T]\\
    \\
    \partial_t u = \dv\left( \frac{\nabla u-\frac{\nabla v}{v}u}{\left|\nabla u-\frac{\nabla v}{v}u\right|} \right) & \text{on } \Omega_b \times [0,T]    \end{cases}
\end{align}
with initial and boundary conditions defined as in \eqref{eq:g-osmosis1}.
By setting $f\equiv\mu_v$ on $\Omega$ and making use of the actual values of the reference image $v$ only on edges to define the drift term $\dd$ on $\Omega_b$, by evolving \eqref{eq:CDR}  the information is diffused on $\Omega\setminus\Omega_b$, which can be thought of as the piecewise-constant region of the image.
The process thus converges to an approximation $w$ of the reference image $v$.

\begin{figure}[b!]
    \centering
    \begin{tabular}{cc|cc}
        \includegraphics[width=0.2\textwidth]{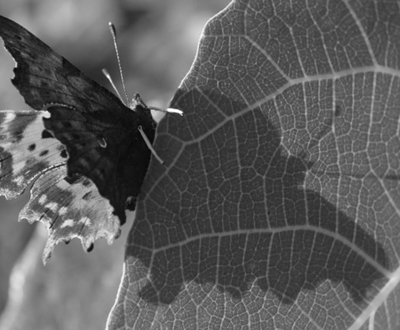} &
        \includegraphics[width=0.2\textwidth]{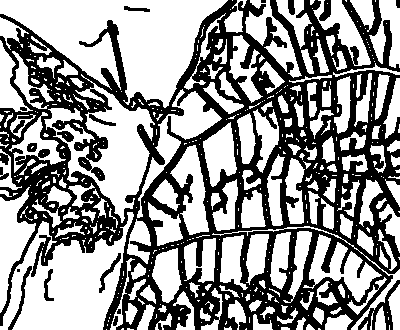} &
        \includegraphics[width=0.2\textwidth]{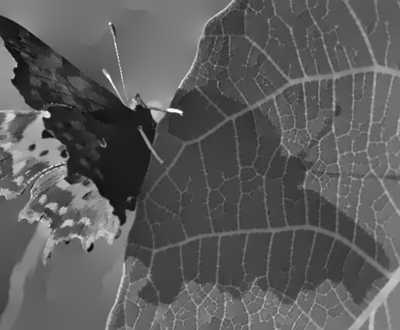} &
        \includegraphics[width=0.2\textwidth]{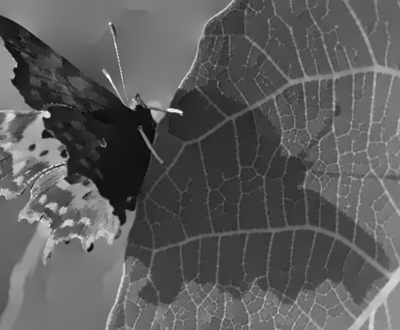} 
        \\
        & & 0.8581 & \textbf{0.8632}
        \\
        \includegraphics[width=0.2\textwidth]{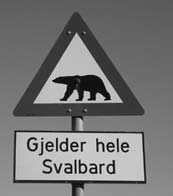} &
        \includegraphics[width=0.2\textwidth]{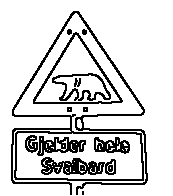} &
        \includegraphics[width=0.2\textwidth]{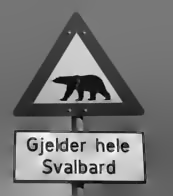} &
        \includegraphics[width=0.2\textwidth]{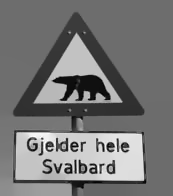} 
        \\
        & & 0.9432 & \textbf{0.9626}
        \\
        \includegraphics[width=0.2\textwidth]{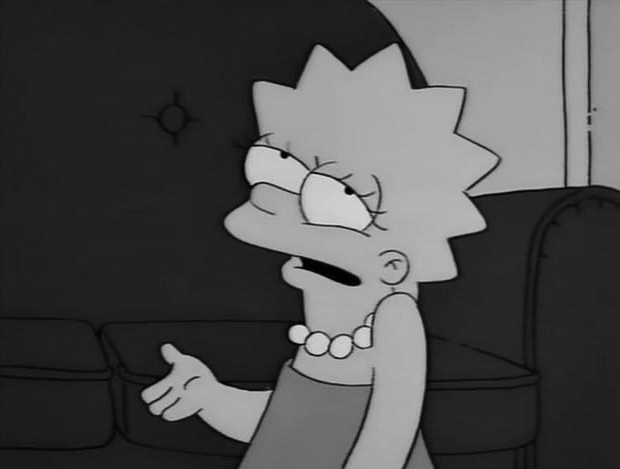} &
        \includegraphics[width=0.2\textwidth]{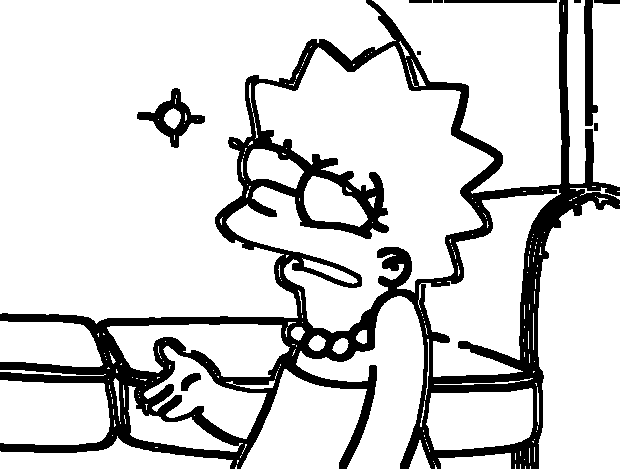} &
        \includegraphics[width=0.2\textwidth]{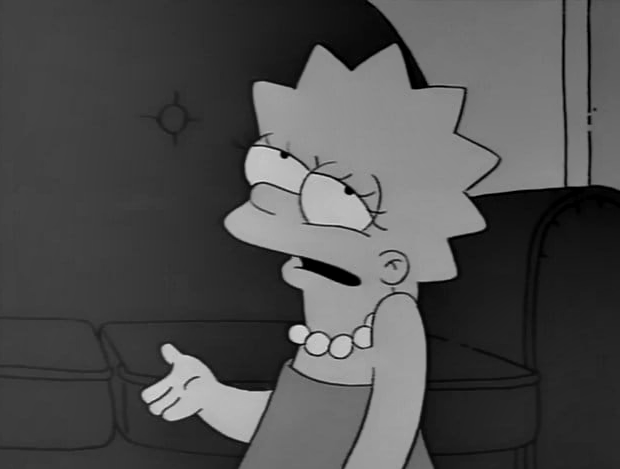} &
        \includegraphics[width=0.2\textwidth]{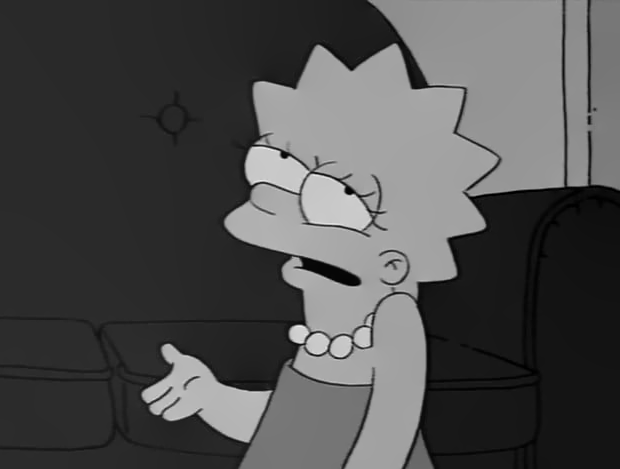} 
        \\
        & & 0.9058 & \textbf{0.9453}
        \\
        \includegraphics[width=0.2\textwidth]{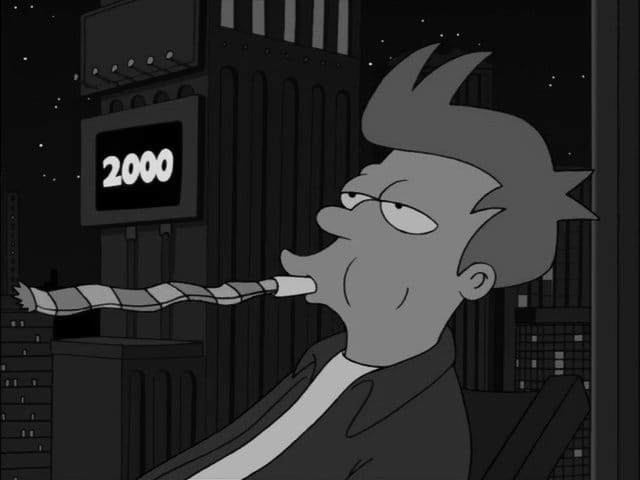} &
        \includegraphics[width=0.2\textwidth]{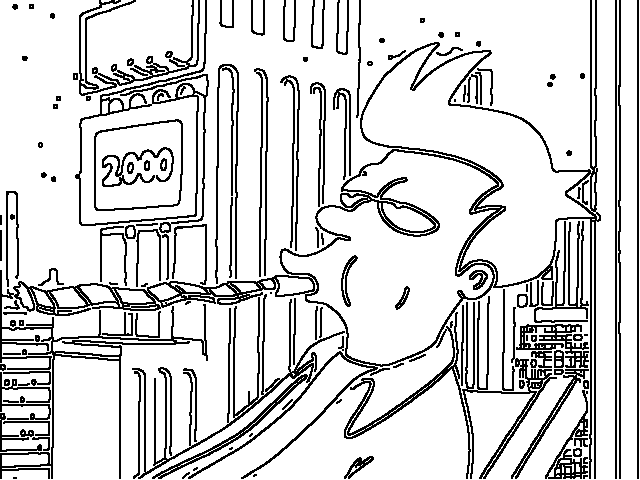} &
        \includegraphics[width=0.2\textwidth]{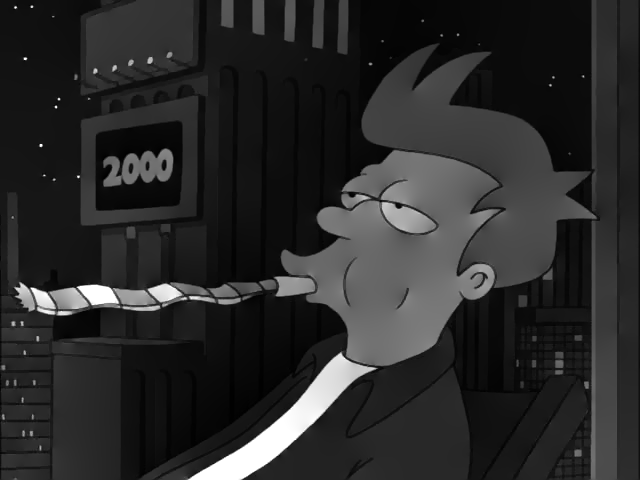} &
        \includegraphics[width=0.2\textwidth]{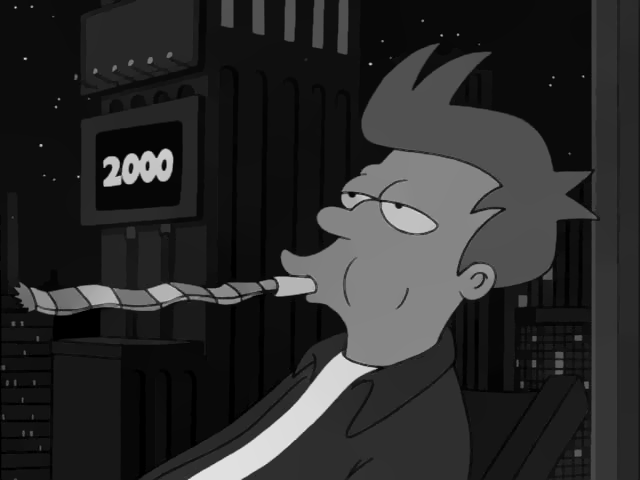} 
        \\
        & & 0.8827 & \textbf{0.9160}
        \\
        \includegraphics[width=0.2\textwidth]{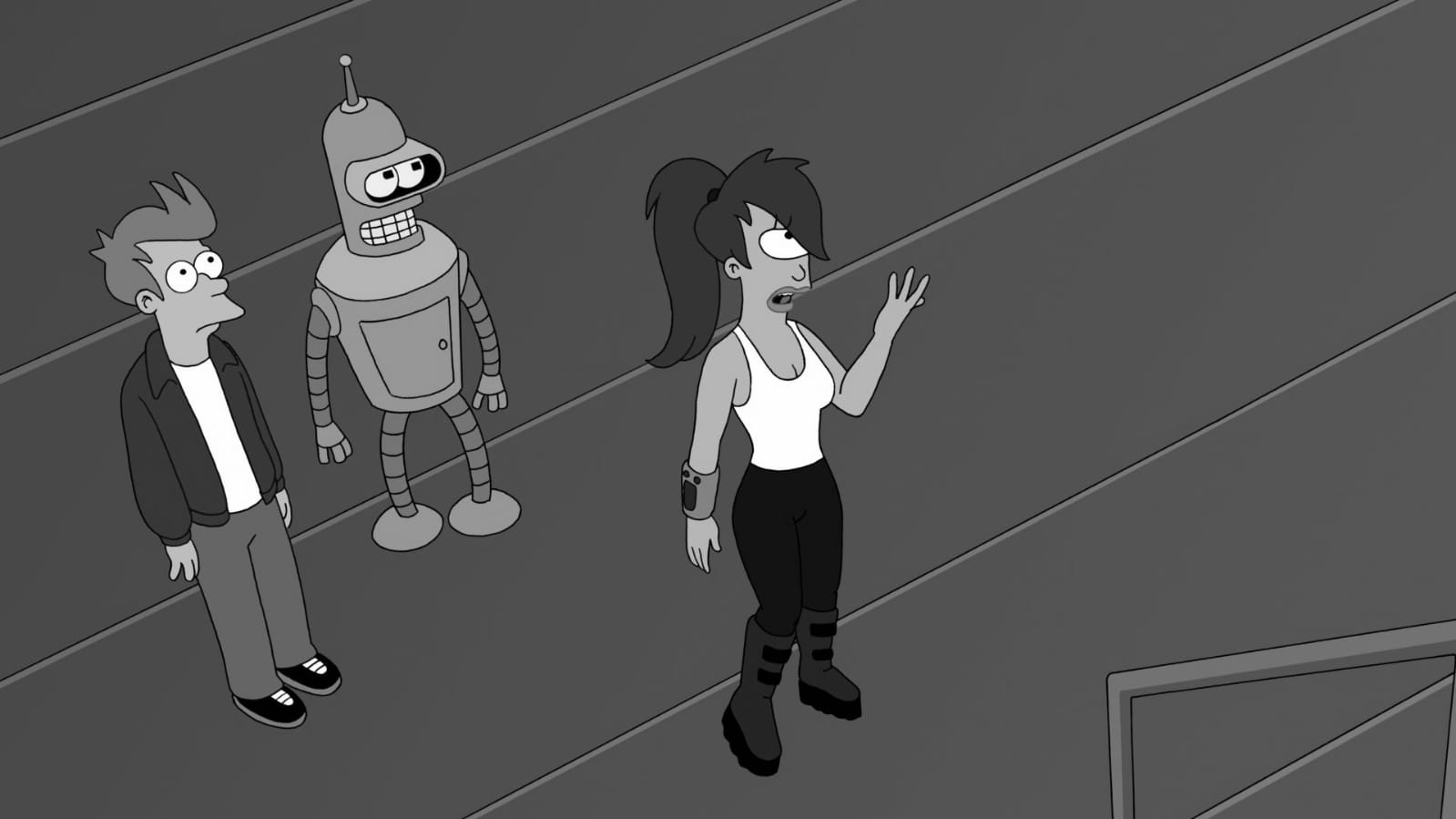} &
        \includegraphics[width=0.2\textwidth]{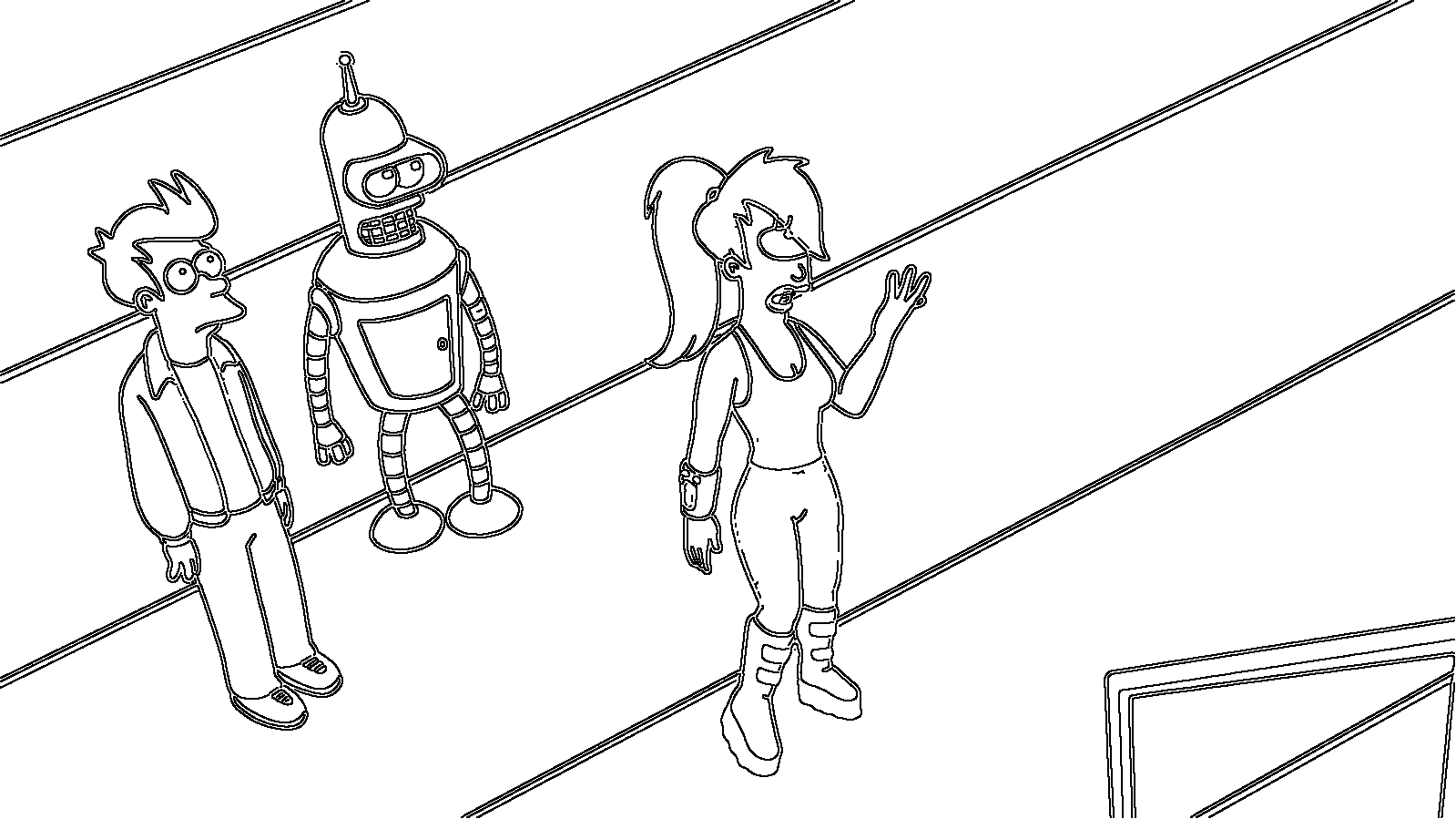} &
        \includegraphics[width=0.2\textwidth]{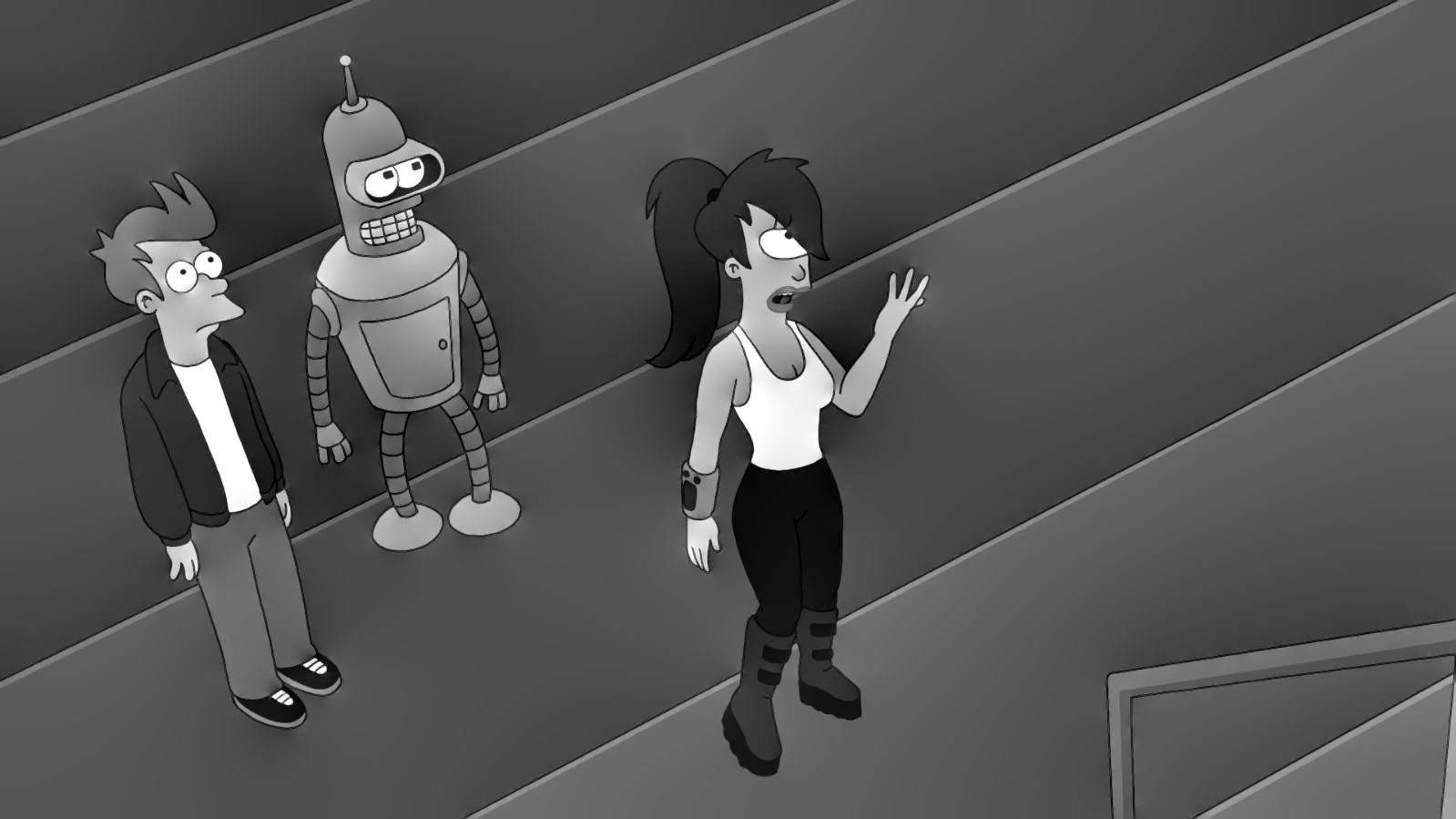} &
        \includegraphics[width=0.2\textwidth]{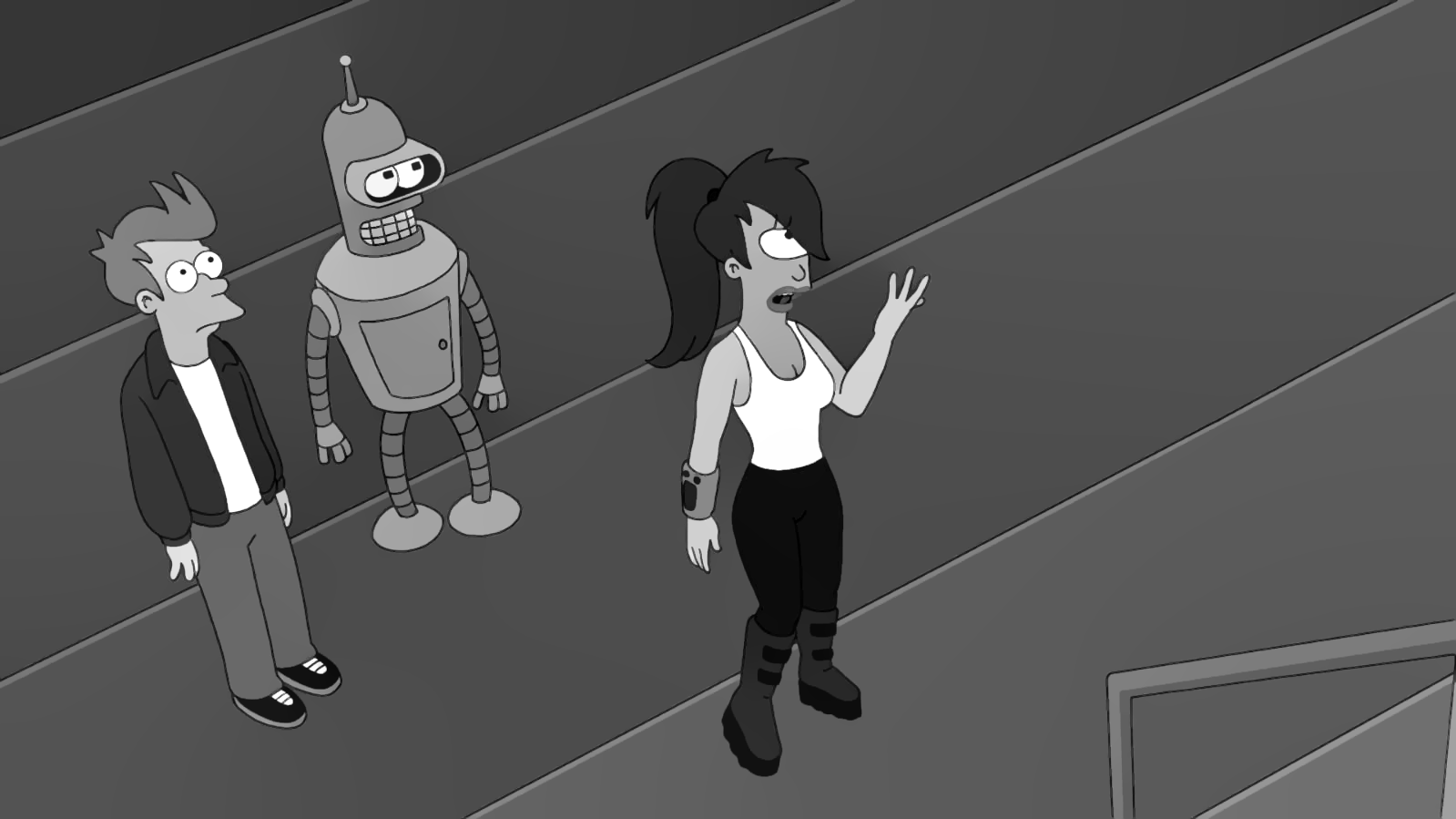} 
        \\
        & & 0.9671 & \textbf{0.9814}
        \\
        \hline
        Input $v$ & Edge mask & Linear \cite{vogel-ssvm} & Non-linear 
    \end{tabular}
    \caption{Compact data representation results with relative SSIM values. Edge mask has been computed by standard Canny segmentation.}
    \label{fig:cdr}
\end{figure}

Fig. \ref{fig:cdr} compares the results obtained via linear and non-linear osmosis. In general, the reconstructed images seem to have a lower contrast than the original ones, but the non linear model favours a better approximation, as confirmed by the SSIM values, in particular on strongly piecewise-constant images, such as cartoons.

\section{Conclusions}\label{sec:concl}

We presented a non-linear extension of the linear osmosis model originally proposed in \cite{weickert-ssvm}. The non-linearity is encoded in the model through a  diffusivity term $g$ whose definition depends on the magnitude of the diffusion-transport term,
which allows the model to balance the intensity values inside and outside the region of interest and, at the same time, to preserve the underlying features on the boundary (i.e., edges), thus preventing smoothing artefacts.  

Forward-Euler discretisation is used to approximate the time evolution. As far as spatial discretisation is concerned,  
 suitable finite difference schemes are defined.
A semi-implicit scheme is used to deal with nonlinear term in a computationally efficient way.

The proposed model enjoys conservation properties such average grey value and non-negativity preservation, both in its continuous formulation and in its discrete form. Unconditional stability of the semi-implicit iterative scheme can be proved using standard scale-space properties. A fast convergence to the steady-state solution can thus be obtained.

The efficiency and the
numerical accuracy of the presented schemes are validated
through three imaging applications: shadow and spot-light removal, and compact data representation.
The results are accurate, can be computed efficiently and outperform the ones obtained by alternative osmosis models and state-of-the-art approaches both visually and in terms of SSIM values. The simplicity of the model guarantees a low computational cost and, furthermore, a fast computation. Moreover, and most importantly, it only requires a preliminary segmentation of the region of interest (shadow, light-spot, edge mask) without any other tuning of model hyperparameter, which makes it, essentially, fully automatic.

An interesting future research direction is the study of a combined non-linear and anisotropic osmosis model as well as the possible extension to higher-order differential order so as to reduce the observed reconstruction drawbacks.

\section*{Acknowledgements}

LC acknowledges the support received by the I3S laboratory of Sophia-Antipolis, France, the CNRS PRIME project Imag'In and and by the EU H2020 RISE projects NoMADS, GA 777826. Research by SM was supported in part by the National Group for
Scientific Computation (GNCS-INDAM), Research Projects 2021.
SP acknowledges the support from the Leverhulme
Trust Research Project Grant (RPG-2018-121) ``Unveiling
the Invisible - Mathematics for Conservation in
Arts and Humanities''.
GR acknowledges the support received by University of Bologna and the I3S laboratory of Sophia-Antipolis, France.

\bibliographystyle{elsarticle-num}
\bibliography{main}

\begin{thebibliography}{10}
\expandafter\ifx\csname url\endcsname\relax
  \def\url#1{\texttt{#1}}\fi
\expandafter\ifx\csname urlprefix\endcsname\relax\def\urlprefix{URL }\fi
\expandafter\ifx\csname href\endcsname\relax
  \def\href#1#2{#2} \def\path#1{#1}\fi

\bibitem{weickert98}
J.~Weickert, {Anisotropic Diffusion in Image Processing}, B.G. Teubner,
  Stuttgart, 1998.

\bibitem{aubert2006mathematical}
G.~Aubert, P.~Kornprobst, Mathematical problems in image processing: {P}artial
  {D}ifferential {E}quations and the {C}alculus of {V}ariations, Vol. 147,
  Springer-Verlag, New York, 2006.
\newblock \href {https://doi.org/10.1007/978-0-387-44588-5}
  {\path{doi:10.1007/978-0-387-44588-5}}.

\bibitem{ChanShen}
T.~Chan, J.~Shen, Image Processing and Analysis: Variational, PDE, Wavelet, and
  Stochastic Methods, Society for Industrial and Applied Mathematics,
  Philadelphia, 2005.
\newblock \href {https://doi.org/10.1137/1.9780898717877}
  {\path{doi:10.1137/1.9780898717877}}.

\bibitem{schoenlieb}
C.-B. Schönlieb, Partial Differential Equation Methods for Image Inpainting,
  Cambridge Monographs on Applied and Computational Mathematics, Cambridge
  University Press, 2015.
\newblock \href {https://doi.org/10.1017/CBO9780511734304}
  {\path{doi:10.1017/CBO9780511734304}}.

\bibitem{Hagenburg12o}
K.~Hagenburg, M.~Breuß, J.~Weickert, O.~Vogel, Novel schemes for hyperbolic
  pdes using osmosis filters from visual computing, 2011, pp. 532--543.
\newblock \href {https://doi.org/10.1007/978-3-642-24785-9_45}
  {\path{doi:10.1007/978-3-642-24785-9_45}}.

\bibitem{weickert-ssvm}
J.~Weickert, K.~Hagenburg, M.~Breu{\ss}, O.~Vogel, Linear osmosis models for
  visual computing, in: A.~Heyden, F.~Kahl, C.~Olsson, M.~Oskarsson, X.-C. Tai
  (Eds.), Energy Minimization Methods in Computer Vision and Pattern
  Recognition, Springer Berlin Heidelberg, Berlin, Heidelberg, 2013, pp.
  26--39.

\bibitem{Parisotto2019AnisotropicOF}
S.~Parisotto, L.~Calatroni, M.~Caliari, C.-B. Schönlieb, J.~Weickert,
  Anisotropic osmosis filtering for shadow removal in images, Inverse Problems
  35~(5) (2019) 054001.
\newblock \href {https://doi.org/10.1088/1361-6420/ab08d2}
  {\path{doi:10.1088/1361-6420/ab08d2}}.

\bibitem{vogel-ssvm}
O.~Vogel, K.~Hagenburg, J.~Weickert, S.~Setzer, A fully discrete theory for
  linear osmosis filtering, in: A.~Kuijper, K.~Bredies, T.~Pock, H.~Bischof
  (Eds.), Scale Space and Variational Methods in Computer Vision, Springer
  Berlin Heidelberg, Berlin, Heidelberg, 2013, pp. 368--379.
\newblock \href {https://doi.org/10.1007/978-3-642-38267-3_31}
  {\path{doi:10.1007/978-3-642-38267-3_31}}.

\bibitem{Calatroni_ADI}
L.~Calatroni, C.~Estatico, N.~Garibaldi, S.~Parisotto, Alternating direction
  implicit (adi) schemes for a pde-based image osmosis model, Journal of
  Physics: Conference Series 904 (2017) 012014.
\newblock \href {https://doi.org/10.1088/1742-6596/904/1/012014}
  {\path{doi:10.1088/1742-6596/904/1/012014}}.

\bibitem{Parisotto2020}
S.~Parisotto, L.~Calatroni, A.~Bugeau, N.~Papadakis, C.-B. Schönlieb,
  Variational osmosis for non-linear image fusion, IEEE Transactions on Image
  Processing 29 (2020) 5507--5516.
\newblock \href {https://doi.org/10.1109/TIP.2020.2983537}
  {\path{doi:10.1109/TIP.2020.2983537}}.

\bibitem{Benalia}
S.~Benalia, M.~Hachama, A nonlocal method for image shadow removal, Computers
  \& Mathematics with Applications 107 (2022) 95--103.
\newblock \href {https://doi.org/10.1016/j.camwa.2021.12.023}
  {\path{doi:10.1016/j.camwa.2021.12.023}}.

\bibitem{Poisson}
P.~P\'{e}rez, M.~Gangnet, A.~Blake, Poisson image editing, ACM Trans. Graph.
  22~(3) (2003) 313–318.
\newblock \href {https://doi.org/10.1145/882262.882269}
  {\path{doi:10.1145/882262.882269}}.

\bibitem{GraDomain}
R.~Fattal, D.~Lischinski, M.~Werman, Gradient domain high dynamic range
  compression, ACM Trans. Graph. 21~(3) (2002) 249–256.
\newblock \href {https://doi.org/10.1145/566654.566573}
  {\path{doi:10.1145/566654.566573}}.

\bibitem{Finlayson1}
G.~Finlayson, S.~Hordley, C.~Lu, M.~Drew, On the removal of shadows from
  images, IEEE Transactions on Pattern Analysis and Machine Intelligence 28~(1)
  (2006) 59--68.
\newblock \href {https://doi.org/10.1109/TPAMI.2006.18}
  {\path{doi:10.1109/TPAMI.2006.18}}.

\bibitem{Finlayson2}
G.~Finlayson, M.~Drew, C.~Lu, Entropy minimization for shadow removal,
  International Journal of Computer Vision 85~(1) (2009) 35--57.
\newblock \href {https://doi.org/10.1007/s11263-009-0243-z}
  {\path{doi:10.1007/s11263-009-0243-z}}.

\bibitem{ShadowReview2013}
X.~Chunxia, S.~Ruiyun, X.~Donglin, M.~Kwanâ-Liu, Fast shadow removal using
  adaptive multi-scale illumination transfer, Computer Graphics Forum 32~(8)
  (2013) 207--218.
\newblock \href {https://doi.org/10.1111/cgf.12198}
  {\path{doi:10.1111/cgf.12198}}.

\bibitem{Guo}
R.~Guo, Q.~Dai, D.~Hoiem, Paired regions for shadow detection and removal, IEEE
  Transactions on Pattern Analysis and Machine Intelligence 35~(12) (2013)
  2956--2967.
\newblock \href {https://doi.org/10.1109/TPAMI.2012.214}
  {\path{doi:10.1109/TPAMI.2012.214}}.

\bibitem{Le}
H.~M. Le, D.~Samaras, Shadow removal via shadow image decomposition, CoRR
  abs/1908.08628 (2019).
\newblock \href {http://arxiv.org/abs/1908.08628} {\path{arXiv:1908.08628}}.

\bibitem{BabaSiggraph2003}
M.~Baba, N.~Asada, Shadow removal from a real picture, in: ACM SIGGRAPH 2003
  Sketches \& Applications, SIGGRAPH '03, ACM, New York, NY, USA, 2003, pp.
  1--1.
\newblock \href {https://doi.org/10.1145/965400.965488}
  {\path{doi:10.1145/965400.965488}}.

\bibitem{Carlsson1988SketchBC}
S.~Carlsson, Sketch based coding of grey level images, Signal Process. 15~(1)
  (1988) 57–83.
\newblock \href {https://doi.org/10.1016/0165-1684(88)90028-X}
  {\path{doi:10.1016/0165-1684(88)90028-X}}.

\bibitem{Mainberger2011EdgebasedCO}
M.~Mainberger, A.~Bruhn, J.~Weickert, S.~Forchhammer, Edge-based compression of
  cartoon-like images with homogeneous diffusion, Pattern Recognition 44 (2011)
  1859--1873.
\newblock \href {https://doi.org/10.1016/j.patcog.2010.08.004}
  {\path{doi:10.1016/j.patcog.2010.08.004}}.

\bibitem{p-laplace}
G.~Baravdish, Y.~Cheng, O.~Svensson, F.~{\AA}str{\"o}m, Extension of p-laplace
  operator for image denoising, in: L.~Bociu, J.-A. D{\'e}sid{\'e}ri, A.~Habbal
  (Eds.), System Modeling and Optimization, Springer International Publishing,
  Cham, 2016, pp. 107--116.
\newblock \href {https://doi.org/10.1007/978-3-319-55795-3_9}
  {\path{doi:10.1007/978-3-319-55795-3_9}}.

\bibitem{horn_johnson_1985}
R.~A. Horn, C.~R. Johnson, Matrix Analysis, Cambridge University Press, 1985.
\newblock \href {https://doi.org/10.1017/CBO9780511810817}
  {\path{doi:10.1017/CBO9780511810817}}.

\bibitem{BERMAN19791}
A.~Berman, R.~J. Plemmons, Nonnegative Matrices in the Mathematical Sciences,
  Society for Industrial and Applied Mathematics, 1994.
\newblock \href {https://doi.org/10.1137/1.9781611971262}
  {\path{doi:10.1137/1.9781611971262}}.

\bibitem{HKLM21}
M.~Huska, S.~H. Kang, A.~Lanza, S.~Morigi, A variational approach to additive
  image decomposition into structure, harmonic, and oscillatory components,
  SIAM Journal on Imaging Sciences 14~(4) (2021) 1749--1789.
\newblock \href {https://doi.org/10.1137/20M1355987}
  {\path{doi:10.1137/20M1355987}}.

\bibitem{Liu}
F.~Liu, M.~Gleicher, Texture-consistent shadow removal, in: D.~Forsyth,
  P.~Torr, A.~Zisserman (Eds.), Computer Vision -- ECCV 2008, Springer Berlin
  Heidelberg, Berlin, Heidelberg, 2008, pp. 437--450.
\newblock \href {https://doi.org/10.1007/978-3-540-88693-8_32}
  {\path{doi:10.1007/978-3-540-88693-8_32}}.

\bibitem{Shor}
Y.~Shor, D.~Lischinski, The shadow meets the mask: Pyramid‐based shadow
  removal, Comput. Graph. Forum 27 (2008) 577--586.
\newblock \href {https://doi.org/10.1111/j.1467-8659.2008.01155.x}
  {\path{doi:10.1111/j.1467-8659.2008.01155.x}}.

\bibitem{Finlayson3}
C.~Fredembach, G.~Finlayson, Simple shadow removal, in: 18th International
  Conference on Pattern Recognition (ICPR'06), Vol.~1, 2006, pp. 832--835.
\newblock \href {https://doi.org/10.1109/ICPR.2006.1054}
  {\path{doi:10.1109/ICPR.2006.1054}}.

\bibitem{Arbel}
E.~Arbel, H.~Hel-Or, Texture-preserving shadow removal in color images
  containing curved surfaces, in: 2007 IEEE Conference on Computer Vision and
  Pattern Recognition, 2007, pp. 1--8.
\newblock \href {https://doi.org/10.1109/CVPR.2007.383081}
  {\path{doi:10.1109/CVPR.2007.383081}}.

\bibitem{Wu1}
T.-P. Wu, C.-K. Tang, A bayesian approach for shadow extraction from a single
  image, in: Tenth IEEE International Conference on Computer Vision (ICCV'05)
  Volume 1, Vol.~1, 2005, pp. 480--487 Vol. 1.
\newblock \href {https://doi.org/10.1109/ICCV.2005.4}
  {\path{doi:10.1109/ICCV.2005.4}}.

\bibitem{Wu2}
T.-P. Wu, C.-K. Tang, M.~S. Brown, H.-Y. Shum, Natural shadow matting, ACM
  Trans. Graph. 26~(2) (2007) 8–es.
\newblock \href {https://doi.org/10.1145/1243980.1243982}
  {\path{doi:10.1145/1243980.1243982}}.

\end{thebibliography}

\end{document}